\documentclass{tac}
\usepackage{color}
\usepackage[usenames,dvipsnames,svgnames,table]{xcolor}
\usepackage[arrow, curve, frame, matrix, graph, tips]{xy}
\usepackage{amssymb}
\usepackage{amsmath}
\usepackage[colorlinks=true,urlcolor=blue,linkcolor=blue,citecolor=Maroon]{hyperref}
\usepackage{tabularx}

\author{Mark Weber}
\thanks{}
\address{Department of Mathematics, Macquarie University}

\title{Operads as polynomial 2-monads}
\copyrightyear{2015}
\keywords{operads; polynomial functors}
\amsclass{18D20; 18D50; 55P48}
\eaddress{mark.weber.math@gmail.com}

\newtheorem{thm}{\bf Theorem}
\newtheorem{prop}{\bf Proposition}
\newtheorem{lem}{\bf Lemma}
\newtheorem{cor}{\bf Corollary}

\newtheoremrm{defn}{\bf Definition}
\newtheoremrm{rem}{\bf Remark}
\newtheoremrm{rems}{\bf Remarks}
\newtheoremrm{exam}{\bf Example}
\newtheoremrm{exams}{\bf Examples}
\newtheoremrm{notn}{\bf Notation}
\newtheoremrm{org}{\bf Organisational Scheme}
\newtheoremrm{constn}{\bf Construction}

\newcommand{\tn}[1]{\textnormal{#1}}
\newcommand{\tnb}[1]{\textnormal{\bf #1}}

\newcommand{\tensor}{\otimes}

\newcommand{\N}{\mathbb{N}}
\newcommand{\comp}{\circ}
\newcommand{\id}{\tn{id}}
\newcommand{\ca}[1]{\mathcal{#1}}
\newcommand{\ladj}{\dashv}
\newcommand{\iso}{\cong}

\newcommand{\Set}{\tnb{Set}}

\newcommand{\Cat}{\tnb{Cat}}

\newcommand{\CAT}{\tnb{CAT}}
\renewcommand{\implies}{\Rightarrow}

\newcommand{\op}{\tn{op}}

\newcommand{\TwoCAT}{\mathbf{2}{\tnb{-CAT}}}
\newcommand{\TwoCat}{\mathbf{2}{\tnb{-Cat}}}

\newcommand{\Span}[1]{{\tnb{Span}_{#1}}}

\newcommand{\Polyc}[1]{{\tnb{Poly}_{#1}}}
\newcommand{\PFun}[1]{{\tnb{P}_{#1}}}

\newcommand{\LaxAlg}[1]{\tn{Lax-}{#1}\tn{-Alg}}
\newcommand{\PsAlg}[1]{\tn{Ps-}{#1}\tn{-Alg}}
\newcommand{\PsAlgs}[1]{\tn{Ps-}{#1}\tn{-Alg}_{\tn{s}}}
\newcommand{\PsAlgl}[1]{\tn{Ps-}{#1}\tn{-Alg}_{\tn{l}}}
\newcommand{\Algl}[1]{{#1}\tn{-Alg}_{\tn{l}}}

\newcommand{\Alg}[1]{{#1}\tn{-Alg}}
\newcommand{\Algs}[1]{{#1}\tn{-Alg}_{\tn{s}}}

\newcommand{\Opd}{\tnb{Opd}}
\newcommand{\Coll}{\tnb{Coll}}
\newcommand{\PolyMnd}[1]{\tnb{PolyMnd}_{#1}}

\newcommand{\PolyEnd}[1]{\tnb{PolyEnd}_{#1}}
\renewcommand{\P}{\mathbb{P}}

\newcommand{\sm}{\tnb{sm}}

\newcommand{\CatAsOp}{\underline{\Cat}}
\newcommand{\Com}{\mathsf{Com}}
\newcommand{\Ass}{\mathsf{Ass}}
\newcommand{\SMCMnd}{\tnb{S}}

\newcommand{\MCMnd}{\tnb{M}}

\begin{document}

\maketitle

\begin{abstract}
In this article we give a construction of a polynomial 2-monad from an operad and describe the algebras of the 2-monads which then arise. This construction is different from the standard construction of a monad from an operad in that the algebras of our associated 2-monad are the categorified algebras of the original operad. Moreover it enables us to characterise operads as categorical polynomial monads in a canonical way. This point of view reveals categorical polynomial monads as a unifying environment for operads, Cat-operads and clubs. We recover the standard construction of a monad from an operad in a 2-categorical way from our associated 2-monad as a coidentifier of 2-monads, and understand the algebras of both as weak morphisms of operads into a Cat-operad of categories. Algebras of operads within general symmetric monoidal categories arise from our new associated 2-monad in a canonical way. When the operad is sigma-free, we establish a Quillen equivalence, with respect to the model structures on algebras of 2-monads found by Lack, between the strict algebras of our associated 2-monad, and those of the standard one.
\end{abstract}

%\tableofcontents

\section{Introduction}
\label{sec:intro}

In contemporary mathematics there has been a proliferation of operadic notions \cite{LodayVallette-AlgebraicOperads}. These include cyclic operads, modular operads, dioperads, properads and so on, with the basic combinatorics underpinning these notions being more involved than that of the standard operads that arose originally in algebraic topology in the 1970's. One of the many contributions of Batanin and Berger in \cite{BataninBerger-HtyThyOfAlgOfPolyMnd}, is to exhibit these contemporary operadic notions as algebras of very particular standard operads. In this article we use unadorned name ``operad'' for what are commonly referred to as ``coloured symmetric operads of sets'' or also as ``symmetric multicategories''. More precisely, Batanin and Berger exhibited many of these contemporary operadic notions as algebras of $\Sigma$-free operads, an operad being \emph{$\Sigma$-free} when its symmetric group actions admit no fixed points.

Given an operad $T$ with set of colours $I$, and a symmetric monoidal category $\ca V$, one can consider the algebras of $T$ in $\ca V$. For nice enough $\ca V$ one has an associated monad whose algebras are coincide with those of the operad. When $\ca V = \Set$ this is a monad on $\Set/I$ which we denote as $T/\Sigma$. When $T$ is $\Sigma$-free, $T/\Sigma$ is a polynomial monad in the sense of \cite{GambinoKock-PolynomialFunctors}. In fact as explained by Kock \cite{KockJ-PolyFunTrees} and Szawiel-Zawadowski \cite{SzawielZawadowski-TheoriesOfAnalyticMonads}, finitary polynomial monads may be identified with $\Sigma$-free operads. In general, polynomial monads are examples of cartesian monads, so may be applied to internal categories, and in this way one may regard $T/\Sigma$ as a 2-monad on $\Cat/I$. Thus any contemporary operadic notion determines a 2-monad, and so the rich theory of 2-dimensional monad theory \cite{BWellKellyPower-2DMndThy} becomes applicable in these contexts. This is part of the technology which underpins \cite{BataninBerger-HtyThyOfAlgOfPolyMnd}, and which is developed further in this paper and \cite{Weber-AlgKan, Weber-CodescCrIntCat}.

In this article we give a new and different construction of a polynomial 2-monad from an operad and describe the algebras of the 2-monads which then arise. Our construction does not require $\Sigma$-freeness, and in the case of a $\Sigma$-free operad $T$ with set of colours $I$, produces a 2-monad on $\Cat/I$ which is different from $T/\Sigma$. In the general case we give an alternative 2-categorical construction of $T/\Sigma$ from $T$, and then establish that this construction restricts to the world of polynomial 2-monads in the $\Sigma$-free case.

This new 2-monad on $\Cat/I$ associated to an operad $T$ is also denoted as $T$. We find this convention to be most convenient, but when using it one must be aware that the conventional $\Cat$-valued algebras of the operad $T$ are the strict algebras of the 2-monad $T/\Sigma$, whereas the algebras of the 2-monad $T$ correspond to categorified algebras of the operad $T$. For example when $T$ is the terminal operad $\Com$, a strict algebra for the 2-monad $T$ is a symmetric strict monoidal category, whereas a strict algebra for $T/\Sigma$ is a commutative monoid in $\Cat$. 

From Corollary \ref{cor:general-operad-algebra}, the algebras of the operad within any symmetric monoidal category admit a canonical description in terms of this new 2-monad, which in \cite{Weber-CodescCrIntCat} enables the construction of the associated PROP in terms of general notions of 2-dimensional monad theory. This is then exploited in \cite{Weber-AlgKan} to systematise various related free constructions, whose combinatorial aspects might in some cases be quite involved, in terms of the universal properties that the associated PROP's enjoy by virtue of the developments of \cite{Weber-CodescCrIntCat}. With the more fundamental role of our associated 2-monad thus established, we feel justified in giving it the same name as the original operad.

Our new associated 2-monad arises naturally from a new characterisation of operads and related notions in terms of polynomials \cite{GambinoKock-PolynomialFunctors, Weber-PolynomialFunctors} in $\Cat$. In a sense, the spirit is quite similar to \cite{KockJ-DataTypesSymmetries} in that the symmetries of the operad are encoded directly in the polynomial, but the formal setting is quite different. Recall that a \emph{polynomial monad} in $\Cat$ is a monad in a certain bicategory $\Polyc{\Cat}$ whose objects are categories, the underlying endomorphism $I \to I$ of which is a diagram as on the left
\[ \xygraph{{\xybox{\xygraph{{I}="p0" [r] {E}="p1" [r] {B}="p2" [r] {I}="p3" "p0":@{<-}"p1"^-{s}:"p2"^-{p}:"p3"^-{t}}}}
[r(5)]
{\xybox{\xygraph{{1}="p0" [r] {\mathbb{P}_*}="p1" [r] {\mathbb{P}}="p2" [r] {1}="p3" "p0":@{<-}"p1"^-{}:"p2"^-{U^{\mathbb{P}}}:"p3"^-{t}}}}} \]
in which $p$ is an exponentiable functor. In particular one has the polynomial monad $\SMCMnd$ indicated on the right in which $\mathbb{P}$ is a skeleton of the category of finite sets and bijections, and a \emph{morphism} from the former to the latter consists of the functors $e$ and $b$ fitting into a commutative diagram
\begin{equation}\label{eq:for-organisational-scheme}
\begin{gathered}
\xygraph{!{0;(1.5,0):(0,.6667)::} {I}="p0" [r] {E}="p1" [r] {B}="p2" [r] {I}="p3" [d] {1}="p4" [l] {\mathbb{P}}="p5" [l] {\mathbb{P_*}}="p6" [l] {1}="p7" "p0":@{<-}"p1"^-{s}:"p2"^-{p}:"p3"^-{t}:"p4"^-{}:@{<-}"p5"^-{}:@{<-}"p6"^-{U^{\mathbb{P}}}:"p7"^-{}:@{<-}"p0"^-{} "p1":"p6"_-{e} "p2":"p5"^-{b} "p1":@{}"p5"|-{\tn{pb}}}
\end{gathered}
\end{equation}
and are compatible with the monad structures. By Theorem \ref{thm:SMultiCatAsPolyMndCat}, Remarks \ref{rem:clubs} and \ref{rem:Cat-operads}, and Proposition \ref{prop:Sigma-freeness-poly} one may identify
\begin{itemize}
\item Operads as situations (\ref{eq:for-organisational-scheme}) in which $I$ is discrete and $b$ is a discrete fibration.
\item $\Sigma$-free operads as situations (\ref{eq:for-organisational-scheme}) in which $I$ is discrete, $b$ is a discrete fibration and $B$ is equivalent to a discrete category.
\item $\Cat$-operads as situations (\ref{eq:for-organisational-scheme}) in which $I$ is discrete and $b$ has the structure of a cloven split fibration.
\item Clubs in the sense of Kelly \cite{Kelly-ClubsDoctrines, Kelly-ClubsDataTypeConstructors} as situations (\ref{eq:for-organisational-scheme}) in which $I = 1$.
\end{itemize}
Our construction of a 2-monad from an operad regards an operad $T$ in this way in which $I$ is the set of colours, and then the usual construction \cite{GambinoKock-PolynomialFunctors, Weber-PolynomialFunctors} of a polynomial functor from a polynomial gives a 2-monad on $\Cat/I$. The basic theory of polynomials and polynomial functors from \cite{GambinoKock-PolynomialFunctors, Weber-PolynomialFunctors} is recalled in Section \ref{sec:poly}.

Like any 2-monad, $T$ has different types of algebra (lax, pseudo and strict), different types of algebra morphism (lax, colax, pseudo and strict) and thus a variety of different 2-categories of algebras depending on which types of algebras and algebra morphisms one is interested in. By Theorem \ref{thm:categorical-algebras-of-operads} $T$-algebras admit an explicit description as the appopriately weak morphisms of operads $T \to \CatAsOp$, where $\CatAsOp$ is a canonical $\Cat$-enriched operad whose objects are categories. That is, the lax, pseudo and strict algebras of the 2-monad $T$ are lax, pseudo and strict morphisms of operads $T \to \CatAsOp$ in the sense of Definitions \ref{defn:lax-morphism-into-CatAsOp} and \ref{defn:lax-morphism-into-CatAsOp-variants}. In the case where $T$ is a category, that is when all its operations are of arity $1$, this description of the algebras of $T$ is well-known and goes back at least to \cite{BWellKellyPower-2DMndThy}. In this case the 2-monad $T$ is the 2-monad on $\Cat/I$ whose 2-category of strict algebras and strict morphisms is the functor 2-category $[T,\Cat]$, and a lax or pseudo algebra is exactly a lax or pseudo functor $T \to \Cat$.

Similarly one has characterisations of the various types of $T$-algebra morphism in Theorem \ref{thm:catalg-morphisms-for-operads} and $T$-algebra 2-cells in Theorem \ref{thm:catalg-2cells-for-operads}. In particular for any operad $T$ and any symmetric monoidal category $\ca V$, algebras of $T$ in $\ca V$ can be seen as lax $T$-algebra morphisms in a canonical way. As explained at the end of Section \ref{sec:Algebras} the via the central examples of \cite{BataninBerger-HtyThyOfAlgOfPolyMnd}, one exhibits the categories cyclic operads, modular operads, dioperads, properads and so on, in a symmetric monoidal category, in this way.

To understand the relationship between the 2-monads $T$ and $T/\Sigma$ on $\Cat/I$ for a given operad $T$ with set of colours $I$, one begins by thinking about the algebras. One feature of the notion of lax, pseudo or strict morphism $H : T \to \CatAsOp$ alluded to above is that $H$ is not equivariant in the strictest sense, but rather that it is so up to coherent isomorphisms which are called the \emph{symmetries} of $H$. When these symmetries are identities, the lax morphism is said to be \emph{commutative}, and it is the commutative strict morphisms $T \to \CatAsOp$ which correspond to the algebras of $T/\Sigma$. In other words strict $T/\Sigma$-algebras are included amongst strict $T$-algebras, and this inclusion is exactly the inclusion of the commutative strict morphisms $T \to \CatAsOp$ amongst the general strict morphisms.

The standard construction of the monad $T/\Sigma$ is via a formula which involves quotienting out by the symmetric group actions of $T$. In this article this quotienting is carried in a 2-categorical way. Starting from the 2-monad $T$, Definition \ref{defn:T1-Sigma} provides a 2-cell of 2-monads $\alpha_T$ as in
\[ \xygraph{!{0;(1.5,0):(0,1)::} {T^{[1]}_{\Sigma}}="p0" [r] {T}="p1" [r] {T/\Sigma}="p2" "p0":@<1.5ex>"p1"|(.45){}="t"^-{d_T}:"p2"^-{q_T} "p0":@<-1.5ex>"p1"|(.45){}="b"_-{c_T} "t":@{}"b"|(.15){}="d"|(.85){}="c" "d":@{=>}"c"^-{\alpha_T}} \]
and then $q_T$ is the universal morphism of 2-monads which post-composes with $\alpha_T$ to give an identity. In the language of 2-category theory, $q_T$ is the \emph{reflexive coidentifier} of $\alpha_T$ in $\tn{Mnd}(\Cat/I)$. The algebras of $T/\Sigma$ defined in this way are identified with commutative operad morphisms $T \to \CatAsOp$ in Theorem \ref{thm:commutative-algebras-of-operads}, and so it follows immediately that our construction of $T/\Sigma$ coincides with the standard one. Moreover as explained in Section \ref{sec:sigma-free}, when $T$ is a $\Sigma$-free operad one can witness this quotienting process as taking place completely in the world of polynomials, and this is why $T/\Sigma$ is a polynomial monad when $T$ is $\Sigma$-free.

In \cite{Lack-HomotopyAspects2Monads} a Quillen model structure on the 2-category $\Algs T$ of strict $T$-algebras and strict morphisms was exhibited, for any 2-monad $T$ on a 2-category $\ca K$ with finite limits and finite colimits. For an operad $T$, the morphism $q_T$ of 2-monads described above determines an adjunction $C_T \ladj \overline{q}_T$ between $\Algs T$ and $\Algs {T/\Sigma}$. With respect to the model structures of \cite{Lack-HomotopyAspects2Monads}, $C_T \ladj \overline{q}_T$ is a Quillen adjunction. Our final result, Theorem \ref{thm:Sigma-free-Quillen-equivalence}, says that when $T$ is $\Sigma$-free, $C_T \ladj \overline{q}_T$ is a Quillen equivalence.

\emph{Notation, terminology and background}. We assume a basic familiarity with some of the elementary notions of 2-category theory -- basic 2-categorical limits such as cotensors, comma objects and isocomma objects; the calculus of mates as explained in \cite{KellyStreet-ElementsOf2Cats}; and the basic notions 2-dimensional monad theory which one can find for instance in \cite{BWellKellyPower-2DMndThy, Lack-Codescent}. This article is a sequel to \cite{Weber-PolynomialFunctors}, and so one can find an exposition of many background notions relevant here, such as fibrations and their definition internal to any 2-category, in \cite{Weber-PolynomialFunctors} in a way that is notationally and terminologically compatible with this article. However some effort has been made to recall important background as needed for the covenience of the reader. For instance one finds the definition of the various types of algebra of a 2-monad recalled in Section \ref{sec:Algebras} just before the details of these definitions are needed in this work. As in \cite{Weber-PolynomialFunctors} we denote by $[n]$ the ordinal $\{0 < ... < n\}$ regarded as a category. We denote by $\Cat$ the 2-category of small categories, and sometimes make use of a 2-category $\CAT$ of large categories, which include standard categories of interest, like $\Set$ as objects.

{\bf Acknowledgements}. The author would like to acknowledge Michael Batanin, Richard Garner, Steve Lack and Ross Street for interesting discussions on the subject of this paper. Thanks also go to the attentive anonymous referee, whose remarks helped improve the presentation of this article. The author would also like to acknowledge the financial support of the Australian Research Council grant No. DP130101172.

\section{Polynomial functors}
\label{sec:poly}

Composition with a functor $f : X \to Y$ defines the effect on objects of the 2-functor $\Sigma_f : \Cat/X \to \Cat/Y$, and its right adjoint denoted $\Delta_f$, is given on objects by pulling back along $f$. When $\Delta_f$ has a right adjoint, $f$ is said to be \emph{exponentiable} and this further right adjoint is denoted $\Pi_f$. The exponentiable functors are closed under composition, and stable by pullback along arbitrary functors. Moreover one has a combinatorial characterisation of exponentiable functors as Giraud-Conduch\'{e} fibrations \cite{Conduche-GCFibrations, Giraud-Descent, Street-Powerful}. In particular Grothendieck fibrations and Grothendieck opfibrations are Giraud-Conduch\'{e} fibrations.

In elementary terms, the effect of $\Pi_f$ on objects is to take distributivity pullbacks along $f$ in the sense to be recalled now from \cite{Weber-PolynomialFunctors}. Given $g : W \to X$, a \emph{pullback around} $(f,g)$ consists of $(p,q,r)$ as on the left in
\[ \xygraph{{\xybox{\xygraph{!{0;(1,0):(0,.75)::} {P}="p0" [r] {Q}="p1" [d(2)] {Y}="p2" [l] {X}="p3" [u] {W}="p4" "p0":"p1"^-{q}:"p2"^-{r}:@{<-}"p3"^-{f}:@{<-}"p4"^-{g}:@{<-}"p0"^-{p}:@{}"p2"|-{\tn{pb}}}}}
[r(3)]
{\xybox{\xygraph{!{0;(1,0):(0,.75)::} {P}="p0" [r] {Q}="p1" [d(2)] {Y}="p2" [l] {X}="p3" [u] {W}="p4" "p0":"p1"^-{q}:"p2"^-{r}:@{<-}"p3"^-{f}:@{<-}"p4"^-{g}:@{<-}"p0"^-{p}:@{}"p2"|-{\tn{dpb}}}}}
[r(3)]
{\xybox{\xygraph{!{0;(1,0):(0,.75)::} {P}="p0" [r] {Q}="p1" [d(2)] {Y}="p2" [l] {X}="p3" [u] {W}="p4" "p0":"p1"^-{}:"p2"^-{\Pi_f(g)}:@{<-}"p3"^-{f}:@{<-}"p4"^-{g}:@{<-}"p0"^-{\varepsilon^f_g}:@{}"p2"|-{\tn{dpb}} "p0":@/_{1.75pc}/"p3"_-{\Delta_f\Pi_f(g)}}}}} \]
such that the morphisms $(gp,f,r,q)$ form a pullback square as indicated. A morphism $(p,q,r) \to (p',q',r')$ consists of morphisms $(\alpha,\beta)$ such that $p = p'\alpha$, $\beta q = q' \alpha$ and $r = r'\beta$. A \emph{distributivity pullback} around $(f,g)$ is a terminal object in the category of pullbacks around $(f,g)$ just described. A general such is denoted as in the middle of the previous display. The connection with $\Pi_f$ is indicated on the right in the previous display, in which $\varepsilon^f$ is the counit of $\Delta_f \ladj \Pi_f$ and $\varepsilon^f_g$ is its component at $g$. Explicitly, the universal property of a distributivity pullback says that given $(p',q',r')$ as in
\[ \xygraph{{P}="p0" [r] {W}="p1" [r] {X}="p2" [d] {Y}="p3" [l(2)] {Q}="p4" "p0":"p1"^-{p}:"p2"^-{g}:"p3"^-{f}:@{<-}"p4"^-{r}:@{<-}"p0"^-{q}:@{}"p3"|-{\tn{dpb}}
"p0" [l] {A}="q0" "p4" [l] {B}="q1" "q0"(:@/^{1.5pc}/"p1"^-{p'},:"q1"_-{q'}:@/_{1.5pc}/"p3"_-{r'})
"q0":@{.>}"p0"^-{\alpha} "q1":@{.>}"p4"^-{\beta}} \]
making the square with boundary $(gp',f,r',q')$ a pullback, there exist $\alpha$ and $\beta$ as shown unique such that $p\alpha = p'$, $q\alpha = \beta q'$ and $r\beta = r'$.

In practise one is often interested in obtaining an explicit description of $Q$ and $r$ in terms of the generating data $(f,g)$ of a distributivity pullback. We do this now in the case where $f$ is a discrete opfibration. To this end recall that for a discrete opfibration $f : X \to Y$ one has the corresponding functor $\tilde{f} : Y \to \Set$ whose effect on objects is $y \mapsto f^{-1}\{y\}$, and whose lax colimit is $X$. The data of the lax colimit cocone consists of the inclusions of fibres $i_y : f^{-1}\{y\} \to X$ for all $x \in X$, and lax naturality 2-cells $i_{\alpha} : i_{y_1} \to i_{y_2}\tilde{f}(\alpha)$ for all $\alpha : y_1 \to y_2$ in $Y$. Another way to organise this information uses the $\tn{Fam}$ construction. For a category $C$, the category $\tn{Fam}(C)$ has as objects pairs $(I,h)$ where $I$ is a set regarded as a discrete category, and $h : I \to C$ is a functor. A morphism $(I,h) \to (J,k)$ consists of $(f,\phi)$ where $f : I \to J$ is a function and $\phi : h \to kf$ is a natural transformation. Then the fibres of $f$ and the above lax colimit cocone organise to form a functor{\footnotemark{\footnotetext{As explained in section 5 of \cite{Weber-Fam2fun}, $\overline{f}$ is the $\tn{Fam}$-generic factorisation of $\tilde{f} : Y \to \Set = \tn{Fam}(1)$.}}}
\[ \begin{array}{lcccr} {\overline{f} : Y \longrightarrow \tn{Fam}(X)} && {y \mapsto (f^{-1}\{y\},i_y)} && {\alpha:y_1 \to y_2 \, \mapsto \, (\tilde{f}(\alpha),i_{\alpha}).} \end{array} \]
\begin{lem}\label{lem:Cat-dpb-along-dopfib-fam}
In a distributivity pullback as on the left
\[ \xygraph{{\xybox{\xygraph{{P}="p0" [r] {W}="p1" [r] {X}="p2" [d] {Y}="p3" [l(2)] {Q}="p4" "p0":"p1"^-{}:"p2"^-{g}:"p3"^-{f}:@{<-}"p4"^-{r}:@{<-}"p0"^-{}:@{}"p3"|-{\tn{dpb}}}}}
[r(4)d(.1)]
{\xybox{\xygraph{!{0;(1.5,0):(0,.6667)::} {Q}="p0" [r] {\tn{Fam}(W)}="p1" [d] {\tn{Fam}(X)}="p2" [l] {Y}="p3" "p0":"p1"^-{}:"p2"^-{\tn{Fam}(g)}:@{<-}"p3"^-{\overline{f}}:@{<-}"p0"^-{r}:@{}"p2"|-{\tn{pb}}}}}} \]
in $\Cat$ in which $f$ is a discrete opfibration, $r$ and $Q$ can be described explicitly by the pullback on the right.
\end{lem}
\begin{proof}
Since discrete opfibrations are exponentiable functors the distributivity pullback exists, and one just needs to use the adjunction $\Delta_f \ladj \Pi_f$ to unpack the explicit description and match it up with the pullback on the right. An object $z$ of $Q$ over $y \in Y$ as on the left
\[ \xygraph{{\xybox{\xygraph{{[0]}="p0" [r(2)] {Q}="p1" [dl] {Y}="p2"
"p0":"p1"^-{z}:"p2"^-{r}:@{<-}"p0"^-{y} "p0" [d(.5)r] {\scriptsize{=}}}}}
[r(4)]
{\xybox{\xygraph{{f^{-1}\{y\}}="p0" [r(2)] {W}="p1" [dl] {X}="p2" "p0":"p1"^-{h}:"p2"^-{g}:@{<-}"p0"^-{i_y} "p0" [d(.5)r] {\scriptsize{=}}}}}} \]
amounts, by the adjunction $\Delta_f \ladj \Pi_f$, to the morphism $h$ on the right in the previous display where $i_y$ is the inclusion. Thus one can identify objects of $Q$ as pairs $(y,h)$, and then $r$ is given by $r(y,h) = y$. An arrow of $Q$ amounts to a functor $[1] \to Q$, and thus a choice of arrow $\alpha : y_1 \to y_2$ in $Y$ codified itself as a functor $\alpha : [1] \to Y$, together with $\beta : [1] \to Q$ such that $r\beta = \alpha$. Pulling back $\alpha$ along $f$ gives a category whose object set is the disjoint union $f^{-1}\{y_1\} \coprod f^{-1}\{y_2\}$, with one non-identity morphism for each element of $x \in f^{-1}\{y_1\}$. The domain of the morphism corresponding to $x$ is $x$ itself, and its codomain is $\tilde{f}(\alpha)(x)$. Thus by the adjunction $\Delta_f \ladj \Pi_f$, a morphism of $Q$ amounts to a morphism of the pullback.
\end{proof}
As explained in \cite{Weber-PolynomialFunctors} polynomials in $\Cat$ form a 2-bicategory $\Polyc{\Cat}$. Its objects are small categories, an arrow from $I$ to $J$ in $\Polyc {\Cat}$ is a \emph{polynomial} in $\Cat$ from $I$ to $J$, which by definition consists of functors as on the left
\[ \xygraph{{\xybox{\xygraph{{I}="p0" [r] {E}="p1" [r] {B}="p2" [r] {J}="p3" "p0":@{<-}"p1"^-{s}:"p2"^-{p}:"p3"^-{t}}}}
[r(5.5)d(.05)]
{\xybox{\xygraph{!{0;(1.5,0):(0,.5)::} {I}="p0" [ur] {E_1}="p1" [r] {B_1}="p2" [dr] {J}="p3" [dl] {B_2}="p4" [l] {E_2}="p5" "p0":@{<-}"p1"^-{s_1}:"p2"^-{p_1}:"p3"^-{t_1}:@{<-}"p4"^-{t_2}:@{<-}"p5"^-{p_2}:"p0"^-{s_2}
"p1":"p5"_-{f_2} "p2":"p4"^-{f_1}
"p1":@{}"p4"|-{\tn{pb}} "p0" [r(.5)] {\scriptstyle{=}} "p3" [l(.5)] {\scriptstyle{=}}}}}} \]
in which $p$ is exponentiable. A 2-cell $f:(s_1,p_1,t_1) \to (s_2,p_2,t_2)$ in $\Polyc {\ca E}$ is a diagram as on the right in the previous display, and we call the morphisms $f_1$ and $f_2$ the \emph{components} of $f$. The 2-cells of the hom $\Polyc{\Cat}(I,J)$ involve 2-cells between components make the resulting cones into $I$ and $J$, and the cylinder in the middle commutative.

In elementary terms the process of forming the horizontal composite $(s_3,p_3,t_3) = (s_2,p_2,t_2) \comp (s_1,p_1,t_1)$ of polynomials is encapsulated by the commutative diagram
\[ \xygraph{{I}="b1" [r] {E_1}="b2" [r] {B_1}="b3" [r] {J}="b4" [r] {E_2}="b5" [r] {B_2}="b6" [r] {K.}="b7" "b4" [u] {B_1 \times_J E_2}="p1" [u] {F}="dl" ([r(1.5)] {B_3}="dr", [l(1.5)] {E_3}="p2")
"b1":@{<-}"b2"_-{s_1}:"b3"_-{p_1}:"b4"_-{t_1}:@{<-}"b5"_-{s_2}:"b6"_-{p_2}:"b7"_-{t_2} "dl":"p1"_-{}(:"b3"_-{},:"b5"^-{}) "b2":@{<-}"p2"_-{}:"dl"_-{}:"dr"_-{}:"b6"^(.7){} "b1":@{<-}"p2"^-{s_3} "dr":"b7"^-{t_3} "p2":@/^{1pc}/"dr"^-{p_3}
"b3" [u(1.25)] {\scriptstyle{\tn{pb}}} "b5" [u(1.25)] {\scriptstyle{\tn{dpb}}} "b4" [u(.5)] {\scriptstyle{\tn{pb}}}} \]
and by Theorem 4.1.4 of \cite{Weber-PolynomialFunctors}, one has a homomorphism
\[ \begin{array}{lccr} {\PFun {\Cat} : \Polyc {\Cat} \longrightarrow \TwoCAT}
&&& {I \mapsto \Cat/I} \end{array} \]
of 2-bicategories with object map as indicated. The effect of $\PFun {\Cat}$ on arrows is to send the polynomial $(s,p,t)$ to the composite functor $\Sigma_t\Pi_p\Delta_s : \Cat/I \longrightarrow \Cat/J$. See \cite{Weber-PolynomialFunctors} for more details.
\begin{exam}\label{exam:polyfunctor-Cat-middle-dopfib}
As an illustration consider the case of a polynomial $(s,p,t)$ as above in $\Cat$, in which $p : E \to B$ is a discrete opfibration, and write $T : \Cat/I \to \Cat/J$ for the polynomial functor $\PFun{\Cat}(s,p,t)$. Then for $X \to I$ in $\Cat/I$, $TX$ is formed as on the left
\[ \xygraph{*=(0,2)!(0,1.75){\xybox{\xygraph{{\xybox{\xygraph{!{0;(1.5,0):(0,.6667)::} {I}="p1" [r] {E}="p2" [r] {B}="p3" [r] {J}="p4" "p1" [u] {X}="s" "p2" [u] {X \times_I E}="mdpb" [u] {T_{\bullet}X}="tldpb" [r] {TX}="trdpb"
"p1":@{<-}"p2"_-{s}:"p3"_-{p}:"p4"_-{t} "mdpb"(:"s"_-{}:"p1"_-{},:"p2"^-{},:@{<-}"tldpb"_-{}:"trdpb"^-{}(:"p3"|-{},:"p4"^-{}|-{}="codeq"))
"mdpb" ([r(.6)] {\scriptsize{\tn{dpb}}}, [l(.5)d(.5)] {\scriptsize{\tn{pb}}})}}}
[r(6)u(.4)]
{\xybox{\xygraph{!{0;(2.5,0):(0,.4)::} {TX}="p0" [r] {B}="p1" [d] {\tn{Fam}(E)}="p2" [d] {\tn{Fam}(I)}="p3" [l] {\tn{Fam}(X)}="p4" [u] {\tn{Fam}(X \times_I E)}="p5" "p0":"p1"^-{}:"p2"^-{\overline{p}}:"p3"^-{}:@{<-}"p4"^-{}:@{<-}"p5"^-{}:@{<-}"p0"^-{} "p5":"p2"^{} "p0":@{}"p2"|-{\tn{pb}}:@{}"p4"|-{\tn{pb}}}}}}}}} \]
and so in view of the fact that the $\tn{Fam}$ construction preserves pullbacks and Lemma \ref{lem:Cat-dpb-along-dopfib-fam}, one has pullbacks in $\CAT$ as on the right. Unpacking the composite pullback on the right, one finds that $TX$ has the following explicit description. An object is a pair $(b,h)$ where $b$ is an object of $B$ and $h : p^{-1}\{b\} \to X$ whose composite with $X \to I$ is $si_b$, where $i_b : p^{-1}\{b\} \to E$ is the inclusion of the fibre. A morphism $(b_1,h_1) \to (b_2,h_2)$ is a pair $(\beta,\gamma)$ where $\beta : b_1 \to b_2$ is in $B$ and $\gamma$ is a natural transformation as on the left
\[ \xygraph{{\xybox{\xygraph{{p^{-1}\{b_1\}}="p0" [r(2)] {p^{-1}\{b_2\}}="p1" [dl] {X}="p2" [d] {I}="p3" "p0":"p1"^-{\tilde{p}\beta}:"p2"^-{h_2}:@{<-}"p0"^-{h_1} "p2":"p3"^-{}
"p0" [d(.5)r(.85)] :@{=>}[r(.3)]^-{\gamma}}}}
[r(2)] {=} [r(2)]
{\xybox{\xygraph{{p^{-1}\{b_1\}}="p0" [r(2)] {p^{-1}\{b_2\}}="p1" [dl] {E}="p2" [d] {I}="p3" "p0":"p1"^-{\tilde{p}\beta}:"p2"^-{i_{b_2}}:@{<-}"p0"^-{i_{b_1}} "p2":"p3"^-{s}
"p0" [d(.5)r(.85)] :@{=>}[r(.3)]^-{i_{\beta}}}}}} \]
satisfying the equation, in which $\tilde{p}(\beta)$ and $i_{\beta}$ are as described prior to Lemma \ref{lem:Cat-dpb-along-dopfib-fam}. In the cases of interest in this article, $I$ is typically discrete, in which case this last equation is satisfied automatically.
\end{exam}
\begin{rem}\label{rem:spans}
A span in $\Cat$ as on the left
\[ \xygraph{{\xybox{\xygraph{{I}="p0" [r] {B}="p1" [r] {J}="p2" "p0":@{<-}"p1"^-{s}:"p2"^-{t}}}} [r(4)]
{\xybox{\xygraph{{I}="p0" [r] {B}="p1" [r] {B}="p2" [r] {J}="p3" "p0":@{<-}"p1"^-{s}:"p2"^-{1_B}:"p3"^-{t}}}}} \]
is identified with a polynomial as on the right, and the composition of polynomials generalises the usual pullback composition of spans. In particular a functor $f: I \to J$ determines the spans $f^{\bullet}$ and $f_{\bullet}$
\[ \xygraph{{\xybox{\xygraph{!{0;(.8,0):(0,1)::} {I}="p0" [r] {I}="p1" [r] {J}="p2" "p0":@{<-}"p1"^-{1_I}:"p2"^-{f}}}}
[r(2.5)]
{\xybox{\xygraph{!{0;(.8,0):(0,1)::} {J}="p0" [r] {I}="p1" [r] {I}="p2" "p0":@{<-}"p1"^-{f}:"p2"^-{1_I}}}}
[r(3.25)u(.275)]
{\xybox{\xygraph{!{0;(1,0):(0,.5)::} {J}="p0" [ur] {I}="p1" [r] {I}="p2" [dr] {J}="p3" [dl] {J}="p4" [l] {J}="p5" "p0":@{<-}"p1"^-{f}:"p2"^-{1_I}:"p3"^-{f}:@{<-}"p4"^-{1_J}:@{<-}"p5"^-{1_J}:"p0"^-{1_J}
"p1":"p5"_-{f} "p2":"p4"^-{f} "p1":@{}"p4"|-{\tn{pb}}}}}} \]
as on the left and middle respectively, and one has an adjunction $f^{\bullet} \ladj f_{\bullet}$ in $\Span {\Cat}$ and thus also in $\Polyc {\Cat}$. The counit $c_f$ of this adjunction in $\Polyc{\Cat}$ is given by the diagram on the right in the previous display.
\end{rem}
For a locally cartesian closed category $\ca E$, as described in \cite{GambinoKock-PolynomialFunctors}, the categories $\PolyEnd {\ca E}$ and $\PolyMnd {\ca E}$ of polynomial endomorphisms and polynomial monads respectively. We now adapt these definitions to the case $\ca E = \Cat$.
\begin{defn}\label{defn:PolyEnd}
An object of the category $\PolyEnd {\Cat}$ is a pair $(I,P)$ where $I$ is a small category and $P : I \to I$ is a polynomial. A morphism $(I,P) \to (J,Q)$ is a pair $(f,\phi)$ where $f : I \to J$ is a functor and $\phi : f^{\bullet} \comp P \comp f_{\bullet} \to Q$ is in $\Polyc {\Cat}(J,J)$. Given morphisms $(f,\phi) : (I,P) \to (J,Q)$ and $(g,\gamma) : (J,Q) \to (K,R)$, the underlying functor of the composite $(g,\gamma)(f,\phi)$ is $gf$, and the 2-cell datum of this composite is given by
\[ \xygraph{!{0;(1.5,0):(0,.6667)::} {g^{\bullet} \comp f^{\bullet} \comp P \comp f_{\bullet} \comp g_{\bullet}}="p0" [r(3)] {g^{\bullet} \comp Q \comp g_{\bullet}}="p1" [r(1.5)] {R.}="p2" "p0":"p1"^-{g^{\bullet} \comp \phi \comp g_{\bullet}}:"p2"^-{\gamma}} \]
\end{defn}
In more elementary terms, writing $P = (s,p,t)$ and $Q = (s',p',t')$, a morphism $(f,\phi) : (I,P) \to (J,Q)$ of $\PolyEnd{\Cat}$ amounts to a commutative diagram
\begin{equation}\label{diag:PolyEnd-morphism}
\begin{gathered}
\xygraph{!{0;(1.5,0):(0,.6667)::} {I}="p0" [r] {E}="p1" [r] {B}="p2" [r] {I}="p3" [d] {J}="p4" [l] {B'}="p5" [l] {E'}="p6" [l] {J}="p7" "p0":@{<-}"p1"^-{s}:"p2"^-{p}:"p3"^-{t}:"p4"^-{f}:@{<-}"p5"^-{t'}:@{<-}"p6"^-{p'}:"p7"^-{s'}:@{<-}"p0"^-{f} "p1":"p6"_-{f_2} "p2":"p5"^-{f_1} "p1":@{}"p5"|-{\tn{pb}}}
\end{gathered}
\end{equation}
because $f^{\bullet} \comp P \comp f_{\bullet} = (fs,p,ft)$, and the composition just described amounts to stacking such diagrams vertically. The various mates of $\phi : f^{\bullet} \comp P \comp f_{\bullet} \to Q$ with respect to $f^{\bullet} \ladj f_{\bullet}$ are denoted
\[ \begin{array}{lcccr} {\phi^c : f^{\bullet} \comp P \to Q \comp f^{\bullet}} &&
{\phi^l : P \comp f_{\bullet} \to f_{\bullet} \comp Q} &&
{\tilde{\phi} : P \to f_{\bullet} \comp Q \comp f^{\bullet}.} \end{array} \]
Note in particular that when $Q$ underlies a monad on $I$ in $\Polyc{\Cat}$, then the composite $f_{\bullet} \comp Q \comp f^{\bullet}$ underlies a monad on $I$.
\begin{defn}\label{defn:PolyMnd}
An object of the category $\PolyMnd{\Cat}$ is a again a pair $(I,P)$ with $P$ this time a monad on $I$ in $\Polyc{\Cat}$, and we shall often adopt the abuse of referring to both the endomorphism and the monad as $P$. A morphism $(I,P) \to (J,Q)$ in $\PolyMnd{\Cat}$ is a morphism $(f,\phi)$ in $\PolyEnd{\Cat}$, together with the condition that $\tilde{\phi} : P \to f_{\bullet} \comp Q \comp f^{\bullet}$ is a morphism of monads on $I$.
\end{defn}
This last condition of Definition \ref{defn:PolyMnd} admits reformulations in the language of \cite{Street-FTM}, namely that $(f^{\bullet},\phi^c)$ is a monad opfunctor, or equivalently, that $(f_{\bullet},\phi^l)$ is a monad functor.
\begin{exam}\label{exam:SMCMnd}
The basic example of a 2-monad on $\Cat$ arising from a polynomial in $\Cat$ is the 2-monad $\SMCMnd$ for symmetric monoidal categories, and was described in detail in Section 5 of \cite{Weber-PolynomialFunctors}. Its underlying endomorphism in $\Polyc{\Cat}$ is
\[ \xygraph{{1}="p0" [r] {\mathbb{P}_*}="p1" [r] {\mathbb{P}}="p2" [r] {1}="p3" "p0":@{<-}"p1"^-{}:"p2"^-{U^{\mathbb{\P}}}:"p3"^-{}} \]
where $\mathbb{P}$ is the category of natural numbers and permutations (that is, a skeleton of the category of finite sets and bijections), $\mathbb{P}_*$ is the corresponding skeleton of finite pointed sets and base point preserving bijections, and $U^{\mathbb{P}}$ is the forgetful functor. We also denote this polynomial by $\SMCMnd$. As explained in \cite{Weber-PolynomialFunctors}, the properties on $U^{\mathbb{P}}$ ensure that the 2-monad $\SMCMnd$ has good formal properties -- it is familial, opfamilial and sifted colimit preserving.
\end{exam}

\section{Operads as polynomial monads}
\label{sec:SMultiCats-Poly}
In this section we describe collections and operads in terms of polynomials, culminating in the characterisation in Theorem \ref{thm:SMultiCatAsPolyMndCat} of operads as polynomial monads over $\SMCMnd$. In the Remarks \ref{rem:clubs} and \ref{rem:Cat-in-Opd} which follow, we exhibit Kelly's clubs and $\Cat$-operads in similar terms. The 2-monad associated to an operad is then unpacked in elementary terms in Lemmas  \ref{lem:endofunctor-from-collection}-\ref{lem:mult-of-operad}.

We begin by recalling some basic definitions and fixing our notation and terminology. It will often be convenient to denote a typical element $(x_1,...,x_n)$ of a cartesian product $\prod_{i=1}^n X_i$ of sets as $(x_i)_{1{\leq}i{\leq}n}$, or as $(x_i)_i$ when $n$ is understood or when we wish it to be implicit. Moreover we denote by $\Sigma_n$ the group of permutations of $\{1,...,n\}$.
\begin{defn}\label{defn:collections}
\begin{enumerate}
\item A \emph{collection} $T$ consists of a set $I$ whose elements are called the \emph{objects} or \emph{colours} of $X$, and for each pair $((i_j)_{1{\leq}j{\leq}n},i)$ consisting of a sequence $(i_j)_j$ of elements of $I$ and a single element $i \in I$, one has a set $T((i_j)_j;i)$ whose elements are called \emph{arrows} of $T$ with \emph{source} $(i_j)_j$ and \emph{target} $i$, and a typical element may be denoted as $\alpha : (i_j)_j \to i$. Furthermore given an arrow $\alpha : (i_j)_{1{\leq}j{\leq}n} \to i$ and a permutation $\rho \in \Sigma_n$, one has an arrow $\alpha\rho : (i_{\rho j})_j \to i$, this assignation being functorial in the sense that $\alpha 1_n = \alpha$ and $(\alpha\rho_1)\rho_2 = \alpha(\rho_1\rho_2)$ for all $\rho_1,\rho_2 \in \Sigma_n$.
\item Let $S$ and $T$ be collections with object sets $I$ and $J$ respectively. Then a morphism $F : S \to T$ consists of a function $f : I \to J$ between object sets and for each $((i_j)_{1{\leq}j{\leq}n},i)$, a function $F_{((i_j)_j,i)}:S((i_j)_j;i) \to T((fi_j)_j;fi)$. These arrow mappings must be \emph{equivariant} in the sense that given $\alpha : (i_j)_{1{\leq}j{\leq}n} \to i$ in $S$ and a permutation $\rho \in \Sigma_n$, $(f\alpha)\rho = f(\alpha\rho)$.
\end{enumerate}
The category of collections and their morphisms is denoted $\Coll$.
\end{defn}
\begin{defn}\label{defn:operads}
\begin{enumerate}
\item An \emph{operad} is a collection $T$, with object set denoted $I$, together with
\begin{itemize}
\item (units): for $i \in I$, an arrow $1_i:(i) \to i$.
\item (compositions): given an arrow $\alpha : (i_j)_{1{\leq}j{\leq}n} \to i$ of $T$, and a sequence $(\beta_j:(i_{jk})_{1{\leq}k{\leq}m_j} \to i_j)_j$ of arrows of $T$, their \emph{composite} is an arrow $\alpha \comp (\beta_j)_j : (i_{jk})_{jk} \to i$, where
\[ (i_{jk})_{jk} = (i_{11},...,i_{1m_1},...,i_{n1},...,i_{nm_n})  \]
is the sequence of length $(m_1+...+m_n)$ obtained by concatenating the domains of the $y_j$.
\end{itemize}
This data must satisfy the following axioms. The \emph{unitality} and \emph{associativity} of composition say that given
\[ \begin{array}{lccccr} {\alpha : (i_j)_{1{\leq}j{\leq}n} \to i} && {\beta_j : (i_{jk})_{1{\leq}k{\leq}m_j} \to i_j} && {\gamma_{jk}:(i_{jkl})_{1{\leq}l{\leq}p_{jk}} \to i_{jk}} \end{array} \]
one has
\[ \begin{array}{lccr} {1_i \comp (\alpha) = \alpha = \alpha \comp (1_{i_j})_j} &&& {\alpha \comp (\beta_j \comp (\gamma_{jk})_k)_j = (\alpha \comp (\beta_j)_j) \comp (\gamma_{jk})_{jk}.} \end{array} \]
\emph{Equivariance} of composition says that given $\rho \in \Sigma_n$ and $\rho_j \in \Sigma_{m_j}$ for $1 \leq j \leq n$, one has
$(\alpha \comp (\beta_j)_j)(\rho(\rho_j)_j) = (\alpha\rho) \comp (\beta_j\rho_j)_j$
where $(\rho(\rho_j)_j)$ is the permutation of $\Sigma_{i=1}^n m_i$ symbols given by permuting the $n$-blocks $(m_1,...,m_n)$ using $\rho$, and permuting the elements within the $j$-th block using $\rho_j$.
\item A morphism $S \to T$ of operads is a morphism $F : S \to T$ of the underlying collections, with underlying object map denoted as $f : I \to J$, such that
\[ \begin{array}{lccr} {F1_i = 1_{fi}} &&& {F(\alpha \comp (\beta_j)_j) = (F\alpha) \comp (F\beta_j)_j} \end{array} \]
for all objects $i$ of $S$, and arrows $\alpha$ and $(\beta_j)_j$ of $S$ as above.
\end{enumerate}
The category of operads and their morphisms is denoted $\Opd$.
\end{defn}
At the end of this section we will have established, for an operad $T$ with set of colours $I$, the corresponding 2-monad $T$ on $\Cat/I$. The explicit description of the 2-monad $T$ is in terms of labelled operations in the sense of
\begin{defn}\label{defn:labelled-operations}
An \emph{operation of $T$ labelled in $X$} is a pair $(\alpha,(x_j)_j)$, where $\alpha : (i_j)_j \to i$ is an arrow of $T$, and $x_j \in X_{i_j}$. 
A \emph{morphism} $(\alpha,(x_j)_j) \to (\beta,(y_j)_j)$ is a pair $(\rho,(\gamma_j)_j)$ where $\rho$ is a permutation such that $\alpha = \beta\rho$, and $\gamma_j : x_j \to y_{\rho j}$ is a morphism of $X_{\rho j}$ for each $j$.
\end{defn}
It is also useful to depict a labelled operation $(\alpha,(x_j)_j)$ of Definition \ref{defn:labelled-operations} as
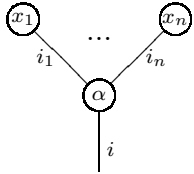
\[ \xygraph{!{0;(.8,0):(0,1)::} 
{\scriptstyle{\alpha}} *\xycircle<6pt,6pt>{-}="p0" [ul]
{\scriptstyle{x_1}} *\xycircle<6pt,6pt>{-}="p1" [r(2)]
{\scriptstyle{x_n}} *\xycircle<6pt,6pt>{-}="p2"
"p0" (-"p1",-"p2",-[d],[u(.75)] {...},[u(.5)l(.7)] {\scriptstyle{i_1}},[u(.5)r(.75)] {\scriptstyle{i_n}},[d(.7)r(.15)] {\scriptstyle{i}})} \]
and in such diagramatic terms, a morphism amounts to a shuffling of the inputs of the operations, together with a levelwise family of morphisms of $X$. Operations of $T$ labelled by $X$ form the category $TX$, which lives over $I$ via the assignation of codomains of the labelled operations. The full details will be established in Lemmas \ref{lem:endofunctor-from-collection}-\ref{lem:mult-of-operad} below.
\begin{constn}\label{const:collection-PolyEndo}
Denoting by $\PolyEnd {\Cat}/\SMCMnd$ the slice category of $\PolyEnd{\Cat}$ over the polynomial endofunctor $\SMCMnd$ of Example \ref{exam:SMCMnd}, we now construct a functor
\[ \ca N : \Coll \longrightarrow \PolyEnd {\Cat}/\SMCMnd. \]
To any collection $T$ whose set of colours is $I$ we associate a morphism
\begin{equation}\label{eq:SMCat|->PolyEndIntoSym}
\begin{gathered}
\xygraph{!{0;(1.5,0):(0,.6667)::} {I}="p0" [r] {E_T}="p1" [r] {B_T}="p2" [r] {I}="p3" [d] {1}="p4" [l] {\mathbb{P}}="p5" [l] {\mathbb{P}_*}="p6" [l] {1}="p7" "p0":@{<-}"p1"^-{s_T}:"p2"^-{p_T}:"p3"^-{t_T}:"p4"^-{}:@{<-}"p5"^-{}:@{<-}"p6"^-{U^{\P}}:"p7"^-{}:@{<-}"p0"^-{} "p1":"p6"_-{e_X} "p2":"p5"^-{b_X} "p0":@{}"p6"|-{=} "p1":@{}"p5"|-{\tn{pb}} "p2":@{}"p4"|-{=}}
\end{gathered}
\end{equation}
of $\PolyEnd {\Cat}$ as follows. Denoting by $T_n$ the set of arrows of $T$ whose source is a sequence of length $n$, $n \mapsto T_n$ is the effect on objects of a functor $\mathbb{P}^{\op} \to \Set$, and the corresponding discrete fibration is $b_T : B_T \to \mathbb{P}$. In explicit terms an object of $B_T$ is an arrow $\alpha : (i_j)_j \to i$ of $T$, and an arrow $\alpha \to \beta$ of $B_T$ is a permutation $\rho$ such that $\alpha = \beta\rho$. An object of $E_T$ is a pair $(\alpha,j)$ where $\alpha : (i_j)_{1{\leq}j{\leq}n} \to i$ is an arrow of $T$ and $1 \leq j \leq n$, and an arrow $(\alpha,j) \to (\beta,k)$ of $E_T$ is a permutation $\rho$ such that $\alpha = \beta\rho$ and $\rho j = k$. Thus a typical arrow of $E_T$ can be written as $\rho : (\alpha\rho,j) \to (\alpha,\rho j)$. The object maps of $s_T$, $p_T$, $t_T$, $b_T$ and $e_T$ are
\[ \xygraph{{i_j}="p0" [r] {(\alpha,j)}="p1" [r] {\alpha}="p2" [r] {i}="p3" [dl] {n}="p4" [l] {(n,j)}="p5"  "p0":@{<-|}"p1"^-{s_T}:@{|->}"p2"^-{p_T}:@{|->}"p3"^-{t_T} "p1":@{|->}"p5"_-{e_T} "p2":@{|->}"p4"^-{b_T}} \]
in which $n$ is the length of the domain sequence of $\alpha$, the arrow maps are defined analogously, and the pullback square is easily verified. Since $U^{\P}$ is a discrete fibration with finite fibres, so is $p_T$ since such properties on a functor are pullback stable. Thus $p_T$ is an exponentiable functor. We denote by $P_T$ the polynomial $(s_T,p_T,t_T)$, and by $\ca NT : P_T \to \sm$ the morphism $(b_T,e_T)$ of $\PolyMnd {\Cat}$.

Given a morphism $F : S \to T$ of collections the functor $F_1 : B_S \to B_T$ on objects acts as the arrow map of $F$, and sends $\rho : \alpha\rho \to \alpha$ to $\rho : (F\alpha)\rho \to F\alpha$. Clearly one has $F_1b_T = b_S$. The functor $F_2 : E_S \to E_T$ sends $(\alpha,j)$ and $\rho : (\alpha\rho,j) \to (\alpha,\rho j)$ to $(F\alpha,j)$ and $\rho : ((F\alpha)\rho,j) \to (F\alpha,\rho j)$ respectively. Clearly one has $F_2e_T = e_X$ and that $(f,F_1,F_2)$ are the components of a morphism $(I,(s_S,p_S,t_S)) \to (J,(s_T,p_T,t_T))$ of $\PolyEnd {\Cat}$.
\end{constn}
\begin{prop}\label{prop:SMultiGph->PolyEndCat}
The functor $\ca N$ restricts to an equivalence between $\Coll$ and the full subcategory of $\PolyEnd {\Cat}/\SMCMnd$ consisting of those morphisms
\[ \xygraph{!{0;(1.5,0):(0,.6667)::} {I}="p0" [r] {E}="p1" [r] {B}="p2" [r] {I}="p3" [d] {1}="p4" [l] {\mathbb{P}}="p5" [l] {\mathbb{P}_*}="p6" [l] {1}="p7" "p0":@{<-}"p1"^-{s}:"p2"^-{p}:"p3"^-{t}:"p4"^-{}:@{<-}"p5"^-{}:@{<-}"p6"^-{U^{\P}}:"p7"^-{}:@{<-}"p0"^-{} "p1":"p6"_-{e} "p2":"p5"^-{b} "p0":@{}"p6"|-{=} "p1":@{}"p5"|-{\tn{pb}} "p2":@{}"p4"|-{=}} \]
such that $I$ is discrete and the functor $b$ is a discrete fibration.
\end{prop}
\begin{proof}
We first verify that $\ca N$ is fully faithful. Given collections $S$ and $T$, and a morphism
\[ \xygraph{!{0;(1.5,0):(0,.6667)::} {I}="p0" [r] {E_S}="p1" [r] {B_S}="p2" [r] {I}="p3" [d] {J}="p4" [l] {B_T}="p5" [l] {E_T}="p6" [l] {J}="p7" "p0":@{<-}"p1"^-{s_S}:"p2"^-{p_S}:"p3"^-{t_S}:"p4"^-{f_0}:@{<-}"p5"^-{t_T}:@{<-}"p6"^-{p_T}:"p7"^-{s_T}:@{<-}"p0"^-{f_0} "p1":"p6"_-{f_2} "p2":"p5"^-{f_1} "p0":@{}"p6"|-{=} "p1":@{}"p5"|-{\tn{pb}} "p2":@{}"p4"|-{=}} \]
$\ca NS \to \ca NT$, one defines $F : S \to T$ with object map $f=f_0$, and with effect on arrows given by the object map of $f_1$. Since $f_1b_T = b_S$ $F$'s arrow map is equivariant. By definition $\ca NF = (f_0,f_1,f_2)$, and this equation determines $F$ uniquely.

For a morphism into $(1,\SMCMnd)$ in $\PolyEnd {\Cat}$ as in the statement, it suffices to exhibit it has $\ca NT$ for some collection $T$. We take the set of objects of $T$ to be $I$, and the set of arrows of $T$ to be $B$. The target of $\alpha \in B$ is taken to be $t\alpha$. Since $(e,b)$ is the structure on $p$ of a $U^{\P}$-fibration, $p^{-1}\{\alpha\}$ is a finite linearly ordered set, and applying $s$ componentwise to this produces the source sequence of $\alpha$ in $T$. Denoting by $n$ the length of this sequence and regarding $\rho \in \Sigma_n$ as an arrow of $\mathbb{P}$, $\rho$ lifts to a unique morphism of $B$ with codomain $\alpha$ since $b$ is a discrete fibration, and we denote this unique morphism as $\rho : \alpha\rho \to \alpha$. Thus we we have the required symmetric group actions, and their functoriality is just that of $b$. By construction $\ca NT$ is the morphism in $\PolyEnd {\Cat}$ of the statement.
\end{proof}
\begin{prop}\label{prop:SMultiCat->PolyMndCat}
The functor $\ca N$ lifts to a functor $\overline{\ca N}$ making the square
\[ \xygraph{!{0;(3,0):(0,.3333)::} {\Opd}="p0" [r] {\PolyMnd {\Cat}/\SMCMnd}="p1" [d] {\PolyEnd {\Cat}/\SMCMnd}="p2" [l] {\Coll}="p3" "p0":"p1"^-{\overline{\ca N}}:"p2"^-{}:@{<-}"p3"^-{\ca N}:@{<-}"p0"^-{}} \]
in which the vertical functors are the forgetful functors, a pullback.
\end{prop}
\begin{proof}
We shall first establish a bijection between operad structures on a collection $T$, and polynomial monad structures on $P_T$ making $\ca NT$ a morphism of polynomial monads. Second, given collections $S$ and $T$ and a morphism $F : S \to T$ of their underlying collections, we shall prove that $F$ is a morphism of operads iff $\ca NF$ is a morphism of the corresponding polynomial monads over $\SMCMnd$.

Let $T$ be a collection with object set $I$. To give units for $T$ is to give a functor $u_{T,1}:I \to B_T$ such that: (1) $t_Tu_{T,1} = 1_I$, (2) for each $i$ the fibre $p_T^{-1}\{u_{T,1}i\}$ consists of a unique element $u_{T,2}i$, and (3) $su_{T,2}i = i$. Since $b_T$ sends elements with singleton fibres to $1 \in \mathbb{P}$ and $e_T$ sends the unique elements of those fibres to $(1,1) \in \mathbb{P}_*$, $\ca NT$ commutes with these unit maps and those of $\SMCMnd$. Thus to give units for $T$ is to give $u_T:1_I \to P_T$ with respect to which $\ca NT$ is compatible.

We now characterise compositions for an operad structure on $T$ in similar terms. The polynomial $P_T \comp P_T$ is formed as
\[ \xygraph{{I}="b1" [r] {E_T}="b2" [r] {B_T}="b3" [r] {I}="b4" [r] {E_T}="b5" [r] {B_T}="b6" [r] {I.}="b7" "b4" [u] {B_T \times_I E_T}="p1" [u] {F_T}="dl" ([r(1.5)] {B_T^{(2)}}="dr", [l(1.5)] {E_T^{(2)}}="p2")
"b1":@{<-}"b2"_-{s_T}:"b3"_-{p_T}:"b4"_-{t_T}:@{<-}"b5"_-{s_T}:"b6"_-{p_T}:"b7"_-{t_T} "dl":"p1"_-{}(:"b3"_-{},:"b5"^-{}) "b2":@{<-}"p2"_-{}:"dl"_-{}:"dr"_-{}:"b6"^(.7){q} "b1":@{<-}"p2"^-{s_T^{(2)}} "dr":"b7"^-{t_T^{(2)}} "p2":@/^{1pc}/"dr"^-{p_T^{(2)}}
"b3" [u(1.25)] {\scriptstyle{\tn{pb}}} "b5" [u(1.25)] {\scriptstyle{\tn{dpb}}} "b4" [u(.5)] {\scriptstyle{\tn{pb}}}} \]
An object of $B_T^{(2)}$ can be identified with a functor $b:[0] \to B^{(2)}_T$. Writing $\alpha = qb$, an object of $B_T^{(2)}$ may be regarded as the data: (1) an arrow $\alpha : (i_j)_{1{\leq}j{\leq}n} \to i$ of $T$ viewed also as a functor $\alpha : [0] \to B_T$, and (2) a functor $b:[0] \to B_T^{(2)}$ over $B_T$. The functor $q$ is the effect of $\Pi_{p_X}$ on the functor $B_T \times_I E_T \to E_T$, and the pullback of $\alpha : [0] \to B_T$ along $p_T$ is the discrete subcategory of $E_T$ consisting of the pairs $(\alpha,j)$ for $1 \leq j \leq n$. Thus (2) amounts to giving an arrow $\beta_j$ of $T$ for each $1 \leq j \leq n$ whose target is $i_j$, and so an object of $B_T^{(2)}$ is exactly the data $(\alpha,(\beta_j)_j)$ that can be composed in the multicategory $T$.

Similarly an arrow of $B_T^{(2)}$ can be identified with a functor $[1] \to B_T$ together with a functor $[1] \to B_T^{(2)}$ over $B_T$. This first datum is just an arrow of $B_T$, and so is of the form $\rho : \alpha\rho \to \alpha$, where $\alpha : (i_j)_{1{\leq}j{\leq}n} \to i$ is an arrow of $T$, and $\rho \in \Sigma_n$. Pulling back the functor $[1] \to B_T$ so determined along $p_T$ produces a category with objects of the form $(\alpha\rho,j)$ or $(\alpha,j)$ for $1 \leq j \leq n$, and invertible arrows $(\alpha\rho,j) \to (\alpha,\rho j)$. Thus the second piece of data determining an arrow of $B_T^{(2)}$ amounts to giving morphisms $\rho_j : \beta_j\rho_j \to \beta$ in $B$, such that $t\beta_j = i_j$ for each $j$. Thus the general form of an arrow of $B_T^{(2)}$ is
\begin{equation}\label{eq:arrow-of-B2}
(\rho,(\rho_j)_{1{\leq}j{\leq}n}) : (\alpha\rho,(\beta_j\rho_j)_j) \longrightarrow (\alpha,(\beta_j)_j)
\end{equation}
where $\alpha : (i_j)_{1{\leq}j{\leq}n} \to i$ is an arrow of $T$, $\rho \in \Sigma_n$, and for each $j$, $\beta_j : (i_{jk})_{1{\leq}k{\leq}m_j} \to i_j$ is an arrow of $T$ and $\rho_j \in \Sigma_{m_j}$.

A description of $E_T^{(2)}$ is now easily obtainable, since $E_T^{(2)}$ is obtained by pulling back the functor $B_T^{(2)} \to B_T$ which we now know explicitly. So an object of $E_T^{(2)}$ consists of $(\alpha,(\beta_j)_j,j,k)$, where $\alpha : (i_j)_{1{\leq}j{\leq}n} \to i$ and $\beta_j:(i_{jk})_{1{\leq}k{\leq}m_j} \to i_j$ are arrows of $T$, $1 \leq j \leq n$ and $1 \leq k \leq m_j$. Morphisms of $E_T^{(2)}$ are of the form
\begin{equation}\label{eq:arrow-of-E2}
(\rho,(\rho_j)_{1{\leq}j{\leq}n}) : (\alpha\rho,(\beta_j\rho_j)_j,j,k) \longrightarrow (\alpha,(\beta_j)_j,\rho j,\rho_jk)
\end{equation}
and the explicit descriptions of the functors $s_T^{(2)}$, $p_T^{(2)}$ and $t_T^{(2)}$ are now self-evident.

Given these details, an object map for a functor $m_{T,1}:B^{(2)}_T \to B_T$ amounts to assignations $(\alpha,(\beta_j)_j) \mapsto \alpha \comp (\beta_j)_j$. Giving $m_{T,1}$ on arrows amounts to assigning to (\ref{eq:arrow-of-B2}), an arrow $(\alpha\rho) \comp (\beta\rho_j)_j \to \alpha \comp (\beta_j)_j$ of $B_T$, and the compatibility of $m_{T,1}$ with $\ca NT$ and the corresponding component of $\SMCMnd$'s multiplication, amounts to the underlying permutation of this arrow of $B_T$ being obtained via the substitution of permutations, which corresponds to equivariance. To say that the target of $\alpha \comp (\beta_j)_j$ is that of $\alpha$ for all $(\alpha,(\beta_j)_j)$, is to say that $m_1t_T = t^{(2)}_T$. A functor $m_{T,2}:E^{(2)}_T \to E_T$ providing the other component of $m_T :P_T \comp P_T \to P_T$ is determined by its restrictions to the fibres of $p^{(2)}_T$ which are finite discrete, and giving these amounts to specifying that for all $(\alpha,(\beta_j)_j)$, the domain of the composite $\alpha \comp (\beta_j)_j$ is the concatenation of the domains of the $\beta_j$ as for composition in an operad. In summary, to give $T$ a composition operation is to give $m_T : P_T \comp P_T \to P_T$ in $\Polyc {\Cat}$. The straightforward though tedious verification that the unit and associative laws for $(u_T,m_T)$ correspond with the unitality and associativity of composition for $T$ is left to the reader.

Let $S$ and $T$ be collections and $F : S \to T$ be a morphism of their underlying collections. To say $\ca NF$ is compatible with units amounts to the equation $F_1u_{S,1} = u_{T,1}$, the equation $F_2u_{S,2} = u_{T,2}$ being a consequence, and this in turn is equivalent to saying that $F$ sends identities in $S$ to identities in $T$. We leave to the reader the straightforward verification that $\ca NF$'s compatibility with multiplications amounts to the formulae $F(\alpha \comp (\beta_j)_j) = F\alpha \comp (F\beta_j)_j$ expressing $F$'s compatibilty with composition.
\end{proof}
By Propositions \ref{prop:SMultiGph->PolyEndCat} and \ref{prop:SMultiCat->PolyMndCat} we have
\begin{thm}\label{thm:SMultiCatAsPolyMndCat}
The functor $\overline{\ca N}$ restricts to an equivalence between \\ $\Opd$ and the full subcategory of $\PolyMnd {\Cat}/\SMCMnd$ consisting of those monad morphisms
\[ \xygraph{!{0;(1.5,0):(0,.6667)::} {I}="p0" [r] {E}="p1" [r] {B}="p2" [r] {I}="p3" [d] {1}="p4" [l] {\mathbb{P}}="p5" [l] {\mathbb{P}_*}="p6" [l] {1}="p7" "p0":@{<-}"p1"^-{s}:"p2"^-{p}:"p3"^-{t}:"p4"^-{}:@{<-}"p5"^-{}:@{<-}"p6"^-{U^{\P}}:"p7"^-{}:@{<-}"p0"^-{} "p1":"p6"_-{e} "p2":"p5"^-{b} "p0":@{}"p6"|-{=} "p1":@{}"p5"|-{\tn{pb}} "p2":@{}"p4"|-{=}} \]
such that $I$ is discrete and the functor $b$ is a discrete fibration.
\end{thm}
\begin{rem}\label{rem:clubs}
A \emph{club} in the sense of Max Kelly \cite{Kelly-ClubsDoctrines, Kelly-ClubsDataTypeConstructors} can be identified as a 2-monad $A$ on $\Cat$ together with a cartesian monad morphism $\phi : A \to \SMCMnd$. In general when one has a cartesian monad morphism into a polynomial monad, the domain monad is also easily exhibited as polynomial, and so clubs can be identified as those objects
\[ \xygraph{!{0;(1.5,0):(0,.6667)::} {I}="p0" [r] {E}="p1" [r] {B}="p2" [r] {I}="p3" [d] {1}="p4" [l] {\mathbb{P}}="p5" [l] {\mathbb{P}_*}="p6" [l] {1}="p7" "p0":@{<-}"p1"^-{s}:"p2"^-{p}:"p3"^-{t}:"p4"^-{}:@{<-}"p5"^-{}:@{<-}"p6"^-{U^{\P}}:"p7"^-{}:@{<-}"p0"^-{} "p1":"p6"_-{e} "p2":"p5"^-{b} "p0":@{}"p6"|-{=} "p1":@{}"p5"|-{\tn{pb}} "p2":@{}"p4"|-{=}} \]
of $\PolyMnd {\Cat}/\SMCMnd$ such that $I = 1$.
\end{rem}
\begin{rem}\label{rem:Cat-operads}
A \emph{$\Cat$-operad} is defined in the same way as an operad is, except that the homs are categories and the units and compositions define functors, this being an instance of how the notion of operad can be enriched. Equivalently denoting by $\Opd_I$ the category of operads with objects set $I$ and morphisms whose object function is $1_I$, a $\Cat$-operad with set of objects $I$ is a category internal to $\Opd_I$. As explained in \cite{Weber-PolynomialFunctors} pullbacks in $\Polyc{\Cat}(I,I)$ are formed componentwise, and its straightforward to verify that the restriction
\[ \Opd_I \longrightarrow \Polyc{\Cat}(I,I) \]
of $\overline{\ca N}$ preserves pullbacks. From this it is straightforward to see that one can identify $\Cat$-operads with objects
\[ \xygraph{!{0;(1.5,0):(0,.6667)::} {I}="p0" [r] {E}="p1" [r] {B}="p2" [r] {I}="p3" [d] {1}="p4" [l] {\mathbb{P}}="p5" [l] {\mathbb{P}_*}="p6" [l] {1}="p7" "p0":@{<-}"p1"^-{s}:"p2"^-{p}:"p3"^-{t}:"p4"^-{}:@{<-}"p5"^-{}:@{<-}"p6"^-{U^{\P}}:"p7"^-{}:@{<-}"p0"^-{} "p1":"p6"_-{e} "p2":"p5"^-{b} "p0":@{}"p6"|-{=} "p1":@{}"p5"|-{\tn{pb}} "p2":@{}"p4"|-{=}} \]
of $\PolyMnd {\Cat}/\SMCMnd$, together with the structure of a split fibration on $b$.
\end{rem}
\begin{rem}\label{rem:Cat-in-Opd}
A category can be regarded as an operad $T$ in which the source of every arrow is a sequence of length $1$, which is so iff in its underlying polynomial depicted on the left
\[ \xygraph{{\xybox{\xygraph{{I}="p0" [r] {E_T}="p1" [r] {B_T}="p2" [r] {I}="p3" "p0":@{<-}"p1"^-{s_T}:"p2"^-{p_T}:"p3"^-{t_T}}}}
[r(4)]
{\xybox{\xygraph{{I}="p0" [r] {E}="p1" [r] {B}="p2" [r] {I}="p3" "p0":@{<-}"p1"^-{s}:"p2"^-{p}:"p3"^-{t}}}}} \]
$p_T$ is an isomorphism. In this case $E_T$ and $B_T$ are discrete, and $b_T : B_T \to \mathbb{P}$ and $e_T : E_T \to \mathbb{P}_*$ are determined uniquely by the polynomial $(s_T,p_T,t_T)$. For any polynomial as on the right in the previous display in which $p$ is an isomorphism and $E$ and $B$ are discrete, one as a unique isomorphism $(s,p,t) \iso (sp^{-1},1_B,t)$ with a span of sets. Thus the equivalence of Theorem \ref{thm:SMultiCatAsPolyMndCat} essentially restricts to an equivalence of categories which on objects identifies a category with its corresponding monad in $\Span{\Set}$.
\end{rem}
An operad $T$ with object set $I$ determines a 2-monad $(I,P_T)$ in the 2-bicategory $\Polyc {\Cat}$, and so by means of $\PFun {\Cat}$, a 2-monad on $\Cat/I$.
\begin{notn}\label{notn:associated-2-monad}
Given an operad $T$ with object set $I$, we also denote the associated 2-monad on $\Cat/I$ as $T$.
\end{notn}
\begin{exam}\label{exam:Com}
The terminal operad which has one object and a unique arrow of with source of length $n$, is usually denoted as $\Com$. Its corresponding polynomial is $\SMCMnd$. Following notation \ref{notn:associated-2-monad} one thus has $\Com = \SMCMnd$.
\end{exam}
We now turn to the task of giving an explicit description of this 2-monad $T$ on $\Cat/I$. Let $X \in \Cat/I$. We regard $X$ both as $X \to I$ a category equipped with a functor into $I$, and as $(X_i)_{i \in I}$ an $I$-indexed family of categories. Applying the general calculation of Example \ref{exam:polyfunctor-Cat-middle-dopfib} to the polynomial $(s_T,p_T,t_T)$ one obtains
\begin{lem}\label{lem:endofunctor-from-collection}
Let $T$ be a collection with object set $I$ and $X \in \Cat/I$. Then $TX$ may be identified with the category of operations of $T$ labelled in $X$, in the sense of Definition \ref{defn:labelled-operations}.
\end{lem}
Similarly one can unpack the explicit description of $Tf : TX \to TY$ given $f : X \to Y$ in $\Cat/I$. By definition $Tf$ is induced from the functoriality of pullbacks and distributivity pullbacks in
\begin{equation}\label{eq:diag-for-2-functoriality-collection}
\begin{gathered}
\xygraph{!{0;(1.5,0):(0,.55)::}
{I}="p0" [r] {E_T}="p1" [r] {B_T}="p2" [r] {I}="p3" "p0":@{<-}"p1"^-{}:"p2"^-{}:"p3"^-{}
"p0" ([lu] {X}="q0",[ld] {Y}="r0")
"q0" [r] {X \times_I E_T}="q1" [ur] {T_{\bullet}X}="q2" [r] {TX}="q3"
"q2" (:"q1"(:"q0":"p0",:"p1"),:"q3"(:"p2",:"p3"))
"r0" [r] {Y \times_I E_T}="r1" [dr] {T_{\bullet}Y}="r2" [r] {TY}="r3"
"r2" (:"r1"(:"r0":"p0",:"p1"),:"r3"(:"p2",:"p3"))
"q0":"r0"_-{f} "q1":@/_{1pc}/@{.>}"r1" "q2":@/_{1pc}/@{.>}"r2" "q3":@/_{1pc}/@{.>}"r3"_(.3){Tf}
"p1" ([u(1)r(.3)] {\scriptsize{\tn{dpb}}}, [d(1)r(.3)] {\scriptsize{\tn{dpb}}})}
\end{gathered}
\end{equation}
and then by tracing through the explicit descriptions as in Example \ref{exam:polyfunctor-Cat-middle-dopfib} one can verify
\begin{lem}\label{lem:endofunctor-from-collection-arrow-map}
Let $T$ be a collection with object set $I$. Given a morphism $f : X \to Y$ in $\Cat/I$, one has
\[ \begin{array}{lccr} {Tf(\alpha,(x_j)_j) = (\alpha,(fx_j)_j)} &&&
{Tf(\rho,(\beta_j)_j) = (\rho,(f\beta_j)_j).} \end{array} \]
\end{lem}
Remembering that the pullbacks and distributivity pullbacks that appear in (\ref{eq:diag-for-2-functoriality-collection}) enjoy a 2-dimensional universal property, one can in much the same way verify
\begin{lem}\label{lem:endofunctor-from-collection-2cell-map}
Let $T$ be a collection with object set $I$. Given morphisms $f$ and $g : X \to Y$ in $\Cat/I$ and a 2-cell $\phi : f \to g$, one has
\[ \begin{array}{lcr} {(T\phi)_{(\alpha,(x_j)_j)}} & {=} &
{(\alpha,(fx_j)_j) \xrightarrow{(1,(\phi_{x_j})_j)} (\alpha,(gx_j)_j).} \end{array} \]
\end{lem}
Lemmas \ref{lem:endofunctor-from-collection}-\ref{lem:endofunctor-from-collection-2cell-map} together describe the endo-2-functor of $\Cat/I$ corresponding to a collection $T$ with object set $I$ in terms of labelled operations. We now extend this to a description of the 2-monad corresponding to an operad. In the proof of Proposition \ref{prop:SMultiCat->PolyMndCat} we obtained an explicit understanding of how the identity arrows arrows of an operad provide the unit data for a monad in $\Polyc {\Cat}$. Putting this together with the explicit description of the homomorphism $\PFun {\Cat} : \Polyc {\Cat} \to \TwoCat$ of 2-bicategories, we obtain
\begin{lem}\label{lem:unit-of-operad}
Let $T$ be an operad with object set $I$. Given an object $X$ of $\Cat/I$ one has
\[ \begin{array}{lccr} {\eta^T_X(x) = (1_i,(x))} &&& {\eta^T_X(\beta) = (1_1,(\beta))} \end{array} \]
for any $x$ and $\beta : x \to y$ in $X_i$.
\end{lem}
Similarly from the explicit understanding of how the compositions of an operad give rise to the multiplication data for a monad in $\Polyc {\Cat}$ obtained in Proposition \ref{prop:SMultiCat->PolyMndCat}, we further obtain
\begin{lem}\label{lem:mult-of-operad}
Let $T$ be an operad with object set $I$ and $X \in \Cat/I$. Then the effect of $\mu^T_X$ on objects is given by
\[ \mu^T_X(\alpha,(\alpha_j,(x_{jk})_k)_j) = (\alpha(\alpha_j)_j,(x_{jk})_{jk}). \]
The effect of $\mu^T_X$ on an arrow
\[ (\rho,(\rho_j,(\beta_{jk})_k)_j) : (\alpha\rho,(\alpha_j\rho_j,(x_{jk})_k)_j) \longrightarrow (\alpha,(\alpha_j,(y_{jk})_k)_j) \]
of $T^2X$ is
\[ (\rho(\rho_j)_j,(\beta_{jk})_{jk}) : ((\alpha(\alpha_j)_j)(\rho(\rho_j)_j),(x_{jk})_{jk}) \longrightarrow (\alpha(\alpha_j)_j,(y_{jk})_{jk}). \]
\end{lem}
In diagramatic terms the object map of $\mu^T_X$ may be depicted as
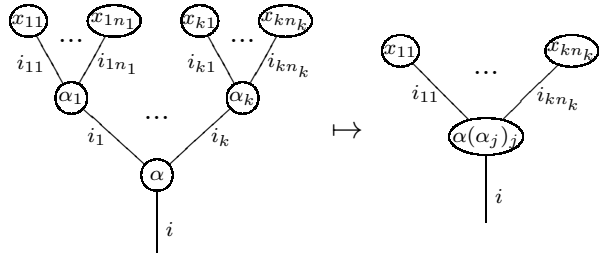
\[ \xygraph{{\xybox{\xygraph{!{0;(.9,0):(0,.9)::} 
{\scriptstyle{\alpha}} *\xycircle<6pt,6pt>{-}="p0" [ul]
{\scriptstyle{\alpha_1}} *\xycircle<6pt,6pt>{-}="p1" [r(2)]
{\scriptstyle{\alpha_k}} *\xycircle<6pt,6pt>{-}="p2"
"p0" (-"p1",-"p2",-[d],[u(.75)] {...},[u(.5)l(.7)] {\scriptstyle{i_1}},[u(.5)r(.75)] {\scriptstyle{i_k}},[d(.7)r(.15)] {\scriptstyle{i}})
"p1" [u(1)l(.5)] {\scriptstyle{x_{11}}} *\xycircle<7pt,6pt>{-}="q1" [r]
{\scriptstyle{x_{1n_1}}} *\xycircle<10pt,6pt>{-}="q2"
"p1" (-"q1",-"q2",[u(.75)] {...},[u(.45)l(.47)] {\scriptstyle{i_{11}}},[u(.45)r(.55)] {\scriptstyle{i_{1n_1}}})
"p2" [u(1)l(.5)] {\scriptstyle{x_{k1}}} *\xycircle<7pt,6pt>{-}="r1" [r]
{\scriptstyle{x_{kn_k}}} *\xycircle<10pt,6pt>{-}="r2"
"p2" (-"r1",-"r2",[u(.75)] {...},[u(.45)l(.47)] {\scriptstyle{i_{k1}}},[u(.43)r(.55)] {\scriptstyle{i_{kn_k}}})}}}
[r(2)]
{\mapsto}
[r(1.5)]
{\xybox{\xygraph{!{0;(.9,0):(0,1)::} 
{\scriptstyle{\alpha(\alpha_j)_j}} *\xycircle<14pt,7pt>{-}="p0" [ul]
{\scriptstyle{x_{11}}} *\xycircle<7pt,6pt>{-}="p1" [r(2)]
{\scriptstyle{x_{kn_k}}} *\xycircle<10pt,6pt>{-}="p2"
"p0" (-"p1",-"p2",-[d],[u(.75)] {...},[u(.5)l(.7)] {\scriptstyle{i_{11}}},[u(.45)r(.85)] {\scriptstyle{i_{kn_k}}},[d(.7)r(.15)] {\scriptstyle{i}})}}}} \]

\section{Categorical algebras of operads as weak operad morphisms into \underline{Cat}}
\label{sec:Algebras}

Recall that for a general 2-monad $(T,\eta,\mu)$ on a 2-category $\ca K$, an object $A \in \ca K$ can have various types of $T$-algebra structure. A \emph{lax $T$-algebra structure} on $A$ consists of an arrow $a:TA \to A$, coherence 2-cells $\overline{a}_0:1_A \to a\eta_A$ and $\overline{a}_2:aT(a) \to a\mu_A$ such that
\[ \xygraph{*{\xybox{\xygraph{!{0;(2,0):(0,.5)::} {a}="l" [r] {a\eta_Aa}="m" [d] {a}="r" "l":"m"^-{\overline{a}_0a}:"r"^-{\overline{a}_2\eta_{TA}}:@{<-}"l"^-{\id}}}}
[r(4.25)]
*!(0,.1){\xybox{\xygraph{!{0;(2.5,0):(0,.4)::} {aT(a)T^2(a)}="tl" [r] {a\mu_AT^2(a)}="tr" [d] {a\mu_A\mu_{TA}}="br" [l] {aT(a)T(\mu_A)}="bl" "tl":"tr"^-{\overline{a}_2T^2(a)}:"br"^-{\overline{a}_2\mu_{TA}}:@{<-}"bl"^-{\overline{a}_2T(\mu_A)}:@{<-}"tl"^-{aT(\overline{a}_2)}}}}
[r(4.25)]
*{\xybox{\xygraph{!{0;(2,0):(0,.5)::} {a}="l" [l] {aT(a)T(\eta_A)}="m" [d] {a}="r" "l":"m"_-{aT(\overline{a}_0)}:"r"_-{\overline{a}_2T(\eta_A)}:@{<-}"l"_-{\id}}}}} \]
commute. We denote a lax $T$-algebra as a pair $(A,a)$ leaving the coherence data $\overline{a}_0$ and $\overline{a}_2$ implicit. When these coherences are isomorphisms $(A,a)$ is called a \emph{pseudo} $T$-algebra, and when they are identities $(A,a)$ is called a \emph{strict} $T$-algebra.

Similarly, one has various types of $T$-algebra morphism structure on $f : A \to B$ in $\ca K$, where $A$ and $B$ underlie lax $T$-algebras $(A,a)$ and $(B,b)$. A \emph{lax morphism} $(A,a) \to (B,b)$ is a pair $(f,\overline{f})$, where $f:A \to B$ and $\overline{f}:bT(f) \to fa$ such that
\[ \xygraph{{\xybox{\xygraph{{f}="l" [dl] {bT(f)\eta_A}="m" [r(2)] {fa\eta_A}="r" "l":"m"_-{\overline{b}_0f}:"r"_-{\overline{f}\eta_A}:@{<-}"l"_-{f\overline{a}_0}}}}
[r(6)]
{\xybox{\xygraph{!{0;(2,0):(0,.5)::} {bT(b)T^2(f)}="p1" [r(2)] {b\mu_BT^2(f)}="p2" [d] {fa\mu_A}="p3" [l] {faT(a)}="p4" [l] {bT(fa)}="p5" "p1":"p2"^-{\overline{b}_2T^2(f)}:"p3"^-{\overline{f}\mu_A}:@{<-}"p4"^-{f\overline{a}_2}:@{<-}"p5"^-{\overline{f}T(a)}:@{<-}"p1"^-{bT(\overline{f})}}}}} \]
commute. Modifying this definition by reversing the direction of $\overline{f}$ gives the notion of \emph{colax morphism}, when $\overline{f}$ is an isomorphism $f$ is said to be a \emph{pseudo morphism}, and when $\overline{f}$ is an identity $f$ is said to be a \emph{strict morphism}.

Of course, there are also algebra 2-cells. Given lax $T$-algebras $(A,a)$ and $(B,b)$, and lax morphisms of $T$-algebras $(f,\overline{f})$ and $(g,\overline{g}) : (A,a) \to (B,b)$, an \emph{algebra 2-cell} $(f,\overline{f}) \to (g,\overline{g})$ is a 2-cell $\psi : f \to g$ in $\ca K$ such that $(\psi a)\overline{f} = \overline{g}(bT(\psi))$. Algebra 2-cells between colax morphisms are defined similarly. As such, to a 2-monad $T$ one can associate a variety of different 2-categories of algebras. The established notation for these, see for instance \cite{BWellKellyPower-2DMndThy, Lack-Codescent} is given in the following table.

\begin{table}[h]
\centering
\begin{tabular}{|l|l|l|}
\hline
Name & Objects & Arrows \\ \hline \hline
$\LaxAlg{T}$ & lax $T$-algebras & lax morphisms \\ \hline
$\PsAlgl{T}$ & pseudo $T$-algebras & lax morphisms \\ \hline
$\PsAlg{T}$ & pseudo $T$-algebras & pseudo morphisms \\ \hline
$\PsAlgs{T}$ & pseudo $T$-algebras & strict morphisms \\ \hline
$\Algl{T}$ & strict $T$-algebras & lax morphisms \\ \hline
$\Alg{T}$ & strict $T$-algebras & pseudo morphisms \\ \hline
$\Algs{T}$ & strict $T$-algebras & strict morphisms \\ \hline
\end{tabular}
\end{table}
\noindent In each case, the 2-cells are just the $T$-algebra 2-cells between the appropriate $T$-algebra morphisms.

We denote by $\CatAsOp$ the $\Cat$-operad whose objects are small categories and whose homs are given by the functor categories
\[ \begin{array}{rcl} {\CatAsOp((A_j)_{1{\leq}j{\leq}n}; B)} & = &
{\left[ \prod_{j=1}^n A_j,B \right].} \end{array} \]
In Theorem \ref{thm:categorical-algebras-of-operads}, for an operad $T$, we describe the various types of algebras of the corresponding 2-monad as weak operad morphisms into $\CatAsOp$. Then in Theorem \ref{thm:catalg-morphisms-for-operads} we establish the corresponding description of $T$-algebra morphisms as the appropriate type of natural transformation between weak operad morphisms $T \to \CatAsOp$, and in Theorem \ref{thm:catalg-2cells-for-operads} we do the same for algebra 2-cells.
\begin{defn}\label{defn:lax-morphism-into-CatAsOp}
Let $T$ be an operad with object set $I$. A \emph{lax morphism} of operads $H : T \to \CatAsOp$ consists of
\begin{itemize}
\item $\forall \, i \in I$, a category $H_i$.
\item $\forall \, \alpha : (i_j)_{1{\leq}j{\leq}n} \to i$ in $T$, a functor $H_{\alpha} : \prod_{j=1}^n H_{i_j} \to H_i$.
\item $\forall \, \alpha$ and $\rho \in \Sigma_n$, a natural transformation $\xi_{\alpha,\rho} : H_{\alpha\rho}c_{\rho} \to H_{\alpha}$, where $c_{\rho} : \prod_j H_i \iso \prod_j H_{i_{\rho j}}$ is given by permuting the factors according to $\rho$.
\item $\forall \, i \in I$, a natural transformation $\nu_i : 1_{H_i} \to H_{1_i}$.
\item $\forall \, (\alpha,(\beta_j)_j)$ where $\alpha : (i_j)_{1{\leq}j{\leq}n} \to i$ and $\beta_j : (i_{jk})_k \to i_j$ are in $T$, a natural transformation $\sigma_{\alpha,(\beta_j)_j} : H_{\alpha}(\prod_j H_{\beta_j}) \to H_{\alpha(\beta_j)_j}$.
\end{itemize}
such that $\xi_{\alpha,1} = \id$, $\xi_{\alpha,\rho_1\rho_2} = \xi_{\alpha,\rho_1}(\xi_{\alpha\rho_1,\rho_2}c_{\rho_1})$, and
\[ \xygraph{{\xybox{\xygraph{!{0;(1,0):(0,1)::} {H_{\alpha}}="p0" [r(2)] {H_{\alpha}\prod_jH_{1_{i_j}}}="p1" [dl] {H_{\alpha}}="p2" "p0":"p1"^-{\id\cdot\prod_j\nu}:"p2"^-{\sigma}:@{<-}"p0"^-{\id}}}}
[r(3.5)d(.02)]
{\xybox{\xygraph{!{0;(1,0):(0,1)::} {H_{\alpha}}="p0" [r(2)] {H_{1_i}H_{\alpha}}="p1" [dl] {H_{\alpha}}="p2" "p0":"p1"^-{\nu\cdot\id}:"p2"^-{\sigma}:@{<-}"p0"^-{\id}}}}
[r(4.5)d(.07)]
{\xybox{\xygraph{!{0;(3,0):(0,.3333)::} {H_{\alpha\rho}\prod_jH_{\beta_j,\rho_j}}="p0" [r] {H_{(\alpha(\beta_j)_j)(\rho(\rho_j)_j)}}="p1" [d] {H_{\alpha(\beta_j)_j}}="p2" [l] {H_{\alpha}\prod_jH_{\beta_j}}="p3" "p0":"p1"^-{\sigma}:"p2"^-{\xi}:@{<-}"p3"^-{\sigma}:@{<-}"p0"^-{\xi\prod_j\xi}}}}} \]
\[ \xygraph{!{0;(3,0):(0,.3333)::} {H_{\alpha}\prod_jH_{\beta_j}\prod_{j,k}H_{\gamma_{jk}}}="p0" [r(2)] {H_{\alpha}\prod_j(H_{\beta_j}\prod_kH_{\gamma_{jk}})}="p1" [d] {H_{\alpha}\prod_jH_{\beta_j(\gamma_{jk})_k}}="p2" [l] {H_{\alpha(\beta_j(\gamma_{jk})_k)_j}}="p3" [l] {H_{\alpha(\beta_j)_j}\prod_{j,k}H_{\gamma_{jk}}}="p4" "p0":"p1"^-{\id \cdot c_{\tn{sh}_n}}:"p2"^-{\id\cdot\prod_j\sigma}:"p3"^-{\sigma}:@{<-}"p4"^-{\sigma}:@{<-}"p0"^-{\sigma\cdot\id}} \]
commute, where $\tn{sh}_n \in \Sigma_{2n}$ is the ``shuffle'' permutation{\footnotemark{\footnotetext{As an endofunction of $\{1,...,2n\}$, $\tn{sh}_n$ is defined by $\tn{sh}_n(j) = \frac{j+1}{2}$ when $j$ is odd, and $\tn{sh}_n(j) = n + \frac{j}{2}$ when $j$ is even.}}}.
\end{defn}
\begin{defn}\label{defn:lax-morphism-into-CatAsOp-variants}
In the context of Definition \ref{defn:lax-morphism-into-CatAsOp} the functors $H_{\alpha}$ are called the \emph{products}, and the natural transformations $\xi_{\alpha,\rho}$, $\nu_i$ and $\sigma_{\alpha,(\beta)_j}$ are called the \emph{symmetries}, \emph{units} and \emph{substitutions} for $H$. When the units and substitutions are invertible, $H$ is said to be a \emph{pseudo morphism}, and when they are identities $H$ is said to be a \emph{strict morphism}. 
\end{defn}
Clearly the symmetries are isomorphisms. However even for a strict morphism, the symmetries need not be identities.
\begin{defn}\label{defn:commutative-morphism-into-CatAsOp}
In the context of Definitions \ref{defn:lax-morphism-into-CatAsOp} and \ref{defn:lax-morphism-into-CatAsOp-variants}, a lax morphism $H : T \to \CatAsOp$ is \emph{commutative} when its symmetries are identities.
\end{defn}
For any operad $T$ with object set $I$ and symmetric monoidal category $(\ca V,\tensor)$, it is standard to consider algebras of $T$ in $\ca V$. In explicit terms such an algebra consists of an object $H_i$ in $\ca V$ for each $i \in I$, morphisms $H_{\alpha} : \bigotimes_j H_{i_j} \to H_i$ for any $\alpha : (i_j)_j \to i$ in $T$, satisfying
\[ \begin{array}{lcccr} {H_{\alpha\rho} = H_{\alpha}c_{\rho}} &&
{H_{1_i} = 1_{H_i}} &&
{H_{\alpha(\beta_j)_j} = H_{\alpha}(\bigotimes_j H_{\beta_j})} \end{array} \]
for all $i \in I$, $\alpha : (i_j)_{1{\leq}j{\leq}n} \to i$ and $\beta_j : (i_{jk})_k \to i_j$ in $T$, and $\rho \in \Sigma_n$. Thus in particular one has
\begin{exam}\label{exam:Cat-algebras-as-commutative-strict-morphisms}
A commutative strict morphism $H : T \to \CatAsOp$ is the same thing as an algebra of the operad $T$ in $(\Cat,\times)$.
\end{exam}
On the other hand the most fundamental general class of examples which are rarely commutative is
\begin{exam}\label{exam:smoncat-as-pseudo-operad-morphism}
A symmetric monoidal category $(\ca V,\tensor)$ is the same thing as a pseudo morphism $\ca V : \Com \to \CatAsOp$. The unique object of $\Com$ is sent to the underlying category also denoted as $\ca V$, if $\alpha$ is the unique morphism of $\Com$ of arity $n$ then $\ca V_{\alpha}$ is the tensor product functor $\tensor : \ca V^n \to \ca V$, and the rest of the data of $\ca V : \Com \to \CatAsOp$ corresponds exactly to the coherence morphisms of $\ca V$. Lax morphisms $\Com \to \CatAsOp$ are also very well studied and are referred to either as \emph{symmetric lax monoidal categories} or \emph{functor operads} in the literature.
\end{exam}
This last example generalises in the following way.
\begin{exam}\label{exam:smoncat-as-pseudo-operad-morphism-general-version}
Let $\ca V$ be a symmetric monoidal category and $T$ be an operad. Then one has a pseudo morphism $\ca V_{\bullet} : T \to \CatAsOp$ in which $(\ca V_{\bullet})_i = \ca V$, $(\ca V_{\bullet})_{\alpha}$ where $\alpha$'s source is a sequence of length of $n$ is $\tensor : \ca V^{n} \to \ca V$, and the units, substitutions and symmetries are given by the coherences for $\ca V$'s symmetric monoidal structure. In particular when $\ca V$ is the terminal category, $\ca V_{\bullet}$ is denoted as $1 : T \to \CatAsOp$, which clearly corresponds via Example \ref{exam:Cat-algebras-as-commutative-strict-morphisms} to the terminal algebra in $(\Cat,\times)$ of the operad $T$.
\end{exam}
\begin{exam}\label{exam:lax-morphisms-for-categories}
As discussed in Remark \ref{rem:Cat-in-Opd} an operad $T$ with only unary arrows is a category. In this case a lax, pseudo or strict morphism $T \to \CatAsOp$ is the same thing as a lax functor, pseudo functor or functor $T \to \Cat$, in the usual senses, respectively. Moreover all such morphisms $T \to \CatAsOp$ are commutative.
\end{exam}
The various types of morphisms $H : T \to \CatAsOp$ defined above will now be exhibited as being structure on the underlying object of $\Cat/I$ whose fibre over $i \in I$ is the category $H_i$. We abuse notation and refer to this object of $\Cat/I$ also as $H$, or as $H \to I$.
\begin{thm}\label{thm:categorical-algebras-of-operads}
Let $T$ be an operad with object set $I$, regard it as a 2-monad on $\Cat/I$ as in Section \ref{sec:SMultiCats-Poly}, and let $H \to I$ be an object of $\Cat/I$. To give $H$ the structure of a lax, pseudo or strict $T$-algebra is to give it the structure of a lax, pseudo or strict morphism $H : T \to \CatAsOp$ respectively, in the sense defined above.
\end{thm}
\begin{proof}
Before proceeding to unpack what the action $a : TH \to H$ amounts to, we observe first that there is a canonical factorisation system{\footnotemark{\footnotetext{This is in a very strict algebraic sense, of being a decomposition of $TH$ as a composite monad in $\Span {\Set}$ via a distributive law, a situation that was studied in \cite{RosebrughWood-DistLawsFact}.}}} on $TH$. Recall that a general arrow of $TH$ is of the form $(\rho,(\gamma_j)_j) : (\alpha\rho,(x_j)_j) \to (\alpha,(y_j)_j)$ where $\alpha : (i_j)_{1{\leq}j{\leq}n} \to i$ is from $T$, $\rho \in \Sigma_n$ and $\gamma_j : x_j \to y_{\rho j}$ is from $H_{\rho j}$, and we say that $(\rho,(\gamma_j)_j)$ is \emph{levelwise} when $\rho$ is an identity, and \emph{permutative} when the $\gamma_j$ are all identities. Both these types of maps are closed under composition and contain the identities, a map is an identity iff it is both levelwise and permutative, and a general map $(\rho,(\gamma_j)_j)$ factors in a unique way
\[ (\alpha\rho,(x_j)_j) \xrightarrow{(1_n,(\gamma_j)_j)} (\alpha\rho,(y_{\rho j})_j) \xrightarrow{(\rho,(1_{y_{\rho j}})_j)} (\alpha,(y_j)_j) \]
as a levelwise map followed by a permutative map. Denoting the subcategory of $TH$ containing all the levelwise maps as $T_LH$, and the subcategory of $TH$ containing all the permutative maps as $T_RH$, to give a functor $f : TH \to C$ into any category $C$, is to give functors $f_L : T_LH \to C$ and $f_R : T_RH \to C$ which agree on objects, and with arrow maps compatible in the following sense -- if $(r : a \to b, l : b \to c)$ is a composable pair in $TH$ in which $l$ is levelwise and $r$ permutative, then $f_L(l)f_R(r) = f_R(r')f_L(l')$ in $C$, where $lr = r'l'$ is the levelwise-permutative factorisation of $lr$.

The object map of $a : TH \to H$ gives $a(\alpha,(y_j)_j) \in H_i$ for each $(\alpha,(y_j)_j)$ as above. Allowing the $(y_j)_j$ to vary, this amounts to the object map of a functor $H_{\alpha} : \prod_{j=1}^n H_{i_j} \to H_i$ for each $\alpha$. The arrow map of $a_L$ gives $H_{\alpha}(\gamma_j)_j : H_{\alpha}(x_j)_j \to H_{\alpha}(y_j)_j$ in $H_i$, for each $(\alpha,(y_j)_j)$ as above, and $\gamma_j : x_j \to y_j$ in $H_j$. Allowing the $(\gamma_j)_j$ to vary, we see that to give $a_L$ is to give functors $H_{\alpha}$ for all $\alpha$. The data of the arrow map of $a_R$ gives, for each $(\alpha,(y_j)_j)$ as above and $\rho \in \Sigma_n$, a morphism $H_{\alpha\rho}(y_{\rho j})_j \to H_{\alpha}(y_j)_j$, and allowing the $(y_j)_j$ to vary, this amounts to the components of a natural transformation $\xi_{\alpha,\rho} : H_{\alpha \rho} \to H_{\alpha}c_{\rho}$. In these terms the functoriality of $a_R$ amounts to the equations $\xi_{\alpha,1} = \id$ and $\xi_{\alpha,\rho_1\rho_2} = \xi_{\alpha,\rho_1}(\xi_{\alpha\rho_1,\rho_2}c_{\rho_1})$ and the compatibility of the arrow maps of $a_L$ and $a_R$ amounts to the naturality of the $\xi_{\alpha,\rho}$. Thus we have verified that to give $a : TH \to H$ is to give the products and symmetries of a lax morphism $H : T \to \CatAsOp$.

By the explicit description of the unit of the 2-monad $T$ given in Lemma \ref{lem:unit-of-operad}, to give $\overline{a}_0 : 1 \to a\eta^T$ is to give $\nu_i : 1_{H_i} \to H_{1_i}$, and $\overline{a}_0$ is invertible or an identity iff the $\nu_i$ are so. By the explicit description of the multiplication of the 2-monad $T$ given in Lemma \ref{lem:mult-of-operad}, one can similarly reconcile the data of $\overline{a}_2$ with that of the substitutions $\sigma_{\alpha,(\beta)_j}$. The naturality of $\overline{a}_2$ with respect to levelwise maps corresponds to the naturality of the $\sigma_{\alpha,(\beta)_j}$, the naturality of $\overline{a}_2$ with respect to permutative maps corresponds to the axiom of compatibility between subtitutions and symmetries, and the axioms involving just units and subtitutions amount to the lax algebra coherence axioms for $(\overline{a}_0,\overline{a}_2)$.
\end{proof}
By from Example \ref{exam:smoncat-as-pseudo-operad-morphism} and Theorem \ref{thm:categorical-algebras-of-operads} one recovers the fact that when $T = \Com$, pseudo $T$-algebras are exactly symmetric monoidal categories. The following example is an interesting variant of this.
\begin{exam}\label{exam:Ass-vs-M}
The single-coloured operad $\Ass$ for monoids has an $n$-ary operation for each permutation $\rho \in \Sigma_n$. By Theorem \ref{thm:categorical-algebras-of-operads} a strict algebra structure of the 2-monad $\Ass$ on $\ca V \in \Cat$, consists of an $n$-ary tensor product functor $\bigotimes_{\rho} : \ca V^n \to \ca V$ for each permutation $\rho \in \Sigma_n$, and for $\rho_1, \rho_2 \in \Sigma_n$, an isomorphism
\[ \xygraph{{\ca V^n}="p0" [r(2)] {\ca V^n}="p1" [dl] {\ca V}="p2" "p0":"p1"^-{c_{\rho_2}}:"p2"^-{\bigotimes_{\rho_1}}:@{<-}"p0"^-{\bigotimes_{\rho_1\rho_2}} "p0" [d(.5)r(.85)] :@{=>}[r(.3)]^-{\xi_{\rho_1,\rho_2}}} \]
in which $c_{\rho_2}$ permutes the factors according to $\rho_2$. This data must satisfy the axioms $\bigotimes_{1_1} = 1_{\ca V}$, $\bigotimes_{\rho}(\bigotimes_{\rho_j})_j = \bigotimes_{\rho(\rho_j)_j}$, $\xi_{\alpha,1} = \id$, $\xi_{\alpha,\rho_1\rho_2} = \xi_{\alpha,\rho_1}(\xi_{\alpha\rho_1,\rho_2}c_{\rho_1})$, and $\xi_{\rho}(\xi_{\rho_j})_j = \xi_{\rho(\rho_j)_j}$. When the $\xi$'s are identities, all the $n$-ary tensor products coincide giving a strict monoidal structure on $\ca V$. Thus a general strict $\Ass$-algebra structure on $\ca V$ is a fattened version of a monoidal structure.
\end{exam}
We now proceed to understand the algebra morphisms of the 2-monad associated to an operad. To this end we make
\begin{defn}\label{defn:lax-nat}
Let $T$ be an operad with object set $I$. Suppose that $H$ and $K$ are lax morphisms of operads $T \to \CatAsOp$. A \emph{lax-natural transformation}
\[ (f,\overline{f}) : H \longrightarrow K \]
consists of
\begin{itemize}
\item $\forall \, i \in I$, a functor $f_i : H_i \to K_i$.
\item $\forall \, \alpha : (i_j)_{1{\leq}j{\leq}n} \to i$ in $T$, a natural transformation
\[ \begin{array}{c} {\overline{f}_{\alpha} : K_{\alpha}(\prod_{j}f_{i_j}) \longrightarrow f_iH_{\alpha}} \end{array} \]
\end{itemize}
such that
\[ \xygraph{*!(0,-.1){\xybox{\xygraph{!{0;(2,0):(0,.5)::} {K_{\alpha\rho}c_{\rho}\prod_jf_{i_j}}="p0" [r] {f_iH_{\alpha\rho}c_{\rho}}="p1" [d] {f_iH_{\alpha}}="p2" [l] {K_{\alpha}\prod_jf_{i_j}}="p3" "p0":"p1"^-{\overline{f}\cdot\id}:"p2"^-{\id\cdot\xi^H}:@{<-}"p3"^-{\overline{f}}:@{<-}"p0"^-{\xi^K\cdot\id}}}}
[r(3.75)]
*{\xybox{\xygraph{!{0;(1,0):(0,1)::} {f_i}="p0" [dl] {K_{1_i}f_i}="p1" [r(2)] {f_iH_{1_i}}="p2" "p0":"p1"_-{\nu^K\cdot\id}:"p2"_-{\overline{f}}:@{<-}"p0"_-{\id\cdot\nu^H}}}}
[r(4.25)]
*{\xybox{\xygraph{!{0;(2.5,0):(0,.4)::} {K_{\alpha}\prod_jK_{\beta_j}\prod_{j,k}f_{i_{jk}}}="p0" [dr] {K_{\alpha(\beta_j)_j}}="p1" [d] {f_iH_{\alpha(\beta_j)_j}}="p2" [l] {f_iH_{\alpha}\prod_jH_{\beta_j}}="p3" [u] {K_{\alpha}\prod_jf_{i_j}H_{\beta_j}}="p4"
"p0":"p1"^(.6){\sigma^K\cdot\id}:"p2"^-{\overline{f}}:@{<-}"p3"^-{\id\cdot\sigma^H}:@{<-}"p4"^-{\overline{f}\cdot\id}:@{<-}"p0"^-{\id\cdot\prod_j\overline{f}}}}}} \]
commute. Modifying this definition by reversing the direction of the $\overline{f}_{\alpha}$ gives the definition of a \emph{colax-natural transformation}. When the $\overline{f}_i$ are invertible, $f$ is said to be \emph{pseudo-natural}, and when they are identites $f$ is just called a \emph{natural transformation}.
\end{defn}
\begin{exam}\label{exam:lax-nat-Cat}
If as in Example \ref{exam:lax-morphisms-for-categories} $T$ is a category, so that $H$ and $K$ can be regarded as lax functors $T \to \Cat$ in the usual sense, then Definition \ref{defn:lax-nat} in this case gives the usual notions of lax-natural, colax-natural, pseudo-natural and natural transformation $H \to K$.
\end{exam}
\begin{exam}\label{exam:general-operad-algebra}
Let $T$ be an operad and $\ca V$ be a symmetric monoidal category. Then an algebra of the operad $T$ in $\ca V$ (recalled just after Definition \ref{defn:commutative-morphism-into-CatAsOp}) is the same thing as a lax natural transformation $1 \to \ca V_{\bullet}$, where $1$ and $\ca V_{\bullet}$ are the morphisms $T \to \CatAsOp$ defined in Example \ref{exam:smoncat-as-pseudo-operad-morphism-general-version}.
\end{exam}
In the context of Definition \ref{defn:lax-nat} we write $f : H \to K$ for the morphism of $\Cat/I$ whose morphism between fibres over $i \in I$ is $f_i$. The notions just defined will now be seen as structure on $f$.
\begin{thm}\label{thm:catalg-morphisms-for-operads}
Let $T$ be an operad with object set $I$, regard it as a 2-monad on $\Cat/I$ as in Section \ref{sec:SMultiCats-Poly}, and let $f : H \to K$ be a morphism of $\Cat/I$. Suppose also that one has the structure of lax morphism $T \to \CatAsOp$ on both $H$ and $K$. Then to give $f$ the structure of lax, colax, pseudo or strict $T$-algebra morphism, is to give $f$ the structure of lax-natural, colax-natural, pseudo-natural or natural transformation respectively, with respect to the corresponding lax $T$-algebra structures on $H$ and $K$.
\end{thm}
\begin{proof}
Let us write $(a,\overline{a}_0,\overline{a}_2)$ for the action $a : TH \to H$ and coherence cells for the lax $T$-algebra structure on $H$ which corresponds to its lax morphism structure by Theorem \ref{thm:categorical-algebras-of-operads}, and similarly we write $(b,\overline{b}_0,\overline{b}_2)$ in the case of $K$. One has a component of the lax $T$-algebra coherence datum $\overline{f} : bT(f) \to fa$ for each object of $TH$, which is a pair $(\alpha,(y_j)_j)$ where $\alpha : (i_j)_{1{\leq}j{\leq}n} \to i$ is from $T$ and $y_j \in H_{i_j}$, and so such a component is a morphism $K_{\alpha}(fy_j)_j \to fH_{\alpha}(y_j)_j$. Thus the components of $\overline{f} : bT(f) \to fa$ are exactly the components of $\overline{f}_{\alpha}$ for all $\alpha$ for a corresponding lax-natural transformation structure on $f$, and $\overline{f}$ is an isomorphism or an identity iff the $\overline{f}_{\alpha}$ are so. The naturality of $\overline{f}$ with respect to the levelwise maps in $TH$ corresponds exactly to the naturality of the $\overline{f}_{\alpha}$ for all $\alpha$, and the naturality of $\overline{f}$ with respect to the permutative maps corresponds exactly to the axiom on the $\overline{f}_{\alpha}$'s relating to the symmetries of $H$ and $K$. The remaining axioms relating the $\overline{f}_{\alpha}$'s with the units and substitutions of $H$ and $K$ correspond to the lax algebra coherence axioms.
\end{proof}
\begin{rem}\label{rem:colax-algebras-duality}
Reversing the direction of the  units and substitutions in the definition of ``lax morphism'' of operads $H : T \to \CatAsOp$, one obtains the definition of a \emph{colax morphism} $H$. Similarly given a 2-monad $T$ on a 2-category $\ca K$, reversing the direction of the coherence 2-cells $\overline{a}_0$ and $\overline{a}_2$ in the definition of a ``lax $T$-algebra structure'' on $A \in \ca K$, one obtains the definition of \emph{colax $T$-algebra structure} on $A \in \ca K$. For an operad $T$ with object set $I$, to give $h : H \to I$ the structure of a colax morphism $T \to \CatAsOp$ is to give $h^{\op} : H^{\op} \to I^{\op} = I$ the structure of a lax morphism $T \to \CatAsOp$. To give $h$ the structure of a colax $T$-algebra is to give $h^{\op}$ the structure of a lax $T$-algebra. Thus one has versions of Theorems \ref{thm:categorical-algebras-of-operads} and \ref{thm:catalg-morphisms-for-operads} characterising colax $T$-algebras and various kinds of morphisms between them.
\end{rem}
Finally we characterise the algebra 2-cells for the 2-monad on $\Cat/I$ associated to an operad $T$ with object set $I$.
\begin{defn}\label{defn:modification}
Let $T$ be an operad with object set $I$. Suppose that $H$ and $K$ are lax morphisms of operads $T \to \CatAsOp$, and that $(f,\overline{f})$ and $(g,\overline{g}) : H \to K$ are lax-natural transformations. Then a \emph{modification} $\psi : (f,\overline{f}) \to (g,\overline{g})$ consists of natural transformations $\psi_i : f_i \to g_i$ for all $i \in I$, such that
\[ \xygraph{!{0;(2,0):(0,.5)::} {K_{\alpha}(\prod_{j}f_{i_j})}="p0" [r] {f_iH_{\alpha}}="p1" [d] {g_iH_{\alpha}}="p2" [l] {K_{\alpha}(\prod_{j}g_{i_j})}="p3" "p0":"p1"^-{\overline{f}_{\alpha}}:"p2"^-{\psi_i\cdot\id}:@{<-}"p3"^-{\overline{g}_{\alpha}}:@{<-}"p0"^-{\id\cdot\prod_j\psi_{i_j}}} \]
commutes for all $\alpha : (i_j)_{1{\leq}j{\leq}n} \to i$ in $T$.
\end{defn}
\begin{exam}\label{exam:morphism-of-general-operad-algebra}
Let $T$ be an operad, $\ca V$ be a symmetric monoidal category, and $A$ and $B$ be algebras of $T$ in $\ca V$. Recall from Example \ref{exam:general-operad-algebra} that $A$ and $B$ may be regarded as lax natural transformations $1 \to \ca V_{\bullet}$. To give a morphism $A \to B$ of $T$-algebras in $\ca V$, is to give a modification between the corresponding lax natural transformations.
\end{exam}
In the evident way by reversing the directions of the appropriate coherence cells, one defines modifications between colax natural transformations, and algebra 2-cells between colax morphisms of $T$-algebras. It is straightforward to verify that the modifications of Definition \ref{defn:modification} match up with the algebra 2-cells of $T$ as follows.
\begin{thm}\label{thm:catalg-2cells-for-operads}
Let $T$ be an operad with object set $I$, regard it as a 2-monad on $\Cat/I$. Suppose that one has lax morphisms of operads $H$ and $K : T \to \CatAsOp$, and lax-natural (resp. colax-natural) transformations $(f,\overline{f})$ and $(g,\overline{g}) : H \to K$. Then to give a modification $(f,\overline{f}) \to (g,\overline{g})$ is to give an algebra 2-cell between the corresponding lax (resp. colax) morphisms of $T$-algebras.
\end{thm}
While the algebras of the 2-monad associated to an operad $T$ are different to the algebras of the operad in the usual sense, one does recover the algebras of $T$ in any symmetric monoidal category $\ca V$, from Theorems \ref{thm:categorical-algebras-of-operads}, \ref{thm:catalg-morphisms-for-operads} and \ref{thm:catalg-2cells-for-operads} and Examples \ref{exam:smoncat-as-pseudo-operad-morphism-general-version}, \ref{exam:general-operad-algebra} and \ref{exam:morphism-of-general-operad-algebra}. We record this in
\begin{cor}\label{cor:general-operad-algebra}
For any operad $T$ and any symmetric monoidal category $\ca V$, the category of $T$-algebras in $\ca V$ and morphisms thereof is isomorphic to $\PsAlgl T(1,\ca V_{\bullet})$.
\end{cor}
\begin{exams}\label{exams:modular-operads-etc}
In \cite{BataninBerger-HtyThyOfAlgOfPolyMnd} many contemporary operad notions, were seen as algebras of $\Sigma$-free operads, by exhibiting the corresponding polynomial over $\Set$ whose algebras are the corresponding contemporary operad notion. For instance there is a $\Sigma$-free operad $T$ corresponding to modular operads, and so for each symmetric monoidal category $\ca V$, $\PsAlgl T(1,\ca V_{\bullet})$ is the category modular operads in $\ca V$.
\end{exams}

\section{Commutative operad morphisms into \underline{Cat}}
\label{sec:comm-algebra}

In this section we exhibit a 2-monad $T/\Sigma$ on $\Cat/I$, whose lax, colax, pseudo and strict algebras are exactly the commutative lax, colax, pseudo and strict morphisms $T \to \CatAsOp$, in the sense of Definition \ref{defn:commutative-morphism-into-CatAsOp}. So it is the strict-$(T/\Sigma)$-algebras that correspond to ordinary morphisms of operads $T \to \CatAsOp$. Since these in turn coincide with the algebras for the operad $T$ in $\Cat$ in the usual sense, $T/\Sigma$ is thus the usual monad that one associates to an operad. The novelty here is that since we obtain $T/\Sigma$ from $T$ via some general 2-categorical process, the comparison of these two 2-monads is thus facilitated. We shall exploit this in Sections \ref{sec:sigma-free} and \ref{sec:QuillenEquiv}.

As anticipated in the introduction, we will obtain $T/\Sigma$ from $T$ by exhibiting a canonical 2-cell $\alpha_T$
\begin{equation}\label{eq:def-T-Mod-Sigma}
\begin{gathered}
\xygraph{!{0;(1.5,0):(0,1)::} {T^{[1]}_{\Sigma}}="p0" [r] {T}="p1" [r] {T/\Sigma}="p2" "p0":@<1.5ex>"p1"|(.45){}="t"^-{d_T}:"p2"^-{q_T} "p0":@<-1.5ex>"p1"|(.45){}="b"_-{c_T} "t":@{}"b"|(.15){}="d"|(.85){}="c" "d":@{=>}"c"^-{\alpha_T}}
\end{gathered}
\end{equation}
in the 2-category $\tn{Mnd}(\Cat/I)$ of 2-monads on $\Cat/I$ below in Definition \ref{defn:T1-Sigma}, and then taking the universal 1-cell $q_T$ such that $q_T\alpha_T$ is an identity. In the language of 2-category theory, $q_T$ is the \emph{coidentifier} of $\alpha_T$ in $\tn{Mnd}(\Cat/I)$. From this abstract description we will see that $T/\Sigma$ has the required algebras. Recall from the proof of Theorem \ref{thm:categorical-algebras-of-operads} that for $H \to I$ in $\Cat/I$, $TH$ has a factorisation system in which the left class are the levelwise maps and the right class are the permutative maps. The universal property of $q_T$ means that it is the universal monad morphism which on components sends these permutative maps to identities. In other words at the syntactic level, the monad morphism $q_T$ is the process of modding out by the symmetric group actions of the operad $T$.

Before providing $T^{[1]}_{\Sigma}$ and $\alpha_T$ for (\ref{eq:def-T-Mod-Sigma}) some preliminary remarks are in order. To begin with, it turns out that for us, the 2-functor $(-)^{[1]} : \Cat \to \Cat$ which takes any category to its category of arrows, preserves enough distributivity pullbacks. A general result in this direction is Lemma \ref{lem:radj-pres-dpbs} below. Recall that a natural transformation between functors $\ca A \to \ca B$ is \emph{cartesian} when its naturality squares in $\ca B$ are pullbacks. Given an arrow $f$ of $\ca A$, we say that the natural transformation is \emph{cartesian with respect to} $f$, when its naturality square associated to $f$ is a pullback square. In particular, a cartesian natural tranformation in the usual sense is one that is cartesian with respect to all the morphisms of $\ca A$.
\begin{lem}\label{lem:radj-pres-dpbs}
A functor $R : \ca E \to \ca F$ between categories with pullbacks which has a pullback preserving left adjoint $L$, preserves any distributivity pullbacks of the form
\[ \xygraph{{P}="p0" [r] {A}="p1" [r] {B}="p2" [d] {C}="p3" [l(2)] {Q}="p4" "p0":"p1"^-{p}:"p2"^-{g}:"p3"^-{f}:@{<-}"p4"^-{r}:@{<-}"p0"^-{q}:@{}"p3"|-{\tn{dpb}}} \]
for which the counit of $L \ladj R$ is cartesian with respect to $f$.
\end{lem}
\begin{proof}
We denote by $\varepsilon : LR \to 1_{\ca E}$ the counit of the adjunction $L \ladj R$. Given a pullback $(p_2,q_2,r_2)$ around $(Rf,Rg)$ as on the left
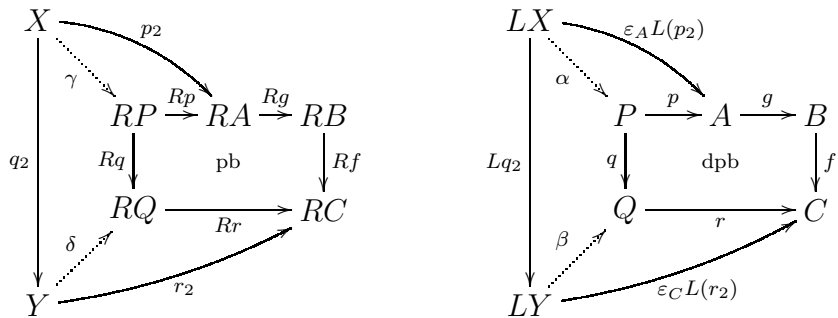
\[ \xygraph{{\xybox{\xygraph{{RP}="p0" [r] {RA}="p1" [r] {RB}="p2" [d] {RC}="p3" [l(2)] {RQ}="p4" "p0":"p1"^-{Rp}:"p2"^-{Rg}:"p3"^-{Rf}:@{<-}"p4"^-{Rr}:@{<-}"p0"^-{Rq}:@{}"p3"|-{\tn{pb}}
"p0" [ul] {X}="q0" "p4" [dl] {Y}="q1" "q0" (:@/^{.75pc}/"p1"^-{p_2},:@{.>}"p0"_-{\gamma},:"q1"_-{q_2}(:@{.>}"p4"^-{\delta},:@/_{.5pc}/"p3"_-{r_2}))}}}
[r(5)]
{\xybox{\xygraph{{P}="p0" [r] {A}="p1" [r] {B}="p2" [d] {C}="p3" [l(2)] {Q}="p4" "p0":"p1"^-{p}:"p2"^-{g}:"p3"^-{f}:@{<-}"p4"^-{r}:@{<-}"p0"^-{q}:@{}"p3"|-{\tn{dpb}}
"p0" [ul] {LX}="q0" "p4" [dl] {LY}="q1" "q0" (:@/^{.75pc}/"p1"^-{\varepsilon_AL(p_2)},:@{.>}"p0"_-{\alpha},:"q1"_-{Lq_2}(:@{.>}"p4"^-{\beta},:@/_{.5pc}/"p3"_-{\varepsilon_CL(r_2)}))}}}} \]
we must exhibit $(\gamma,\delta)$ as shown in $\ca F$ unique making the diagram commute. To this end one has the solid parts of the diagram on the right in $\ca E$ in the previous display, and the square with vertices $(LX,B,C,LY)$ can be decomposed as
\[ \xygraph{!{0;(1.5,0):(0,.6667)::} {LX}="p0" [r] {LRA}="p1" [r] {LRB}="p2" [r] {B}="p3" [d] {C}="p4" [l] {LRC}="p5" [l(2)] {LY}="p6" "p0":"p1"^-{Lp_2}:"p2"^-{LRg}:"p3"^-{\varepsilon_B}:"p4"^-{f}:@{<-}"p5"^-{\varepsilon_C}:@{<-}"p6"^-{Lr_2}:@{<-}"p0"^-{Lq_2} "p2":"p5"_-{LRf}} \]
and so is a pullback by the hypotheses on $L$ and $f$. Thus one has $(\alpha,\beta)$ unique as shown making the diagram on the right commute, but to give such $(\alpha,\beta)$ is, by the adjointness $L \ladj R$, to give the required $(\gamma,\delta)$.
\end{proof}
%
% Maybe add the example of $\tn{ob} : \Cat \to \Set$.
% In the context of this lemma, if the unit is cartesian wrt q then R reflects the dpb.
% When R is the forgetful $\Cat \to \Gph$, the counit is cartesian at f iff f is a discrete Conduch\'{e} fibration. In this case it seems that we can recover the exponentiability of discrete Conduch\'{e} fibrations from such considerations. If so, can we push further and recover the characterisation of exponentiable functors from the category monad on $\Gph$?
%
\begin{rem}\label{rem:radj-pres-2dpbs}
To obtain the 2-categorical analogue of Lemma \ref{lem:radj-pres-dpbs}, in which $L \ladj R$ is a 2-adjunction, one uses the 2-dimensional universal property of distributivity pullbacks in this setting and the 2-dimensional aspects of the adjointness $L \ladj R$, to adapt the above proof.
\end{rem}
\begin{lem}\label{lem:arrowcat-dpbs}
If the pullback on the left
\[ \xygraph{{\xybox{\xygraph{{P}="p0" [r] {A}="p1" [r] {B}="p2" [d] {C}="p3" [l(2)] {Q}="p4" "p0":"p1"^-{p}:"p2"^-{g}:"p3"^-{f}:@{<-}"p4"^-{r}:@{<-}"p0"^-{q}:@{}"p3"|-{\tn{dpb}}}}}
[r(4)]
{\xybox{\xygraph{{P^{[1]}}="p0" [r] {A^{[1]}}="p1" [r] {B^{[1]}}="p2" [d] {C^{[1]}}="p3" [l(2)] {Q^{[1]}}="p4" "p0":"p1"^-{p^{[1]}}:"p2"^-{g^{[1]}}:"p3"^-{f^{[1]}}:@{<-}"p4"^-{r^{[1]}}:@{<-}"p0"^-{q^{[1]}}:@{}"p3"|-{\tn{pb}}}}}} \]
in $\Cat$ is a distributivity pullback around $(f,g)$ and $f$ is a discrete fibration and a discrete opfibration, then the pullback on the right in the previous display is a  distributivity pullback around $(f^{[1]},g^{[1]})$.
\end{lem}
\begin{proof}
The functor $(-) \times [1] : \Cat \to \Cat$ preserves pullbacks since it can be written as the composite
\[ \Cat \xrightarrow{\Delta_{[1]}} \Cat/[1] \xrightarrow{\Sigma_{[1]}} \Cat. \]
The component of the counit of $(-) \times [1] \ladj (-)^{[1]}$ at $X \in \Cat$ is given by the evaluation functor $\tn{ev}_{[1],X} : X^{[1]} \times [1] \to X$. Thus
by Lemma \ref{lem:radj-pres-dpbs}, it suffices to show that if $f : B \to C$ is a discrete fibration and a discrete opfibration, then the naturality square
\begin{equation}\label{eq:ev-nat-sq}
\begin{gathered}
\xygraph{!{0;(2,0):(0,.5)::} {B^{[1]} \times [1]}="p0" [r] {B}="p1" [d] {C}="p2" [l] {C^{[1]} \times [1]}="p3" "p0":"p1"^-{\tn{ev}_{[1],B}}:"p2"^-{f}:@{<-}"p3"^-{\tn{ev}_{[1],C}}:@{<-}"p0"^-{f^{[1]} \times 1_{[1]}}}
\end{gathered}
\end{equation}
is a pullback. To say this is a pullback on objects is to say that for $\alpha : c_0 \to c_1$ in $C$, $i \in \{0,1\}$ and $b \in B$ such that $fb = c_i$, then there exists a unique $\beta : b_0 \to b_1$ of $B$ such that $f\beta = \alpha$ and $b_i = b$. When $i = 0$, this is exactly the condition that $f$ be a discrete opfibration, and when $i = 1$ this corresponds to $f$ being a discrete fibration. In other words $f$ is a discrete fibration and a discrete opfibration iff (\ref{eq:ev-nat-sq}) is a pullback on objects.

An arrow $S$ of $C^{[1]}$ is a commutative square in the category $C$, and the category $[1]$ is just the ordinal $\{0<1\}$, and so it has three arrows $1_0$, $1_1$ and the unique arrow $0 \to 1$. Writing
\[ \xygraph{{c_0}="p0" [r] {c_1}="p1" [d] {c_3}="p2" [l] {c_2}="p3" "p0":"p1"^-{\alpha_0}:"p2"^-{\alpha_3}:@{<-}"p3"^-{\alpha_1}:@{<-}"p0"^-{\alpha_2}} \]
for the arrow $S$, regarding it as an arrow $\alpha_0 \to \alpha_1$ in $C^{[1]}$,
the evaluation functor $\tn{ev}_{[1],C}$ acts on arrows by sending $(S,1_0)$, $(S,1_1)$ and $(S,0 \to 1)$ to $\alpha_2$, $\alpha_3$ and the diagonal $c_0 \to c_3$ respectively. To say that (\ref{eq:ev-nat-sq}) is a pullback on arrows is to say that given a square $S$ in $C$, an arrow $\iota$ in $[1]$ and an arrow $\beta : a \to b$ in $B$ such that $f\beta = \tn{ev}_{[1],C}(S,\iota)$, then there exists a unique square $R$ in $B$ such that $fR = S$ and $\tn{ev}_{[1],B}(B,\iota) = \beta$. In more elementary terms this says, in the cases where $\iota$ is $1_0$, $0 \to 1$ and $1_1$, that $f$ enjoys the respective unique lifting properties depicted in
\[ \xygraph{!{0;(3,0):(0,.75)::} 
{\xybox{\xygraph{{c'_0}="p0" [r] {a}="p1" [d] {b}="p2" [l] {c'_2}="p3" "p0":@{.>}"p1"^-{}:"p2"^-{\beta}:@{<.}"p3"^-{}:@{<.}"p0"^-{}}}}="p0"
[d]
{\xybox{\xygraph{{c_0}="p0" [r] {fa}="p1" [d] {fb}="p2" [l] {c_2}="p3" "p0":"p1"^-{}:"p2"^-{f\beta}:@{<-}"p3"^-{}:@{<-}"p0"^-{}}}}="p1"
[ul]
{\xybox{\xygraph{{a}="p0" [r] {c'_1}="p1" [d] {b}="p2" [l] {c'_2}="p3" "p0":@{.>}"p1"^-{}:@{.>}"p2"^-{}:@{<.}"p3"^-{}:@{<.}"p0"^-{}:"p2"^-{\beta}}}}="p2"
[d]
{\xybox{\xygraph{{fa}="p0" [r] {c_1}="p1" [d] {fb}="p2" [l] {c_2}="p3" "p0":"p1"^-{}:"p2"^-{}:@{<-}"p3"^-{}:@{<-}"p0"^-{}:"p2"^-{f\beta}}}}="p3"
[ul]
{\xybox{\xygraph{{a}="p0" [r] {c'_1}="p1" [d] {c'_3}="p2" [l] {b}="p3" "p0":@{.>}"p1"^-{}:@{.>}"p2"^-{}:@{<.}"p3"^-{}:@{<-}"p0"^-{\beta}}}}="p4"
[d]
{\xybox{\xygraph{{fa}="p0" [r] {c_1}="p1" [d] {c_3}="p2" [l] {fb}="p3" "p0":"p1"^-{}:"p2"^-{}:@{<-}"p3"^-{}:@{<-}"p0"^-{f\beta}}}}="p5"} \]
in which dotted arrows in $B$ are the unique liftings of the corresponding arrows down in $C$. For instance, in the case $\iota = 1_0$ depicted on the left, this says that given $\beta$ in $B$ and a square in $C$ whose left vertical edge is $f\beta$, one can lift that square uniquely to a square in $B$ whose left vertical edge is $\beta$. Clearly one has this unique lifting property when $f$ is a discrete opfibration, and similarly one has the unique lifting property depicted on the right ($\iota = 1_1$) when $f$ is a discrete fibration. As for the unique lifting property depicted in the middle ($\iota = 0 \to 1$), one has this whenever $f$ is a discrete Conduch\'{e} fibration, that is, when $f$ satisfies the unique lifting of factorisations. As is well-known and easy to check directly, both discrete fibrations and discrete opfibrations possess this property.
\end{proof}
For $X \in \Cat$ the arrow category $X^{[1]}$ is a basic 2-categorical limit construction, namely it is the \emph{cotensor} of $X$ with the category $[1]$. As such the data of the corresponding limit cone is
\[ \xygraph{!{0;(1.5,0):(0,1)::} {X^{[1]}}="p0" [r] {X}="p1"
"p0":@<1.5ex>"p1"|(.45){}="t"^-{d_X} "p0":@<-1.5ex>"p1"|(.45){}="b"_-{c_X} "t":@{}"b"|(.15){}="d"|(.85){}="c" "d":@{=>}"c"^-{\alpha_X}} \]
in which $d_X$ and $c_X$ are the functors which on objects take domains and codomains respectively, and the component of $\alpha_X$ at $\beta \in X^{[1]}$ is $\beta$ viewed as an arrow of $X$. By 2-dimensional universality $d_X$ and $c_X$ are the components of 2-natural transformations $d$ and $c : (-)^{[1]} \to 1_{\Cat}$, and the $\alpha_X$ are the components of a modification $\alpha : d \to c$. Moreover $d$ is cartesian with respect to a functor $f : B \to C$ iff $f$ is a discrete opfibration, and $c$ is cartesian with respect to $f$ iff $f$ is a discrete fibration.

Let $T$ be a collection with object set $I$. Recall that the middle map $p_T : E_T \to B_T$ of its corresponding polynomial is a discrete fibration and a discrete opfibration. Since $(-)^{[1]}$ as a right adjoint preserves discrete fibrations and discrete opfibrations, $p_T^{[1]}$ is such and hence exponentiable. Thus applying $(-)^{[1]}$ componentwise to $T$'s corresponding polynomial and identifying $I = I^{[1]}$, gives another polynomial as on the left
\begin{equation}\label{eq:T1-Sigma-poly}
\begin{gathered}
\xygraph{{\xybox{\xygraph{!{0;(1.25,0):(0,1)::} {I}="p0" [r] {E_T^{[1]}}="p1" [r] {B_T^{[1]}}="p2" [r] {I}="p3" "p0":@{<-}"p1"^-{s_T^{[1]}}:"p2"^-{p_T^{[1]}}:"p3"^-{t_T^{[1]}}}}}
[r(5.5)u(.5)]
{\xybox{\xygraph{!{0;(1.5,0):(0,.6)::} {I}="p0" [ur] {E^{[1]}_T}="p1" [r] {B^{[1]}_T}="p2" [dr] {I}="p3" [dl] {B_T}="p4" [l] {E_T}="p5" "p0":@{<-}"p1"^-{s^{[1]}_T}:"p2"^-{p^{[1]}_T}:"p3"^-{t^{[1]}_T}:@{<-}"p4"^-{t_T}:@{<-}"p5"^-{p_T}:"p0"^-{s_T}
"p1":@/_{1.25pc}/"p5"_(.425){d_{E_T}}|(.425){}="dal" "p2":@/_{1.25pc}/"p4"_(.425){d_{B_T}}|(.425){}="dar"
"p1":@/^{1.25pc}/"p5"^(.575){c_{E_T}}|(.575){}="cal" "p2":@/^{1.25pc}/"p4"^(.575){c_{E_T}}|(.575){}="car"
"dal":@{}"cal"|(.35){}="dl"|(.65){}="cl" "dl":@{=>}"cl"^-{\alpha_{E_T}}
"dar":@{}"car"|(.35){}="dr"|(.65){}="cr" "dr":@{=>}"cr"^-{\alpha_{B_T}}}}}}
\end{gathered}
\end{equation}
and since $d$ and $c$ are cartesian with respect to $p_T$, one in fact has a 3-cell $(\alpha_{B_T},\alpha_{E_T})$ in $\Polyc{\Cat}$ as on the right.
\begin{defn}\label{defn:dfib-dopfib-polyendos}
The full sub-2-category of $\Polyc{\Cat}(I,I)$ consisting of those polynomials whose middle map is a discrete fibration and a discrete opfibration is denoted as $\mathfrak{D}_I$.
\end{defn}
Since discrete fibrations and discrete opfibrations are pullback stable and closed under composition, the monoidal structure $\Polyc{\Cat}(I,I)$ has by virtue of the composition of polynomials, restricts to $\mathfrak{D}_I$. Applying $(-)^{[1]}$ componentwise to such polynomials as above is the effect on objects of an endo-2-functor
\[ \ca I_I : \mathfrak{D}_I \longrightarrow \mathfrak{D}_I. \]
The components of $d$ and $c$ give 2-natural transformations $\ca D_I$ and $\ca C_I : \ca I_I \to 1_{\mathfrak{D}_I}$, and the components of $\alpha$ give a modification $\ca A_I : \ca D_I \to \ca C_I$. The component of $\ca A_I$ with respect to the polynomial corresponding to a collection $T$ is $(\alpha_{B_T},\alpha_{E_T})$ depicted in (\ref{eq:T1-Sigma-poly}).

As is well-known, for a symmetric monoidal closed category $\ca V$, the basic notions of monoidal category, lax and strong monoidal functor and monoidal natural transformation admit evident $\ca V$-enriched analogues. In particular the endohoms of $\Polyc{\Cat}$, or indeed those of any 2-bicategory, and the 2-categories $\mathfrak{D}_I$ defined above are monoidal 2-categories in the straightforward $\Cat$-enriched sense. Given monoidal 2-categories $\ca X$ and $\ca Y$, lax monoidal 2-functors $F$ and $G : \ca X \to \ca Y$, and monoidal 2-natural transformations $\phi$ and $\psi : F \to G$, one has an evident notion of \emph{monoidal modification} $\zeta : \phi \to \psi$, which is a modification such that
\[ \xygraph{{\xybox{\xygraph{{I}="p0" [r(1.5)u] {FI}="p1" [d(2)] {GI}="p2" "p0":"p1"^-{\overline{F}_0}:@/_{1pc}/"p2"_(.4){\phi_I}|(.4){}="ml":@{<-}"p0"^-{\overline{G}_0} "p1":@/^{1pc}/"p2"^(.6){\psi_I}|(.6){}="mr"
"ml":@{}"mr"|(.35){}="d"|(.65){}="c" "d":@{=>}"c"^-{\zeta_{I}}}}}
[r(4.5)]
{\xybox{\xygraph{!{0;(3,0):(0,.6667)::} {FX \tensor FY}="p0" [r] {F(X \tensor Y)}="p1" [d] {G(X \tensor Y)}="p2" [l] {GX \tensor GY}="p3" "p0":"p1"^-{\overline{F}_{2,X,Y}}:@/_{1.75pc}/"p2"_(.4){\phi_{X \tensor Y}}|(.4){}="mlr":@{<-}"p3"^-{\overline{G}_{2,X,Y}}:@{<-}@/^{1.75pc}/"p0"^(.6){\phi_X \tensor \phi_Y}|(.6){}="mll" "p1":@/^{1.75pc}/"p2"^(.6){\psi_{X \tensor Y}}|(.6){}="mrr" "p0":@/^{1.75pc}/"p3"^(.6){\psi_X \tensor \psi_Y}|(.6){}="mrl"
"mll":@{}"mrl"|(.35){}="dl"|(.65){}="cl" "dl":@{=>}"cl"^-{\zeta_X \tensor \zeta_Y}
"mlr":@{}"mrr"|(.35){}="dr"|(.65){}="cr" "dr":@{=>}"cr"^-{\zeta_{X \tensor Y}}}}}} \]
commute, where $\overline{F}_0$, $\overline{G}_0$, $\overline{G}_{2,X,Y}$ and $\overline{G}_{2,X,Y}$ are the lax monoidal coherences of $F$ and $G$. 

With these 3-cells monoidal 2-categories form a (strict) 3-category $\tnb{Mon-}{\mathbf 2}{\tnb{-CAT}}$ equipped with a forgetful 3-functor into $\TwoCAT$. As in the unenriched setting monoidal 2-functors $1 \to \ca X$ may be identified with (strict) monoids in $\ca X$. We denote by $\tn{Mon}(\ca X)$ the 2-category $\tnb{Mon-}{\mathbf 2}{\tnb{-CAT}}(1,\ca X)$, whose objects are monoids and monoid morphisms, and whose 2-cells we refer to as \emph{monoid 2-cells}. Also of interest for us are weaker notions of monoid morphism, so we denote by $\tn{Mon}(\ca X)_{\tn{l}}$, $\tn{Mon}(\ca X)_{\tn{c}}$ and $\tn{Mon}(\ca X)_{\tn{ps}}$ the 2-categories whose objects are monoids in $\ca X$, and morphisms are lax, colax and pseudo morphisms of monoids whose coherence data is of the form
\[ \xygraph{{\xybox{\xygraph{{I}="p0" [r] {X}="p1" [r] {X^{\tensor 2}}="p2" [d] {Y^{\tensor 2}}="p3" [l] {Y}="p4"
"p0":"p1"^-{u_X}:@{<-}"p2"^-{m_X}:"p3"^-{f^{\tensor 2}}:"p4"^-{m_Y}:@/^{1pc}/@{<-}"p0"^-{u_Y} "p1":"p4"^-{f}
"p0" [d(.4)r(.4)] :@{=>}[r(.3)] "p1" [d(.4)r(.4)] :@{<=}[r(.3)]}}}
[r(3.25)]
{\xybox{\xygraph{{I}="p0" [r] {X}="p1" [r] {X^{\tensor 2}}="p2" [d] {Y^{\tensor 2}}="p3" [l] {Y}="p4"
"p0":"p1"^-{u_X}:@{<-}"p2"^-{m_X}:"p3"^-{f^{\tensor 2}}:"p4"^-{m_Y}:@/^{1pc}/@{<-}"p0"^-{u_Y} "p1":"p4"^-{f}
"p0" [d(.4)r(.4)] :@{<=}[r(.3)] "p1" [d(.4)r(.4)] :@{=>}[r(.3)]}}}
[r(3.25)]
{\xybox{\xygraph{{I}="p0" [r] {X}="p1" [r] {X^{\tensor 2}}="p2" [d] {Y^{\tensor 2}}="p3" [l] {Y}="p4"
"p0":"p1"^-{u_X}:@{<-}"p2"^-{m_X}:"p3"^-{f^{\tensor 2}}:"p4"^-{m_Y}:@/^{1pc}/@{<-}"p0"^-{u_Y} "p1":"p4"^-{f}
"p0" [d(.4)r(.55)] *{\iso} "p1" [d(.4)r(.55)] *{\iso}}}}} \]
respectively satisfying the usual axioms, and monoid 2-cells between these. In the case where $\ca X = \Cat$ with cartesian tensor product, one refinds the usual notions of lax, colax and strong monoidal functor between strict monoidal categories, and monoidal natural transformations between these.

In particular, for any 2-category $\ca K$ one has the monoidal 2-category $\tn{End}(\ca K)$ of endo-2-functors of $\ca K$. For $\ca X = \tn{End}(\ca K)$ the 2-categories $\tn{Mon}(\ca X)$, $\tn{Mon}(\ca X)_{\tn{l}}$, $\tn{Mon}(\ca X)_{\tn{c}}$ and $\tn{Mon}(\ca X)_{\tn{ps}}$ just defined are denoted $\tn{Mnd}(\ca K)$, $\tn{Mnd}(\ca K)_{\tn{l}}$, $\tn{Mnd}(\ca K)_{\tn{c}}$ and $\tn{Mnd}(\ca K)_{\tn{ps}}$ respectively. The 2-cells in any of these 2-categories will be called \emph{2-monad 2-cells}.

For any monoid $M$ in a monoidal 2-category $\ca X$ and monoidal modification $\zeta$ as in
\[ \xygraph{*=(0,2){\xybox{\xygraph{!{0;(2,0):(0,1)::} {1}="p0" [r] {\ca X}="p1" [r] {\ca Y}="p2" "p0":"p1"^-{M}
"p1":@/^{1.5pc}/"p2"^-{F} "p1":@/_{1.5pc}/"p2"_-{G}
"p1" [r(.45)u(.175)] :@/_{.85pc}/[d(.35)]_-{\phi}
"p1" [r(.55)u(.175)] :@/^{.85pc}/[d(.35)]^-{\psi}
"p1" [r(.4)] :[r(.2)]^-{\zeta}}}}} \]
it follows immediately from the 3-category structure of $\tnb{Mon-}{\mathbf 2}{\tnb{-CAT}}$ that $FM$ and $GM$ are monoids in $\ca Y$, the components $\phi_M$ and $\psi_M$ are monoid morphisms, and $\zeta_M$ is a monoid 2-cell.
\begin{prop}\label{prop:mon-mod-poly}
In the context of Definition \ref{defn:dfib-dopfib-polyendos}, $\ca I_I : \mathfrak{D}_I \to \mathfrak{D}_I$ is a strong monoidal 2-functor, $\ca D_I$ and $\ca C_I : \ca I_I \to 1_{\mathfrak{D}_I}$ are monoidal 2-natural tranformations, and $\ca A_I : \ca D_I \to \ca C_I$ is a monoidal modification.
\end{prop}
\begin{proof}
Since $I^{[1]} = I$ one has $\ca I_I(1_I) = 1_I$. Given $P$ and $Q \in \mathfrak{D}_I$, since $(-)^{[1]}$ preserves pullbacks and distributivity pullbacks by Lemma \ref{lem:arrowcat-dpbs}, one has a coherence isomorphism $\ca I_I(P) \comp \ca I_I(Q) \iso \ca I_I(P \comp Q)$ induced by the universal properties of the pullbacks and distributivity pullbacks involved in the formation of $\ca I_I(P) \comp \ca I_I(Q)$. The monoidal 2-functor coherence axioms follow because of the uniqueness inherent in how these coherence isomorphisms were induced. This uniqueness, together with the naturality of $d_X : X^{[1]} \to X$, $c_X : X^{[1]} \to X$ and $\alpha_X : d_X \to c_X$ in $X \in \Cat$, enables one to establish that $\ca D_I$ and $\ca C_I$ are monoidal 2-natural tranformations and $\ca A_I$ is a monoidal modification.
\end{proof}
For an operad $T$ with object set $I$ this last result ensures that $(\alpha_{B_T},\alpha_{E_T})$ depicted in (\ref{eq:T1-Sigma-poly})
underlies a monoid 2-cell in $\mathfrak{D}_I$. Since the effect on homs of $\PFun{\Cat}$ gives a strong monoidal 2-functor $\mathfrak{D}_I \to \tn{End}(\Cat/I)$, $\PFun{\Cat}$ carries $(\alpha_{B_T},\alpha_{E_T})$ to a 2-cell in $\tn{Mnd}(\Cat/I)$.
\begin{defn}\label{defn:T1-Sigma}
The 2-monad 2-cell just described is denoted
\[ \xygraph{!{0;(1.5,0):(0,1)::} {T^{[1]}_{\Sigma}}="p0" [r] {T.}="p1" "p0":@<1.5ex>"p1"|(.45){}="t"^-{d_T} "p0":@<-1.5ex>"p1"|(.45){}="b"_-{c_T} "t":@{}"b"|(.15){}="d"|(.85){}="c" "d":@{=>}"c"^-{\alpha_T}} \]
\end{defn}
We now turn to giving an explicit description of the components of $\alpha_T$. Recall from Lemma \ref{lem:endofunctor-from-collection} that for $X \in \Cat/I$, an arrow of $TX$ is of the form $(\rho,(\gamma_j)_j) : (\alpha\rho,(x_j)_j) \to (\alpha,(y_j)_j)$ where $\alpha : (i_j)_{1{\leq}j{\leq}n} \to i$ is from $T$, $\rho \in \Sigma_n$ and $\gamma_j : x_j \to y_{\rho j}$ is from $X_{\rho j}$. Recall also from the proof of Theorem \ref{thm:categorical-algebras-of-operads} that $(\rho,(\gamma_j)_j)$ is \emph{levelwise} when $\rho$ is an identity, and \emph{permutative} when the $\gamma_j$ are all identities. We denote by $T^{[1]}_{\Sigma} X$ the subcategory of $(TX)^{[1]}$ consisting of the permutative maps.

Given a subcategory $\ca S$ of a category $\ca C$ which contains all the objects, or equivalently a class of maps of $\ca C$ containing all identities and closed under composition, we denote by $\ca C^{[1]}_{\ca S}$ the full subcategory of the arrow category of $\ca C$ consisting of the $f \in \ca S$. One has functors $d_{\ca C,\ca S}$ and $c_{\ca C,\ca S} : \ca C^{[1]}_{\ca S} \to \ca C$ which on objects take domains and codomains of morphisms of $\ca S$ respectively, and a natural transformation $\alpha_{\ca C,\ca S} : d_{\ca C,\ca S} \to c_{\ca C,\ca S}$ whose component at $f \in \ca C^{[1]}_{\ca S}$ is $f$ viewed as an arrow of $\ca C$. 
\begin{lem}\label{lem:T1-Sigma-explicit}
If $T$ is an operad with object set $I$ and $X \in \Cat/I$, then $T^{[1]}_{\Sigma}X = (TX)^{[1]}_{T_{\Sigma}X}$, $d_{T,X} = d_{TX,T_{\Sigma}X}$, $c_{T,X} = c_{TX,T_{\Sigma}X}$ and $\alpha_{T,X} = \alpha_{TX,T_{\Sigma}X}$.
\end{lem}
\noindent In particular, $T^{[1]}_{\Sigma}X$ is the full subcategory of the arrow category of $TX$ consisting of the permutative maps.
\begin{proof}
By the general calculation of Example \ref{exam:polyfunctor-Cat-middle-dopfib} applied to the polynomial $(s_T^{[1]},p_T^{[1]},t_T^{[1]})$, an object of $T^{[1]}_{\Sigma}X$ is a pair $(b,h)$ where $b$ is a morphism of $B_T$ and $h : (p_T^{[1]})^{-1}\{b\} \to X$ is a functor over $I$. Such a morphism $b$ is of the form $\rho : \alpha\rho \to \alpha$, where $\alpha : (i_j)_{1{\leq}j{\leq}n} \to i$ is in $T$ and $\rho \in \Sigma_n$, and in these terms an element of the fibre $(p_T^{[1]})^{-1}\{b\}$ is a morphism of $E_T$ of the form $\rho : (\alpha\rho,j) \to (\alpha,\rho j)$ where $1 \leq j \leq n$. Thus an object of $T^{[1]}_{\Sigma}X$ consists of $\rho : \alpha\rho \to \alpha$ and $(x_j)_j$ where $x_j \in X_{\rho j}$ for $1 \leq j \leq n$, and one can identify this with the permutative map $(\rho,(1_{x_j})_j)$ in $T^{[1]}_{\Sigma}X$. Similarly unpacking the morphisms of $T^{[1]}_{\Sigma}X$ following the general scheme of Example \ref{exam:polyfunctor-Cat-middle-dopfib} one identifies a morphism between permutative maps in $T^{[1]}_{\Sigma}X$ as morphism between them in the arrow category $(TX)^{[1]}$. Having established this explicit description of $T^{[1]}_{\Sigma}$, one obtains those for $d_{T,X}$, $c_{T,X}$ and $\alpha_{T,X}$, using the explicit description of the 2 and 3-cell map of $\PFun{\Cat}$ which involves inducing arrows from the universal properties of pullbacks and distributivity pullbacks.
\end{proof}
The conical colimit of a diagram
\begin{equation}\label{diag:coinv}
\begin{gathered}
\xygraph{{A}="p0" [r(1.5)] {B}="p1" "p0":@<-1.5ex>"p1"_-{g}|-{}="b" "p0":@<1.5ex>"p1"^-{f}|-{}="t" "t":@{}"b"|(.15){}="d"|(.85){}="c" "d":@{=>}"c"^-{\phi}}
\end{gathered}
\end{equation}
in a 2-category $\ca K$ is called a \emph{coidentifier} \cite{Kelly-2CatLimits}. Thus a cocone for (\ref{diag:coinv}) consists of $q : B \to Q$ such that $q\phi = \id$, and such a cocone exhibits $Q$ as the coidentifier of this diagram when for all $X \in \ca K$, composition with $q$ induces an isomorphism of categories between $\ca K(Q,X)$ and the full subcategory of $\ca K(B,X)$ consisting of those 1-cells $h : B \to X$ such that $h\phi$ is an identity. In other words $q$, is the universal 1-cell which by post-composition makes $\phi$ into an identity. It is more common to consider the \emph{coinverter} of the above diagram, in which $q$ is the universal 1-cell which by post-composition makes $\phi$ into an isomorphism.

When $\phi$ has a \emph{$1$-section}, that is there is a morphism $i : B \to A$ such that $\phi i = \id$, including or not including the data of $i$ in (\ref{diag:coinv}) clearly does not affect the resulting colimit, and in this case $(Q,q)$ is said to be a \emph{reflexive} coidentifier. It is completely straightforward to adapt lemma 2.1 of \cite{KellyLackWalters-Coinverters}, which considers the case of reflexive coinverters, to exhibit an analogue of the  $(3 \times 3)$-lemma for reflexive coidentifiers. From this it follows that in $\Cat$, reflexive coidentifiers commute with finite products, and so reflexive coidentifiers are a type of sifted colimit.
\begin{exam}\label{exam:pi0-as-coidentifier}
For any category $X$ the discrete category of its connected components can be obtained as a coidentifier
\[ \xygraph{!{0;(1.5,0):(0,1)::} {X^{[1]}}="p0" [r] {X}="p1" [r] {\pi_0X.}="p2"
"p0":@<1.5ex>"p1"|(.45){}="t"^-{d_X} "p0":@<-1.5ex>"p1"|(.45){}="b"_-{c_X}:@{.>}"p2" "t":@{}"b"|(.15){}="d"|(.85){}="c" "d":@{=>}"c"^-{\alpha_X}} \]
This coidentifier is reflexive via the functor $i_X : X \to X^{[1]}$ given on objects by $i_Xx =  1_x$. Thus one recovers the fact that $\pi_0 : \Cat \to \Set$ preserves finite products from the siftedness of reflexive coidentifiers.
\end{exam}
Let $\ca X$ be a monoidal 2-category and $\ca S$ be a set of objects of $\ca X$. Denote by $\ca S^*$ the set of tensor products of objects of $\ca S$, that is, the objects of the smallest monoidal subcategory of $\ca X$ containing $\ca S$. Consider the diagram
\begin{equation}\label{diag:for-stable-coidentifier-defn}
\begin{gathered}
\xygraph{!{0;(1.5,0):(0,1)::} {A}="p0" [r] {B}="p1" [r] {Q}="p2"
"p0":@<1.5ex>"p1"|(.45){}="t"^-{f} "p0":@<-1.5ex>"p1"|(.45){}="b"_-{g}:"p2"^-{q} "t":@{}"b"|(.15){}="d"|(.85){}="c" "d":@{=>}"c"^-{\phi}}
\end{gathered}
\end{equation}
in $\ca X$.
\begin{defn}\label{defn:stable-coidentifiers}
We say that $q$ exhibits $Q$ as an \emph{$\ca S$-stable coidentifier} when for all $X$ and $Y \in \ca S^*$, the 2-functor
\[ X \tensor (-) \tensor Y : \ca X \longrightarrow \ca X \]
sends (\ref{diag:for-stable-coidentifier-defn}) to a coidentifier. When $\phi$ has a $1$-section then (\ref{diag:for-stable-coidentifier-defn}) is said to be an \emph{$\ca S$-stable reflexive coidentifier}.
\end{defn}
Taking $X = Y = I$ one finds that for any $\ca S$, an $\ca S$-stable coidentifier is a coidentifier.
\begin{defn}\label{defn:monoidally-stable-coidentifier}
When the 2-cell $\phi$ lives in $\tn{Mon}(\ca X)$, then (\ref{diag:for-stable-coidentifier-defn}) is said to be a \emph{monoidally stable coidentifier} when it is an $\{A,B\}$-stable coidentifier in $\ca X$. When $\phi$ has a $1$-section in $\ca X$, then (\ref{diag:for-stable-coidentifier-defn}) is said to be a \emph{monoidally stable reflexive coidentifier}.
\end{defn}
Note that in the context of Definition \ref{defn:monoidally-stable-coidentifier}, while there are monoid structures on $A$ and $B$, there is not \emph{apriori} a monoid structure on $Q$. However one has
\begin{prop}\label{prop:ref-coidentifers-monoids}
If (\ref{diag:for-stable-coidentifier-defn}) is a monoidally stable reflexive coidentifier in a monoidal 2-category $\ca X$, then there exists a unique monoid structure on $Q$ making it a coidentifier in $\tn{Mon}(\ca X)$. Moreover this coidentifier is preserved by any of the inclusions
\[ \begin{array}{lcccr} {\tn{Mon}(\ca X) \hookrightarrow \tn{Mon}(\ca X)_{\tn{l}}} &&
{\tn{Mon}(\ca X) \hookrightarrow \tn{Mon}(\ca X)_{\tn{c}}} &&
{\tn{Mon}(\ca X) \hookrightarrow \tn{Mon}(\ca X)_{\tn{ps}}.} \end{array} \]
\end{prop}
\begin{proof}
We begin by verifying that
\begin{equation}\label{diag:tensor-powers-coidentifier}
\begin{gathered}
\xygraph{!{0;(2,0):(0,1)::} {A^{\tensor n}}="p0" [r] {B^{\tensor n}}="p1" [r] {Q^{\tensor n}}="p2"
"p0":@<1.5ex>"p1"|(.45){}="t"^-{f^{\tensor n}} "p0":@<-1.5ex>"p1"|(.45){}="b"_-{g^{\tensor n}}:"p2"^-{q^{\tensor n}} "t":@{}"b"|(.15){}="d"|(.85){}="c" "d":@{=>}"c"^-{\phi^{\tensor n}}}
\end{gathered}
\end{equation}
is an $\{A,B\}$-stable coidentifier by induction on $n$. In the base case $n = 0$, (\ref{diag:tensor-powers-coidentifier}) is a constant diagram, all constant diagrams of this form are clearly coidentifiers, and so (\ref{diag:tensor-powers-coidentifier}) in this case is an absolute coidentifier. The inductive step follows since by the inductive hypothesis, for all $X$ and $Y$ in $\{A,B\}^*$ the first two rows and first two columns of the evident diagram
\[ \xygraph{!{0;(3.5,0):(0,.3)::}
% first row
{X \tensor A \tensor A^{\tensor n} \tensor Y}="p0" [r] {X \tensor A \tensor B^{\tensor n} \tensor Y}="p1" [r] {X \tensor A \tensor Q^{\tensor n} \tensor Y}="p2"
"p0":@<1.5ex>"p1"|(.5){}="t0"^-{} "p0":@<-1.5ex>"p1"|(.5){}="b0"_-{}:"p2"^-{} "t0":@{}"b0"|(.15){}="d0"|(.85){}="c0" "d0":@{=>}"c0"^-{}
% second row
"p0" [d] {X \tensor B \tensor A^{\tensor n} \tensor Y}="q0" [r] {X \tensor B \tensor B^{\tensor n} \tensor Y}="q1" [r] {X \tensor B \tensor Q^{\tensor n} \tensor Y}="q2"
"q0":@<1.5ex>"q1"|(.5){}="t1"^-{} "q0":@<-1.5ex>"q1"|(.5){}="b1"_-{}:"q2"^-{} "t1":@{}"b1"|(.15){}="d1"|(.85){}="c1" "d1":@{=>}"c1"^-{}
% third row
"q0" [d] {X \tensor Q \tensor A^{\tensor n} \tensor Y}="r0" [r] {X \tensor Q \tensor B^{\tensor n} \tensor Y}="r1" [r] {X \tensor Q \tensor Q^{\tensor n} \tensor Y}="r2"
"r0":@<1.5ex>"r1"|(.5){}="t2"^-{} "r0":@<-1.5ex>"r1"|(.5){}="b2"_-{}:"r2"^-{} "t2":@{}"b2"|(.15){}="d2"|(.85){}="c2" "d2":@{=>}"c2"^-{}
% vertical arrows and 2-cells
"p0":@<1.5ex>"q0"|(.5){}="t3"^-{} "p0":@<-1.5ex>"q0"|(.5){}="b3"_-{}:"r0"^-{} "t3":@{}"b3"|(.15){}="d3"|(.85){}="c3" "d3":@{<=}"c3"^-{}
"p1":@<1.5ex>"q1"|(.5){}="t4"^-{} "p1":@<-1.5ex>"q1"|(.5){}="b4"_-{}:"r1"^-{} "t4":@{}"b4"|(.15){}="d4"|(.85){}="c4" "d4":@{<=}"c4"^-{}
"p2":@<1.5ex>"q2"|(.5){}="t5"^-{} "p2":@<-1.5ex>"q2"|(.5){}="b5"_-{}:"r2"^-{} "t5":@{}"b5"|(.15){}="d5"|(.85){}="c5" "d5":@{<=}"c5"^-{}} \]
are reflexive coidentifiers, and the $(3 \times 3)$-lemma. Thus the unit and multplication of $Q$ is induced from those of $A$ and $B$ in the evident manner
\[ \xygraph{{\xybox{\xygraph{!{0;(1,0):(0,1)::}
% first row
{I}="p0" [r] {I}="p1" [r] {I}="p2"
"p0":@<1.5ex>"p1"|(.5){}="t0"^-{} "p0":@<-1.5ex>"p1"|(.5){}="b0"_-{}:"p2"^-{} "t0":@{}"b0"|(.15){}="d0"|(.85){}="c0" "d0":@{=>}"c0"^-{}
% second row
"p0" [d] {A}="q0" [r] {B}="q1" [r] {Q}="q2"
"q0":@<1.5ex>"q1"|(.5){}="t1"^-{} "q0":@<-1.5ex>"q1"|(.5){}="b1"_-{}:"q2"^-{} "t1":@{}"b1"|(.15){}="d1"|(.85){}="c1" "d1":@{=>}"c1"^-{}
% vertical arrows
"p0":"q0" "p1":"q1" "p2":@{.>}"q2"}}}
[r(4.5)]
{\xybox{\xygraph{!{0;(1.5,0):(0,.6667)::}
% first row
{A \tensor A}="p0" [r] {B \tensor B}="p1" [r] {Q \tensor Q}="p2"
"p0":@<1.5ex>"p1"|(.5){}="t0"^-{} "p0":@<-1.5ex>"p1"|(.5){}="b0"_-{}:"p2"^-{} "t0":@{}"b0"|(.15){}="d0"|(.85){}="c0" "d0":@{=>}"c0"^-{}
% second row
"p0" [d] {A}="q0" [r] {B}="q1" [r] {Q}="q2"
"q0":@<1.5ex>"q1"|(.5){}="t1"^-{} "q0":@<-1.5ex>"q1"|(.5){}="b1"_-{}:"q2"^-{} "t1":@{}"b1"|(.15){}="d1"|(.85){}="c1" "d1":@{=>}"c1"^-{}
% vertical arrows
"p0":"q0" "p1":"q1" "p2":@{.>}"q2"}}}} \]
and the monoid axioms for $Q$ follow easily from the 1-dimensional universal properties of (\ref{diag:tensor-powers-coidentifier}).

Suppose that $C$ is a monoid in $\ca X$, and $h : B \to C$ is a lax morphism of monoids such that $h\phi = \id$ in $\tn{Mon}(\ca X)_{\tn{l}}$. By the way that composition works in $\tn{Mon}(\ca X)_{\tn{l}}$, this last equation amounts to the equation $h\phi = \id$ in $\ca X$. By the 1-dimensional universal property of $q$, one has $k : Q \to C$ unique such that $kq = h$. The lax morphism coherence 2-cells $(\overline{k}_0,\overline{k}_2)$ for $k$ are induced from the corresponding coherences for $h$ via the 2-dimensional universal properties of $q^{\tensor 0}=1_I$ and $q^{\tensor 2}$ as in
\[ \xygraph{
% inducing unit coherence
{\xybox{\xygraph{!{0;(1.25,0):(0,.5)::}
% top
{I}="p0" [r] {I}="p1" [r] {I}="p2" [dr] {I}="p3"
"p0":@<1.5ex>"p1"|(.45){}="pt"^-{} "p0":@<-1.5ex>"p1"|(.45){}="pb"_-{}(:@/_{1pc}/"p3"_-{},:"p2"^-{1_I}:"p3"^-{1_I}) "pt":@{}"pb"|(.15){}="pd"|(.85){}="pc" "pd":@{=>}"pc"^-{\id}
% bottom
"p0" [d(2)] {A}="q0" [r] {B}="q1" [r] {Q}="q2" [dr] {C}="q3"
"q0":@<1.5ex>"q1"|(.45){}="qt"^-{} "q0":@<-1.5ex>"q1"|(.45){}="qb"_-{}(:@/_{1pc}/"q3"_-{h},:"q2"|(.6)*=<5pt>{}:"q3"_-{k}) "qt":@{}"qb"|(.15){}="qd"|(.85){}="qc" "qd":@{=>}"qc"^-{\phi}
% vertical
"p0":"q0" "p1":"q1" "p2":"q2"|(.525)*=<3pt>{} "p3":"q3" "q2" [u(.6)l(.4)] :@{=>}[d(.9)] [u(.65)l(.15)] {\scriptstyle{\overline{h}_0}} "q2" [u(.7)r(.5)] :@{:>}[d(.7)] [u(.35)r(.2)] {\scriptstyle{\overline{k}_0}}}}}
[r(5)]
% inducing multiplication coherence
{\xybox{\xygraph{!{0;(1.25,0):(0,.5)::}
% top
{A^{\tensor 2}}="p0" [r] {B^{\tensor 2}}="p1" [r] {Q^{\tensor 2}}="p2" [dr] {C^{\tensor 2}}="p3"
"p0":@<1.5ex>"p1"|(.35){}="pt"^-{} "p0":@<-1.5ex>"p1"|(.35){}="pb"_-{}(:@/_{1pc}/"p3"_-{},:"p2"^-{q^{\tensor 2}}:"p3"^-{k^{\tensor 2}}) "pt":@{}"pb"|(.15){}="pd"|(.85){}="pc" "pd":@{=>}"pc"^-{\phi^{\tensor 2}}
% bottom
"p0" [d(2)] {A}="q0" [r] {B}="q1" [r] {Q}="q2" [dr] {C}="q3"
"q0":@<1.5ex>"q1"|(.45){}="qt"^-{} "q0":@<-1.5ex>"q1"|(.45){}="qb"_-{}(:@/_{1pc}/"q3"_-{h},:"q2"|(.6)*=<5pt>{}:"q3"_-{k}) "qt":@{}"qb"|(.15){}="qd"|(.85){}="qc" "qd":@{=>}"qc"^-{\phi}
% vertical
"p0":"q0" "p1":"q1" "p2":"q2"|(.525)*=<3pt>{} "p3":"q3" "q2" [u(.6)l(.4)] :@{=>}[d(.9)] [u(.65)l(.15)] {\scriptstyle{\overline{h}_2}} "q2" [u(.7)r(.5)] :@{:>}[d(.7)] [u(.35)r(.2)] {\scriptstyle{\overline{k}_2}}}}}} \]
and the lax morphism coherence axioms for $(\overline{k}_0,\overline{k}_2)$ follow from those for $(\overline{h}_0,\overline{h}_2)$ and the 2-dimensional universal properties of (\ref{diag:tensor-powers-coidentifier}). In this way (\ref{diag:for-stable-coidentifier-defn}) has the 1-dimensional universal property of a coidentifier in $\tn{Mon}(\ca X)_{\tn{l}}$, and reversing the 2-cells in this discussion exhibits the 1-dimensional universal property of a coidentifier in $\tn{Mon}(\ca X)_{\tn{c}}$. By the 2-dimensional universal properties of (\ref{diag:tensor-powers-coidentifier}) again, $\overline{k}_0$ and $\overline{k}_2$ are isomorphisms or identities iff $\overline{h}_0$ and $\overline{h}_2$ are, and so (\ref{diag:for-stable-coidentifier-defn}) also has the 1-dimensional universal property of a coidentifier in $\tn{Mon}(\ca X)_{\tn{ps}}$ and in $\tn{Mon}(\ca X)$. In a similar manner, the 2-dimensional universal properties of (\ref{diag:for-stable-coidentifier-defn}) in the various 2-categories of monoids under consideration are verified directly using the 2-dimensional universal properties of (\ref{diag:tensor-powers-coidentifier}).
\end{proof}
\begin{exam}\label{exam:coidentifier-of-2monads}
For any 2-category $\ca K$ and any endo-2-functor $R$ of $\ca K$, the 2-functor on the left
\[ \begin{array}{lccr} {(-) \comp R : \tn{End}(\ca K) \longrightarrow \tn{End}(\ca K)} &&& {R \comp (-) : \tn{End}(\ca K) \longrightarrow \tn{End}(\ca K)} \end{array} \]
preserves all colimits, and the 2-functor on the right preserves any colimits that $R$ preserves. Thus given 2-monads $S$ and $T$ on $\ca K$ which preserve reflexive coidentifiers, a 2-monad 2-cell $\phi$ between them with a 1-section in $\tn{End}(\ca K)$, and a coidentifier
\[ \xygraph{!{0;(1.5,0):(0,1)::} {S}="p0" [r] {T}="p1" [r] {Q}="p2"
"p0":@<1.5ex>"p1"|(.45){}="t"^-{} "p0":@<-1.5ex>"p1"|(.45){}="b"_-{}:"p2"^-{q} "t":@{}"b"|(.15){}="d"|(.85){}="c" "d":@{=>}"c"^-{\phi}} \]
of $\phi$ in $\tn{End}(\ca K)$, then by Proposition \ref{prop:ref-coidentifers-monoids} there is a unique 2-monad structure on $Q$ making $q$ the coidentifier of $\phi$ in $\tn{Mnd}(\ca K)$.
\end{exam}
\begin{rem}\label{rem:for-def-T1-mod-sigma}
In the situation of Definition \ref{defn:T1-Sigma} the 2-monads $T$ and $T^{[1]}_{\Sigma}$ preserve sifted colimits by \cite{Weber-PolynomialFunctors} Theorem 4.5.1. Moreover since $p_T$ is a discrete fibration it reflects identities, and so the square in
\[ \xygraph{!{0;(1.5,0):(0,.4)::} {I}="p0" [ur] {E^{[1]}_T}="p1" [r] {B^{[1]}_T}="p2" [dr] {I}="p3" [dl] {B_T}="p4" [l] {E_T}="p5" "p0":@{<-}"p1"^-{s^{[1]}_T}:"p2"^-{p^{[1]}_T}:"p3"^-{t^{[1]}_T}:@{<-}"p4"^-{t_T}:@{<-}"p5"^-{p_T}:"p0"^-{s_T}
"p5":"p1"_-{i_{E_T}} "p4":"p2"_-{i_{B_T}}} \]
is easily verified to be a pullback. Thus this diagram exhibits a 1-section of the 2-cell $(\alpha_{B_T},\alpha_{E_T})$ in $\Polyc{\Cat}(I,I)$, and so the situation of Definition \ref{defn:T1-Sigma} conforms to that of Example \ref{exam:coidentifier-of-2monads}. Hence the coidentifier of the 2-monad 2-cell $\alpha_T$ exists, and is computed as in $\tn{End}(\Cat/I)$.
\end{rem}
\begin{defn}\label{defn:T-mod-Sigma}
The coidentifying monad morphism of Remark \ref{rem:for-def-T1-mod-sigma} is denoted
\[ q_T : T \longrightarrow T/\Sigma. \]
\end{defn}
\begin{rem}\label{rem:T-mod-Sigma-restricts-to-discrete-objects}
When $X \to I$ in $\Cat/I$ is discrete all of $TX$'s morphisms are permutative, and so by Lemma \ref{lem:T1-Sigma-explicit}, $T^{[1]}_{\Sigma}X = (TX)^{[1]}$, whence $T/\Sigma(X) = \pi_0(TX)$ by Example \ref{exam:pi0-as-coidentifier}. Thus the 2-monad $T/\Sigma$ just defined restricts to a monad on $\Set/I$.
\end{rem}
In the remainder of this section we establish analogues of Theorems \ref{thm:categorical-algebras-of-operads}, \ref{thm:catalg-morphisms-for-operads} and \ref{thm:catalg-2cells-for-operads}, which say that the algebras of $T/\Sigma$ correspond to the commutative variants of the algebras of $T$. These results are established by using endomorphism 2-monads and their variants \cite{Kelly-CoherenceLaxAlgDistLaws, KellyLack-PropertyLikeStructures, Lack-2CatsCompanion}, the details of which we shall recall as needed.

Let $\ca K$ be a complete 2-category, and for $A \in \ca K$ and $X \in \Cat$ we denote by $A^X$ the cotensor of $X$ with $A$, whose universal property gives isomorphisms as on the left
\[ \begin{array}{lccr} {\ca K(B,A^X) \iso \Cat(X,\ca K(B,A))} &&&
{\langle A,B\rangle C = B^{\ca K(C,A)}} \end{array} \]
2-naturally in $B$. Then the formula on the right for $A$, $B$ and $C \in \ca K$ defines $\langle A,B \rangle \in \tn{End}(\ca K)$. The assignation $B \mapsto \langle A,B \rangle$ is the object map of a 2-functor
\[ \begin{array}{lccr} {\langle A,- \rangle : \ca K \longrightarrow \tn{End}(\ca K)} &&& {\varepsilon_{A,B} : \langle A,B \rangle A \longrightarrow B} \end{array} \]
which is right adjoint to the 2-functor given by evaluating at $A$, and the counit of this adjunction is as denoted on the right in previous display. As explained in \cite{GordonPower-EnrichmentVariation, JanelidzeKelly-Actegories}, $\langle A,B \rangle$ is thus the hom for an enrichment of $\ca K$ in $\tn{End}(\ca K)$. The units $u_A : 1_{\ca K} \to \langle A,A \rangle$ and compositions $c_{ABC} : \langle B,C \rangle \langle A,B \rangle \to \langle A,C \rangle$ for this enrichment are the unique 2-natural transformations making
\[ \xygraph{{\xybox{\xygraph{!{0;(1.5,0):(0,.6667)::} {A}="p0" [r] {\langle A,A \rangle}="p1" [d] {A}="p2" "p0":"p1"^-{u_{A,A}}:"p2"^-{\varepsilon_{A,A}}:@{<-}"p0"^-{1_A}}}}
[r(4)d(.07)]
{\xybox{\xygraph{!{0;(2.5,0):(0,.4)::} {\langle B,C \rangle \langle A,B \rangle A}="p0" [r] {\langle A,C \rangle A}="p1" [d] {C}="p2" [l] {\langle B,C \rangle B}="p3" "p0":"p1"^-{c_{ABC,A}}:"p2"^-{\varepsilon_{A,C}}:@{<-}"p3"^-{\varepsilon_{B,C}}:@{<-}"p0"^-{\langle B,C \rangle \varepsilon_{A,B}}}}}} \]
commute. In particular for $A \in \ca K$, one has the corresponding \emph{endomorphism 2-monad} $\langle A,A \rangle$.

By the universal property of $\varepsilon_{A,A}$ one has a bijection between 2-natural transformations $\phi : T \to \langle A,A \rangle$ and morphisms $a : TA \to A$ in $\ca K$, and when $T$ is a 2-monad, $\phi$ satisfies the axioms of a strict morphism of 2-monads iff $a$ satisfies the axioms of a strict $T$-algebra. As Kelly first observed in \cite{Kelly-CoherenceLaxAlgDistLaws}, this correspondence extends to give a description of lax and pseudo algebra structures in terms of morphisms of 2-monads. In the lax case the data for a lax morphism of 2-monads $T \to \langle A,A \rangle$ includes an underlying 2-natural transformation $\phi$ as above, together with modifications $\overline{\phi}_0$ and $\overline{\phi}_2$ as in
\[ \xygraph{*{\xybox{\xygraph{!{0;(1.5,0):(0,.6667)::} {1_{\ca K}}="p0" [r] {T}="p1" [d] {\langle A,A \rangle}="p2" "p0":"p1"^-{\eta}:"p2"^-{\phi}:@{<-}"p0"^-{u_A} "p0" [d(.4)r(.6)] :@{=>}[r(.2)]^-{\overline{\phi}_0}}}}
[r(2.8)]
*!(0,.03){\xybox{\xygraph{!{0;(1.5,0):(0,.6667)::} {T^2}="p0" [r] {T}="p1" [d] {\langle A,A \rangle}="p2" [l] {\langle A,A \rangle^2}="p3" "p0":"p1"^-{\mu}:"p2"^-{\phi}:@{<-}"p3"^-{c_{AAA}}:@{<-}"p0"^-{\phi^2} "p0" [d(.55)r(.4)] :@{=>}[r(.2)]^-{\overline{\phi}_2}}}}
[r(3)]
*!(0,-.02){\xybox{\xygraph{!{0;(1.5,0):(0,.6667)::} {A}="p0" [r] {TA}="p1" [d] {A}="p2" "p0":"p1"^-{\eta_A}:"p2"^-{a}:@{<-}"p0"^-{1_A} "p0" [d(.4)r(.6)] :@{=>}[r(.2)]^-{\overline{a}_0}}}}
[r(2.8)]
*{\xybox{\xygraph{!{0;(1.5,0):(0,.6667)::} {T^2A}="p0" [r] {TA}="p1" [d] {A}="p2" [l] {TA}="p3" "p0":"p1"^-{\mu_A}:"p2"^-{a}:@{<-}"p3"^-{a}:@{<-}"p0"^-{Ta} "p0" [d(.55)r(.4)] :@{=>}[r(.2)]^-{\overline{a}_2}}}}} \]
which by the 2-dimensional part of the universal property of $\varepsilon_{A,A}$, corresponds to 2-cells $\overline{a}_0$ and $\overline{a}_2$ in $\ca K$. Moreover the lax morphism coherence axioms on $\overline{\phi}_0$ and $\overline{\phi}_2$ correspond to the lax algebra coherence axioms on $\overline{a}_0$ and $\overline{a}_2$, and $\overline{\phi}_0$ and $\overline{\phi}_2$ are invertible iff $\overline{a}_0$ and $\overline{a}_2$ are. In this way, lax or pseudo morphisms $T \to \langle A,A \rangle$ of 2-monads, may be identified with lax or pseudo $T$-algebra structures on $A$.
\begin{thm}\label{thm:commutative-algebras-of-operads}
Let $T$ be an operad with object set $I$ and let $H \to I$ be an object of $\Cat/I$. To give $H$ the structure of a lax, pseudo or strict $(T/\Sigma)$-algebra is to give it the structure of a commutative lax, commutative pseudo or commutative strict morphism $H : T \to \CatAsOp$ respectively.
\end{thm}
\begin{proof}
By Definition \ref{defn:T-mod-Sigma} and Proposition \ref{prop:ref-coidentifers-monoids} to give a lax morphism $T/\Sigma \to \langle H,H \rangle$ of 2-monads is to give a lax morphism $T \to \langle H,H \rangle$ such that underlying 2-natural transformation $\phi$ post-composes with $\alpha_T$ of Definition \ref{defn:T1-Sigma} to give an identity. By the universal property of $\varepsilon_{H,H}$ this is the same as giving a lax $T$-algebra structure on $H$, whose 1-cell datum $a : TH \to H$ post-composes with $\alpha_{T,H}$ to give an identity. By Theorem \ref{thm:categorical-algebras-of-operads} a lax $T$-algebra structure on $H$ is the same thing as a lax morphism $H : T \to \CatAsOp$, and by the explicit description of $\alpha_{T,H}$ provided by Lemma \ref{lem:T1-Sigma-explicit}, the condition that $a : TH \to H$ post-composes with $\alpha_{T,H}$ to give an identity corresponds to the condition that the symmetries of $H : T \to \CatAsOp$ are identities. The pseudo and strict cases follow in the same way by considering pseudo and strict morphisms of 2-monads $T \to \langle H,H \rangle$.
\end{proof}
As Kelly also understood \cite{Kelly-CoherenceLaxAlgDistLaws}, the different types of morphisms of algebras of a 2-monad can also be regarded as morphisms of 2-monads. In the above setting of a complete 2-category $\ca K$, given $f : A \to B$ in $\ca K$ and $T \in \tn{End}(\ca K)$, data $(\tilde{a},\tilde{b},\tilde{\phi})$
\[ \xygraph{{\xybox{\xygraph{!{0;(1.5,0):(0,.6667)::} {T}="p0" [r] {\langle B,B \rangle}="p1" [d] {\langle A,B \rangle}="p2" [l] {\langle A,A \rangle}="p3" "p0":"p1"^-{\tilde{b}}:"p2"^-{\langle f,B \rangle}:@{<-}"p3"^-{\langle A,f \rangle}:@{<-}"p0"^-{\tilde{a}} "p0" [d(.55)r(.4)] :@{=>}[r(.2)]^-{\tilde{\phi}}}}}
[r(4)d(.02)]
{\xybox{\xygraph{!{0;(1.5,0):(0,.6667)::} {TA}="p0" [r] {A}="p1" [d] {B}="p2" [l] {TB}="p3" "p0":"p1"^-{a}:"p2"^-{f}:@{<-}"p3"^-{b}:@{<-}"p0"^-{Tf} "p0" [d(.55)r(.4)] :@{=>}[r(.2)]^-{\phi}}}}} \]
in $\tn{End}(\ca K)$ is in bijection with the data $(a,b,\phi)$ in $\ca K$ by the universal property of $\varepsilon_{A,B}$. Let $\{f,f\}_{\tn{l}}$ be the comma object $\langle A,f \rangle \downarrow \langle f,B \rangle$, denote its defining comma square in $\tn{End}(\ca K)$ as on the left
\[ \xygraph{{\xybox{\xygraph{!{0;(1.5,0):(0,.6667)::} {\{f,f\}_{\tn{l}}}="p0" [r] {\langle B,B \rangle}="p1" [d] {\langle A,B \rangle}="p2" [l] {\langle A,A \rangle}="p3" "p0":"p1"^-{\tilde{b}_f}:"p2"^-{\langle f,B \rangle}:@{<-}"p3"^-{\langle A,f \rangle}:@{<-}"p0"^-{\tilde{a}_f} "p0" [d(.55)r(.4)] :@{=>}[r(.2)]^-{\tilde{\phi}_f}}}}
[r(4)d(.05)]
{\xybox{\xygraph{!{0;(1.5,0):(0,.6667)::} {\{f,f\}_{\tn{l}}A}="p0" [r] {A}="p1" [d] {B}="p2" [l] {\{f,f\}_{\tn{l}}B}="p3" "p0":"p1"^-{a_f}:"p2"^-{f}:@{<-}"p3"^-{b_f}:@{<-}"p0"^-{\{f,f\}_{\tn{l}}(f)} "p0" [d(.55)r(.4)] :@{=>}[r(.2)]^-{\phi_f}}}}} \]
and denote the corresponding data in $\ca K$ as on the right. In terms of this data in $\ca K$, the 1-dimensional part of the comma object universal property says that given $(a,b,\phi)$ as above, one has a unique 2-natural transformation $\phi' : T \to \{f,f\}_{\tn{l}}$ such that $a_f\phi'_A = a$, $b_f\phi'_B = b$ and $\phi_f\phi'_A = \phi$. In particular from
\[ \xygraph{{\xybox{\xygraph{!{0;(1.5,0):(0,.6667)::} {A}="p0" [r] {A}="p1" [d] {B}="p2" [l] {B}="p3" "p0":"p1"^-{1_A}:"p2"^-{f}:@{<-}"p3"^-{1_B}:@{<-}"p0"^-{f}
"p0" [d(.55)r(.4)] :@{=>}[r(.2)]^-{\id}}}}
[r(5.5)]
{\xybox{\xygraph{!{0;(2.5,0):(0,.4)::} {\{f,f\}_{\tn{l}}^2A}="p0" [r] {\{f,f\}_{\tn{l}}A}="p1" [r] {A}="p2" [d] {B}="p3" [l] {\{f,f\}_{\tn{l}}B}="p4" [l] {\{f,f\}_{\tn{l}}^2B}="p5" "p0":"p1"^-{\{f,f\}a_f}:"p2"^-{a_f}:"p3"^-{f}:@{<-}"p4"^-{b_f}:@{<-}"p5"^-{\{f,f\}b_f}:@{<-}"p0"^-{\{f,f\}_{\tn{l}}^2(f)} "p1":"p4"|-{\{f,f\}_{\tn{l}}(f)}
"p0" [d(.55)r(.425)] :@{=>}[r(.15)]^-{\{f,f\}_{\tn{l}}\phi_f}
"p1" [d(.55)r(.425)] :@{=>}[r(.15)]^-{\phi_f}}}}} \]
one induces $\eta^f : 1_{\ca K} \to \{f,f\}_{\tn{l}}$ and $\mu^f : \{f,f\}_{\tn{l}}^2 \to \{f,f\}_{\tn{l}}$ unique providing the unit and multiplication of a 2-monad, making $a_f$ and $b_f$ into strict algebra structures for this 2-monad, and $(f,\phi_f)$ into a lax morphism between them. Moreover, with respect to this 2-monad structure on $\{f,f\}_{\tn{l}}$ and the endomorphism 2-monad structures on $\langle A,A \rangle$ and $\langle B,B \rangle$, $\tilde{a}_f$ and $\tilde{b}_f$ become strict morphisms of 2-monads.

From the universal property of $\varepsilon_{A,B}$ it follows that to give a lax, pseudo or strict morphism $T \to \{f,f\}_{\tn{l}}$ of 2-monads, is the same as giving $A$ and $B$ lax, pseudo or strict algebra structures respectively, and a 2-cell $\phi$ providing the coherence making $(f,\phi)$ a lax morphism of $T$-algebras. Composing with $\tilde{a}_f$ and $\tilde{b}_f$ one recovers the $T$-algebra structures on $A$ and $B$ as morphisms of 2-monads. To summarise, $\{f,f\}_{\tn{l}}$ is a 2-monad on $\ca K$ which classifies lax morphisms of algebras with underlying 1-cell $f$. Replacing the comma object $\langle A,f \rangle \downarrow \langle f,B \rangle$ in the above discussion by either $\langle f,B \rangle \downarrow \langle A,f \rangle$, the isocomma object or the pullback, produces the 2-monads $\{f,f\}_{\tn{c}}$, $\{f,f\}_{\tn{ps}}$ and $\{f,f\}$, which similarly classify colax, pseudo and strict morphisms of algebras with underlying 1-cell $f$ respectively.
\begin{thm}\label{thm:comalg-morphisms-for-operads}
Let $T$ be an operad with object set $I$ and let $f : H \to K$ be a morphism of $\Cat/I$. Suppose also that one has the structure of commutative lax morphism $T \to \CatAsOp$ on both $H$ and $K$. Then to give $f$ the structure of lax, colax, pseudo or strict $(T/\Sigma)$-algebra morphism, is to give $f$ the structure of lax-natural, colax-natural, pseudo-natural or natural transformation respectively.
\end{thm}
\begin{proof}
To give $f$ the structure of a lax $(T/\Sigma)$-algebra morphism is to give a lax morphism $(T/\Sigma) \to \{f,f\}_{\tn{l}}$ of 2-monads whose composites with $\tilde{a}_f$ and $\tilde{b}_f$ correspond as monad morphisms to the given $T/\Sigma$-algebra structures on $H$ and $K$. By Definition \ref{defn:T-mod-Sigma} and Proposition \ref{prop:ref-coidentifers-monoids}, to give a lax morphism $(T/\Sigma) \to \{f,f\}_{\tn{l}}$, is to give a lax morphism $\phi : T \to \{f,f\}_{\tn{l}}$ whose underlying 2-natural transformation post composes with $\alpha_T$ to an identity. This last condition is equivalent, by the 2-dimensional universal property of the defining comma square for $\{f,f\}_{\tn{l}}$ in $\tn{End}(\ca K)$, to the condition that the underlying 2-natural transformations of $\tilde{a}_f\phi$ and $\tilde{b}_f\phi$ post compose with $\alpha_T$ to identities, but this just says in turn that the underlying $T$-algebra structures on $H$ and $K$ correspond to commutative lax morphisms $T \to \CatAsOp$. Thus the result follows from Theorem \ref{thm:catalg-morphisms-for-operads} in the lax case. For colax, pseudo and strict $T/\Sigma$-algebra morphisms, one argues in the same way using $\{f,f\}_{\tn{c}}$, $\{f,f\}_{\tn{ps}}$ and $\{f,f\}$ respectively.
\end{proof}
One also has 2-monads that classify algebra 2-cells. In the situation of a complete 2-category $\ca K$, 1-cells $f$ and $g : A \to B$, and a 2-cell $\gamma : f \to g$, one can define the comma object
\[ \xygraph{!{0;(2,0):(0,.5)::} {\{\gamma,\gamma\}_{\tn{l}}}="p0" [r] {\langle B,B \rangle}="p1" [d] {\langle A,B \rangle^{[1]}}="p2" [l] {\langle A,A \rangle}="p3" "p0":"p1"^-{\tilde{b}_{\gamma}}:"p2"^-{\langle \phi,B \rangle}:@{<-}"p3"^-{\langle A,\phi \rangle}:@{<-}"p0"^-{\tilde{a}_{\gamma}} "p0" [d(.55)r(.425)] :@{=>}[r(.15)]^-{\tilde{\phi}_{\gamma}}} \]
in $\tn{End}(\ca K)$. In a similar manner to our discussion of $\{f,f\}_{\tn{l}}$ above, one can then exhibit the unit and multiplication for a 2-monad structure on $\{\gamma,\gamma\}_{\tn{l}}$, and describe a bijection between lax morphisms of 2-monads $T \to \{\gamma,\gamma\}_{\tn{l}}$ and algebra 2-cells, between lax morphisms of lax $T$-algebras. As before one classifies such algebra 2-cells between stricter types of algebra by using the corresponding stricter type of morphism of 2-monads $T \to \{\gamma,\gamma\}_{\tn{l}}$, and one classifies algebra 2-cells between colax, pseudo and strict morphisms by considering the appropriate 2-monad $\{\gamma,\gamma\}_{\tn{c}}$, $\{\gamma,\gamma\}_{\tn{ps}}$ and $\{\gamma,\gamma\}$, obtained by reversing the direction of the comma object, taking an isocomma object or a pullback respectively. The proof of 
\begin{thm}\label{thm:comalg-2cells-for-operads}
Let $T$ be an operad with object set $I$. Suppose that one has commutative lax morphisms of operads $H$ and $K : T \to \CatAsOp$, and lax-natural (resp. colax-natural) transformations $(f,\overline{f})$ and $(g,\overline{g}) : H \to K$. Then to give a modification $(f,\overline{f}) \to (g,\overline{g})$ is to give an algebra 2-cell between the corresponding lax (resp. colax) morphisms of $(T/\Sigma)$-algebras.
\end{thm}
\noindent then unfolds analogously to that of Theorem \ref{thm:comalg-morphisms-for-operads}.

For each type of 2-category of algebra for 2-monads, a 2-monad morphism $T \to S$ induces a ``forgetful''  2-functor from the corresponding 2-category of algebras of $S$ to those of $T$. In the present situation of $q_T : T \to T/\Sigma$, Theorems \ref{thm:commutative-algebras-of-operads}-\ref{thm:comalg-2cells-for-operads} say that these correspond to the inclusions of the \emph{commutative} lax, pseudo or strict operad morphisms $T \to \CatAsOp$ amongst the general such morphisms. Regardless of which type of algebra 1-cells are considered, these inclusions are clearly 2-fully faithful.
\begin{rem}\label{rem:reconcile-TmodSigma}
At the beginning of this article we introduced the notation $T/\Sigma$ to denote the monad on $\Set/I$ arising from the operad $T$ via the standard construction of a monad from an operad. In this standard view the effect of $T/\Sigma$ on $X \to I$ in $\Set/I$ is given by the formula
\begin{equation}\label{eq:TmodSigma-standard-formula}
(T/\Sigma(X))_i = \coprod_{n \in \N} \left( \coprod_{i_1,...,i_n} T((i_j)_j;i) \times \prod_{j=1}^n X_{i_j} \right) / \Sigma_n
\end{equation}
interpretted as follows. The term in the bracket is acted on by $\Sigma_n$ by permuting the variables $(i_1,...,i_n)$, and then the notation $(-)/\Sigma_n$ is the standard notation for identifying the orbits of this $\Sigma_n$-action. By Remark \ref{rem:T-mod-Sigma-restricts-to-discrete-objects} $T/\Sigma$ given in Definition \ref{defn:T-mod-Sigma} restricts to a monad on $\Set/I$, and by Theorems \ref{thm:commutative-algebras-of-operads} and \ref{thm:comalg-morphisms-for-operads} its algebras coincide with the version of $T/\Sigma$ defined in the standard way via (\ref{eq:TmodSigma-standard-formula}). In this way these two uses of the notation $T/\Sigma$ are consistent. Thus the defining coidentifier of $T/\Sigma$ of Definition \ref{defn:T-mod-Sigma} is an alternative way of expressing the formula (\ref{eq:TmodSigma-standard-formula}).
\end{rem}
\begin{exam}\label{exam:Ass-commutative-algebras}
In \cite{Weber-PolynomialFunctors} the 2-monad $\MCMnd$ on $\Cat$ was exhibited as perhaps the most basic example of a polynomial 2-monad. As we saw in Example \ref{exam:Ass-vs-M} the commutative algebras for the operad $\Ass$ for monoids are exactly strict monoidal categories, and so $\Ass/\Sigma = \MCMnd$. By contrast, $\Com/\Sigma$ is not a cartesian monad (see \cite{Weber-Generic}) and thus not polynomial.
\end{exam}

\section{$\Sigma$-free operads}
\label{sec:sigma-free}

For any operad $T$ with set colours $I$, while corresponding 2-monad $T$ on $\Cat/I$ is a polynomial 2-monad, this is not always so for $T/\Sigma$ as we saw in Example \ref{exam:Ass-commutative-algebras}. In the main result of this section, Theorem \ref{thm:T-mod-Sigma-polynomial}, we recover the fact \cite{KockJ-PolyFunTrees, SzawielZawadowski-TheoriesOfAnalyticMonads} that when the operad $T$ is $\Sigma$-free, $T/\Sigma$ is polynomial, and establish that $q_T : T \to T/\Sigma$ is a polynomial monad morphism.

Let $T$ be a collection with object set $I$. The action of permutations on operations of $T$ provide, for any sequence $(i_j)_{1{\leq}j{\leq}n}$ of objects of $T$, $i \in I$, and any permutation $\rho \in \Sigma_n$, a bijection
\begin{equation}\label{eq:action-for-collection}
(-)\rho : T((i_{\rho j})_j;i) \to T((i_j)_j;i).
\end{equation}
It can happen that the sequences $(i_{\rho j})_j$ and $(i_j)_j$ are in fact equal, for instance when all the $i_j$'s are the same element of $I$. In such cases one can then ask whether $(-)\rho$ has any fixed points. 
\begin{defn}\label{defn:}
A collection $T$ is said to be \emph{$\Sigma$-free} when for all $(i_j)_{1{\leq}j{\leq}n}$, $i$ and $\rho$ as above such that $\rho \neq 1_n$ and $(i_{\rho j})_j = (i_j)_j$, the bijection (\ref{eq:action-for-collection}) has no fixed points. A \emph{$\Sigma$-free operad} is an operad whose underlying collection is $\Sigma$-free.
\end{defn}
We begin by characterising the $\Sigma$-freeness of a collection in various ways. Preliminary to this, it is useful to have various alternative characterisations of those categories equivalent to discrete categories. For any category $X$ we denote by $q_X : X \to \pi_0X$ the surjective-on-objects functor which sends $x \in X$ to its connected component. In other words, $q_X$ is the coidentifier of the 2-cell which arises from taking the cotensor of $X$ with $[1]$ (as in Example \ref{exam:pi0-as-coidentifier}), and moreover is the component at $X$ of the unit of the adjunction with left adjoint $\pi_0 : \Cat \to \Set$. Recall that a category is \emph{indiscrete} when it is equivalent to the terminal category $1$, or equivalently when it is non-empty and there is a unique morphism between any two objects. The straightforward proof of
\begin{lem}\label{lem:ess-discrete}
For $X \in \Cat$ the following statements are equivalent.
\begin{enumerate}
\item $X$ is equivalent to a discrete category.
\item $q_X : X \to \pi_0X$ is fully faithful.
\item $X$ is a groupoid and every morphism of $X$ is unique in its hom-set.\label{lemcase:essdisc-gpd-plus-uniqueness}
\item $X$ is a coproduct of indiscrete categories.\label{lemcase:essdisc-coproduct-indiscrete}
\end{enumerate}
\end{lem}
\noindent is left to the reader. Categories equivalent to discrete categories are closed under various 2-categorical constructions relevant for us.
\begin{lem}\label{lem:ess-disc-inheritance}
\begin{enumerate}
\item If $p : E \to B$ is a discrete fibration and $B$ is equivalent to a discrete category, then so is $E$.\label{lemcase:essdisc-inherited-along-dfib}
\item In a pullback as on the left
\[ \xygraph{{\xybox{\xygraph{{P}="p0" [r] {B}="p1" [d] {C}="p2" [l] {A}="p3" "p0":"p1"^-{}:"p2"^-{}:@{<-}"p3"^-{}:@{<-}"p0"^-{}:@{}"p2"|-{\tn{pb}}}}}
[r(4)]
{\xybox{\xygraph{{P}="p0" [r] {A}="p1" [r] {B}="p2" [d] {C}="p3" [l(2)] {Q}="p4" "p0":"p1"^-{}:"p2"^-{g}:"p3"^-{f}:@{<-}"p4"^-{}:@{<-}"p0"^-{}:@{}"p3"|-{\tn{dpb}}}}}} \]
in $\Cat$, if $A$, $B$ and $C$ are equivalent to discrete categories, then so is $P$.\label{lemcase:essdisc-inherited-by-pb}
\item In a distributivity pullback around $(f,g)$ in $\Cat$ as on the right in the previous display in which $f$ is a discrete fibration, if $A$, $B$ and $C$ are equivalent to discrete categories, then so are $P$ and $Q$.\label{lemcase:essdisc-inherited-by-dpb}
\item If in
\[ \xygraph{!{0;(1.5,0):(0,1)::} {A}="p0" [r] {B}="p1" [r] {Q}="p2"
"p0":@<1.5ex>"p1"|(.45){}="t"^-{f} "p0":@<-1.5ex>"p1"|(.45){}="b"_-{g}:"p2"^-{q} "t":@{}"b"|(.15){}="d"|(.85){}="c" "d":@{=>}"c"^-{\phi}} \]
$q$ is a reflexive coidentifier of $\phi$ and $B$ is equivalent to a discrete category, then so is $Q$, and $q$ is surjective on objects and arrows.\label{lemcase:essdisc-inherited-by-coidentifier}
\end{enumerate}
\end{lem}
\begin{proof}
Using Lemma \ref{lem:ess-discrete}(\ref{lemcase:essdisc-gpd-plus-uniqueness}), together with Lemma \ref{lem:Cat-dpb-along-dopfib-fam} in the case of (\ref{lemcase:essdisc-inherited-by-dpb}), one easily verifies (\ref{lemcase:essdisc-inherited-along-dfib})-(\ref{lemcase:essdisc-inherited-by-dpb}) directly. In the case of (\ref{lemcase:essdisc-inherited-by-coidentifier}) by Lemma \ref{lem:ess-discrete}(\ref{lemcase:essdisc-coproduct-indiscrete}) the coidentifier diagram decomposes as
\[ \xygraph{!{0;(2.5,0):(0,1)::} {\coprod_iA_i}="p0" [r] {\coprod_iB_i}="p1" [r] {\coprod_iQ_i}="p2"
"p0":@<1.5ex>"p1"|(.45){}="t"^-{\coprod_if_i} "p0":@<-1.5ex>"p1"|(.45){}="b"_-{\coprod_ig_i}:"p2"^-{\coprod_iq_i} "t":@{}"b"|(.15){}="d"|(.85){}="c" "d":@{=>}"c"^-{\coprod_i\phi_i}} \]
in which each summand is a reflexive coidentifier diagram, and $B_i$ is indiscrete. Thus by Lemma \ref{lem:ess-discrete}(\ref{lemcase:essdisc-coproduct-indiscrete}) it suffices to consider the case where $B$ is indiscrete. Recall that the functor $(-)_0 : \Cat \to \Set$ which sends every category to its set of objects, has a right adjoint section $\tn{ch} : \Set \to \Cat$, which sends every set $X$ to the category whose objects are the elements of $X$, and where there is a unique arrow between any two objects. Given $x_1$ and $x_2 \in X$, we denote the unique arrow $x_1 \to x_2$ in $\tn{ch}(X)$ simply as $(x_1,x_2)$. When $B$ is indiscrete one may regard it as $\tn{ch}(X)$ for some non-empty set $X$. It suffices to show that in this case the coidentifier is computed as on the left
\[ \xygraph{*=(0,0)!(0,.8){\xybox{\xygraph{{\xybox{\xygraph{!{0;(1.5,0):(0,1)::} {A}="p0" [r] {\tn{ch}(X)}="p1" [r] {\tn{ch}(Q)}="p2"
"p0":@<1.5ex>"p1"|(.4){}="t"^-{f} "p0":@<-1.5ex>"p1"|(.4){}="b"_-{g}:"p2"^-{\tn{ch}(q)} "t":@{}"b"|(.15){}="d"|(.85){}="c" "d":@{=>}"c"^-{\phi}}}}
[r(4.5)]
{\xybox{\xygraph{!{0;(1.5,0):(0,1)::} {A_0}="p0" [r] {X}="p1" [r] {Q}="p2"
"p0":@<1ex>"p1"|(.45){}="t"^-{f_0} "p0":@<-1ex>"p1"|(.45){}="b"_-{g_0}:"p2"^-{q}}}}}}}} \]
where the diagram on the right is a coequaliser in $\Set$. Let $h : \tn{ch}(X) \to C$ be such that $h\phi = \id$. Since $\tn{ch}(q)$ is clearly surjective on objects and arrows and thus an epimorphism in $\Cat$, it suffices to show that there exists $h' : \tn{ch}(Q) \to C$ such that $h = h'\tn{ch}(q)$. The object map $h'_0$ is unique such that $h_0 = h'_0q$ by the coequaliser in $\Set$. Given $y_1$ and $y_2 \in Q$ we must give $h'(y_1,y_2) : h'y_1 \to h'y_2$ in $C$, and this is done by choosing $x_1$ and $x_2$ in $X$ such that $qx_1 = y_1$ and $qx_2 = y_2$, and then defining $h'(y_1,y_2) = h(x_1,x_2)$. The functoriality of $h'$ is immediate from that of $h$ as long as $h'$'s arrow map is well-defined.

To establish this well-definedness we must show that if $qx_1 = qx'_1$ and $qx_2 = qx'_2$, then $h(x_1,x_2) = h(x'_1,x'_2)$. Since $f_0$ and $g_0$ have a common section given by the object map of the 1-section of $\phi$, the equivalence relation on $X \times X$ defined by $(x_1,x_2) \sim (x'_1,x'_2)$ iff $qx_1 = qx'_1$ and $qx_2 = qx'_2$ is the smallest equivalence relation which contains
\[ \begin{array}{c}
{\{((x_1,x_2),(x'_1,x_2)) \, : \, \exists a, \, fa = x_1 \, \tn{and} \, ga = x'_1\}} \\
{\cup} \\
{\{((x_1,x_2),(x_1,x'_2)) \, : \, \exists a, \, fa = x_2 \, \tn{and} \, ga = x'_2\}.}
\end{array} \]
Thus it suffices to show
\begin{itemize}
\item If $x_1$, $x'_1$ and $x_2 \in X$ such that $\exists \, a$, $fa = x_1$ and $ga = x'_1$, then $h(x_1,x_2) = h(x'_1,x_2)$.
\item If $x_1$, $x_2$ and $x'_2 \in X$ such that $\exists \, a$, $fa = x_2$ and $ga = x'_2$, then $h(x_1,x_2) = h(x_1,x'_2)$.
\end{itemize}
In the first of these situations note that one has a triangle as on the left
\[ \xygraph{{\xybox{\xygraph{{x_1}="p0" [r(2)] {x'_1}="p1" [dl] {x_2}="p2" "p0":"p1"^-{\phi_a}:"p2"^-{(x'_1,x_2)}:@{<-}"p0"^-{(x_1,x_2)}}}}
[r(4)]
{\xybox{\xygraph{{x_1}="p0" [dr] {x'_2}="p1" [l(2)] {x_2}="p2" "p0":"p1"^-{(x_1,x'_2)}:@{<-}"p2"^-{\phi_a}:@{<-}"p0"^-{(x_1,x_2)}}}}} \]
in $\tn{ch}(X)$ sent by $h$ to the desired equality since $h\phi_a = \id$, and simlarly one applies $h$ to the triangle on the right for the other situation.
\end{proof}
Recall from Section \ref{sec:SMultiCats-Poly} that
\[ \xygraph{{I}="p0" [r] {E_T}="p1" [r] {B_T}="p2" [r] {I}="p3" "p0":@{<-}"p1"^-{s_T}:"p2"^-{p_T}:"p3"^-{t_T}} \]
denotes the polynomial corresponding to a collection $T$ with object set $I$.
\begin{prop}\label{prop:Sigma-freeness-poly}
For a collection $T$ with object set $I$ the following statements are equivalent:
\begin{enumerate}
\item $T$ is $\Sigma$-free.\label{pcase:sig-free}
\item $B_T$ is equivalent to a discrete category.\label{pcase:BT-ess-disc}
\item The naturality square
\[ \xygraph{!{0;(1.5,0):(0,.6667)::} {E_T}="p0" [r] {B_T}="p1" [d] {\pi_0B_T}="p2" [l] {\pi_0E_T}="p3" "p0":"p1"^-{p_T}:"p2"^-{q_{B_T}}:@{<-}"p3"^-{\pi_0p_T}:@{<-}"p0"^-{q_{E_T}}} \]
is a pullback.\label{pcase:q-cart-at-pT}
\end{enumerate}
\end{prop}
\begin{proof}
(\ref{pcase:sig-free})$\implies$(\ref{pcase:BT-ess-disc}): By Lemma \ref{lem:ess-discrete} it suffices to show that when $T$ is $\Sigma$-free, that every morphism of $B_T$ is unique in its hom-set. Recall that a morphism of $B_T$ is of the form $\rho : \alpha\rho \to \alpha$, where $\alpha : (i_j)_{1{\leq}j{\leq}n} \to i$ is in $T$ and $\rho \in \Sigma_n$. Thus to give a pair of morphisms with the same domain and codomain, is to give $\alpha$ as above, $\rho_1$ and $\rho_2 \in \Sigma_n$ such that $\alpha\rho_1 = \alpha\rho_2$. Thus $\alpha = \alpha(\rho_2\rho_1^{-1})$, and so $\Sigma$-freeness implies $\rho_2\rho_1^{-1} = 1_n$, whence $\rho_1 = \rho_2$.

(\ref{pcase:BT-ess-disc})$\implies$(\ref{pcase:q-cart-at-pT}): By Lemma \ref{lem:ess-disc-inheritance} $E_T$ is also equivalent to a discrete category, and so by Lemma \ref{lem:ess-discrete} $q_{E_T}$ and $q_{B_T}$ are surjective-on-objects equivalences. We check that the naturality square is a pullback on objects. Let $b \in B_T$ and $c \in \pi_0E_T$ such that $q_{B_T}b = \pi_0p_Tc$. Choose $e \in E_T$ such that $q_{E_T}e = c$. Since $p_Te$ and $b$ are in the same connected component, there is a unique isomorphism $p_Te \iso b$, and one has a unique lifting of this to $e \iso e'$ in $E_T$. Thus $e'$ is an object of $E_T$ such that $p_Te' = b$ and $q_{E_T}e'=c$. To see that it is unique, suppose that one has $e_1$ and $e_2$ in $E_T$ such that $q_{E_T}e_1 = q_{E_T}e_2$ and $p_Te_1 = p_Te_2$. By the first of these equations one has a unique isomorphism $e_1 \iso e_2$, and this is sent to an identity by $p_T$. As a discrete fibration $p_T$ reflects identities, and so the isomorphism $e_1 \iso e_2$ is an identity. To say that the naturality square is a pullback on arrows, is to say that given $\beta : b_1 \iso b_2$ in $B_T$ and $c \in \pi_0E_T$ such that $\pi_0p_T1_c = q_{B_T}\beta$, then there is a unique isomorphism $\varepsilon:e_1 \iso e_2$ in $E_T$ such that $q_T\varepsilon = 1_c$ and $p_T\varepsilon = \beta$. But $e_1$ and $e_2$ are determined uniquely since the square is a pullback on objects, and $\varepsilon$ is determined uniquely since $E_T$ and $B_T$ are equivalent to discrete categories.

(\ref{pcase:q-cart-at-pT})$\implies$(\ref{pcase:sig-free}): We prove the contrapositive. Suppose $T$ is not $\Sigma$-free. Then one has $\alpha : (i_j)_{1{\leq}j{\leq}n} \to i$ in $T$ and $1_n \neq \rho \in \Sigma_n$ such that $\alpha\rho = \alpha$. Choose $1 \leq k < l \leq n$ such that $\rho k = l$, for instance by letting $k$ be the least such that $\rho k \neq k$. Then $(\alpha,k)$ and $(\alpha,l)$ are distinct objects of $E_T$, which are in the same connected component since one has $\rho : (\alpha,k) \to (\alpha,l)$, and one has $p_T(\alpha,k) = \alpha = p_T(\alpha,l)$, and so $q_{E_T}$ and $p_T$ are not jointly monic.
\end{proof}
Thanks to this last result, for a $\Sigma$-free collection $T$ one has the 1-cell $(q_{B_T},q_{E_T})$
\begin{equation}\label{diag:local-coidentifier-in-Poly}
\begin{gathered}
\xygraph{!{0;(1.5,0):(0,1)::} {I}="p0" [ur] {E^{[1]}_T}="p1" [r] {B^{[1]}_T}="p2" [dr] {I}="p3" [l] {B_T}="p4" [l] {E_T}="p5" [d(.75)] {\pi_0E_T}="p6" [r] {\pi_0B_T}="p7" "p0":@{<-}"p1"^-{s^{[1]}_T}:"p2"^-{p^{[1]}_T}:"p3"^-{t^{[1]}_T}:@{<-}"p4"^-{t_T}:@{<-}"p5"^-{p_T}:"p0"^-{s_T}
"p1":@/_{1pc}/"p5"_(.5){}|(.5){}="dal" "p2":@/_{1pc}/"p4"_(.5){}|(.5){}="dar"
"p1":@/^{1pc}/"p5"^(.5){}|(.5){}="cal" "p2":@/^{1pc}/"p4"^(.5){}|(.5){}="car"
"dal":@{}"cal"|(.3){}="dl"|(.7){}="cl" "dl":@{=>}"cl"^-{\alpha_{E_T}}
"dar":@{}"car"|(.3){}="dr"|(.7){}="cr" "dr":@{=>}"cr"^-{\alpha_{B_T}}
"p0":@{<-}"p6"_-{\pi_0s_T}:"p7"_-{\pi_0p_T}:"p3"_-{\pi_0t_T}
"p5":"p6"^-{q_{E_T}} "p4":"p7"^-{q_{B_T}}}
\end{gathered}
\end{equation}
in $\Polyc{\Cat}(I,I)$ whose composite with $(\alpha_{B_T},\alpha_{E_T})$ is an identity. Composition of polynomials makes $\Polyc{\Cat}(I,I)$ a monoidal 2-category, and the main technical result of this section is
\begin{lem}\label{lem:monoidally-stable-local-coidentifier}
$(q_{B_T},q_{E_T})$ is the coidentifier of $(\alpha_{B_T},\alpha_{E_T})$ in $\Polyc{\Cat}(I,I)$. This coidentifier is monoidally stable and preserved by
\[ (\PFun{\Cat})_{I,I} : \Polyc{\Cat}(I,I) \longrightarrow \tn{End}(\Cat/I). \]
\end{lem}
\noindent whose proof we defer to Appendix \ref{sec:proof-lem-SigFree-charn}.
\begin{thm}\label{thm:T-mod-Sigma-polynomial}
Let $T$ be a $\Sigma$-free operad with object set $I$. Then
\[ \xygraph{!{0;(1.5,0):(0,.5)::} {I}="p0" [ur] {E_T}="p1" [r] {B_T}="p2" [dr] {I}="p3" [dl] {\pi_0B_T}="p4" [l] {\pi_0E_T}="p5" "p0":@{<-}"p1"^-{s_T}:"p2"^-{p_T}:"p3"^-{t_T}:@{<-}"p4"^-{\pi_0t_T}:@{<-}"p5"^-{\pi_0p_T}:"p0"^-{\pi_0s_T}
"p1":"p5"_-{q_{E_T}} "p2":"p4"^-{q_{B_T}} "p1":@{}"p4"|-{\tn{pb}}} \]
is a morphism of monads in $\Polyc{\Cat}$, which $\PFun{\Cat}$ sends to $q_T$.
\end{thm}
\begin{proof}
Since $\PFun{\Cat}_{I,I}$ preserves the coidentifier (\ref{diag:local-coidentifier-in-Poly}), it may be regarded as sending (\ref{diag:local-coidentifier-in-Poly}) to the defining coidentifier in $\tn{End}(\Cat)$ of $T/\Sigma$. Since (\ref{diag:local-coidentifier-in-Poly}) is monoidally stable, by Proposition \ref{prop:ref-coidentifers-monoids}, the polynomial $(\pi_0s_T,\pi_0p_T,\pi_0t_T)$ acquires a unique 2-monad structure in $\Polyc{\Cat}$ making $(q_{B_T},q_{E_T})$ a morphism of monads. Since $\PFun{\Cat}$ is a homomorphism of 2-bicategories, it sends this 2-monad in $\Polyc{\Cat}$ to a 2-monad on $\Cat/I$, whose underlying endo-2-functor coincides with that of $T/\Sigma$ and which makes $q_T : T \to T/\Sigma$ a morphism of 2-monads. However by Definition \ref{defn:T-mod-Sigma} and Example \ref{exam:coidentifier-of-2monads}, $T/\Sigma$'s 2-monad structure is unique with this last property.
\end{proof}

\section{The Quillen equivalence of $\Algs T$ and $\Algs {T/\Sigma}$ when $T$ is $\Sigma$-free}
\label{sec:QuillenEquiv}

Given a 2-monad $T$ on a finitely complete and cocomplete 2-category $\ca K$, from \cite{Lack-HomotopyAspects2Monads} one has a Quillen model structure on $\Algs T$, in which $f$ is a fibration (resp. a weak equivalence) iff $U^Tf$ is an isofibration (resp. an equivalence) in $\ca K$. We shall call this the \emph{Lack model structure} on $\Algs T$. Given a morphism $\phi : S \to T$ of 2-monads on $\ca K$, the induced forgetful functor $\overline{\phi} : \Algs T \to \Algs S$ preserves fibrations and weak equivalences, and so if a left adjoint to $\overline{\phi}$ exists, then $\overline{\phi}$ is a right Quillen functor. In particular this applies to $\phi = q_T$ above, and so $C_T \ladj \overline{q}_T$ is a Quillen adjunction with respect to the Lack model structures on $\Algs T$ and $\Algs {T/\Sigma}$. In the main result of this section, Theorem \ref{thm:Sigma-free-Quillen-equivalence}, we establish that when $T$ is $\Sigma$-free, this Quillen adjunction is in fact a Quillen equivalence.

As explained at the end of Section \ref{sec:comm-algebra}, for any operad $T$ with object set $I$, the forgetful 2-functor $\overline{q}_T : \Algs {T/\Sigma} \to \Algs T$ induced by the morphism $q_T : T \to T/\Sigma$ of 2-monads, is the inclusion of the commutative $T$-algebras. Since the 2-category $\Algs T$ is cocomplete, this inclusion has a left adjoint
\[ C_T : \Algs T \longrightarrow \Algs {T/\Sigma}. \]
Since $\overline{q}_T$ is fully faithful, the counit of $C_T \ladj \overline{q}_T$ may be taken to be an identity. We will adopt the usual abuse of notation of regarding $\overline{q}_T$ as an inclusion, and writing $C_T$ also for the monad on $\Algs T$ coming from the adjunction $C_T \ladj \overline{q}_T$. Let us now embark on the discussion which will culminate in Proposition \ref{prop:explicit-adjunction-C_T}, in an explicit description of the unit of $C_T \ladj \overline{q}_T$ in terms of reflexive coidentifiers in $\Algs T$.

Recall from Definition \ref{defn:T-mod-Sigma}, that $q_T$ was obtained as the result of coidentifying the 2-monad 2-cell on the right
\begin{equation}\label{eq:2cell-for-q_T-and-its-polynomial}
\begin{gathered}
\xygraph{*{\xybox{\xygraph{!{0;(1.5,0):(0,.6)::} {I}="p0" [ur] {E^{[1]}_T}="p1" [r] {B^{[1]}_T}="p2" [dr] {I}="p3" [dl] {B_T}="p4" [l] {E_T}="p5" "p0":@{<-}"p1"^-{s^{[1]}_T}:"p2"^-{p^{[1]}_T}:"p3"^-{t^{[1]}_T}:@{<-}"p4"^-{t_T}:@{<-}"p5"^-{p_T}:"p0"^-{s_T}
"p1":@/_{1.25pc}/"p5"_(.425){d_{E_T}}|(.425){}="dal" "p2":@/_{1.25pc}/"p4"_(.425){d_{B_T}}|(.425){}="dar"
"p1":@/^{1.25pc}/"p5"^(.575){c_{E_T}}|(.575){}="cal" "p2":@/^{1.25pc}/"p4"^(.575){c_{E_T}}|(.575){}="car"
"dal":@{}"cal"|(.35){}="dl"|(.65){}="cl" "dl":@{=>}"cl"^-{\alpha_{E_T}}
"dar":@{}"car"|(.35){}="dr"|(.65){}="cr" "dr":@{=>}"cr"^-{\alpha_{B_T}}}}}
[r(5.5)]
*!(0,.58){\xybox{\xygraph{!{0;(1.5,0):(0,1)::} {T^{[1]}_{\Sigma}}="p0" [r] {T}="p1" "p0":@<1.5ex>"p1"|(.45){}="t"^-{d_T} "p0":@<-1.5ex>"p1"|(.45){}="b"_-{c_T} "t":@{}"b"|(.15){}="d"|(.85){}="c" "d":@{=>}"c"^-{\alpha_T}}}}}
\end{gathered}
\end{equation}
which itself is the result of applying $\PFun{\Cat}$ to the 2-cell in $\Polyc{\Cat}(I,I)$ on the left.
\begin{constn}\label{constn:left-right-T-modules}
We will now exhibit left and right module structures of the 2-monad $T$ on the endo-2-functor $T^{[1]}_{\Sigma}$, with respect to which $d_T$, $c_T$ and $\alpha_T$ are compatible. Pre and post composition with the polynomial $(s_T,p_T,t_T)$ is the effect on objects of endo-2-functors
\[ \begin{array}{c}
{(-) \comp (s_T,p_T,t_T) : \Polyc{\Cat}(I,I) \longrightarrow \Polyc{\Cat}(I,I)} \\
{(s_T,p_T,t_T) \comp (-) : \Polyc{\Cat}(I,I) \longrightarrow \Polyc{\Cat}(I,I)}
\end{array} \]
and since $(s_T,p_T,t_T)$ underlies a monad in $\Polyc{\Cat}$, these endo-2-functors underlie 2-monads on $\Polyc{\Cat}(I,I)$ which we denote as $T_R$ and $T_L$ respectively. Note that the forgetful 2-functors
\[ \begin{array}{lccr} {\Algs {T_L} \longrightarrow \Polyc{\Cat}(I,I)} &&&
{\Algs {T_R} \longrightarrow \Polyc{\Cat}(I,I)} \end{array} \]
create cotensors, and the left diagram of (\ref{eq:2cell-for-q_T-and-its-polynomial}) exhibits the polynomial $(s^{[1]}_T,p^{[1]}_T,t^{[1]}_T)$ as the cotensor of $[1]$ and $(s_T,p_T,t_T)$ in $\Polyc{\Cat}(I,I)$. Thus $(s^{[1]}_T,p^{[1]}_T,t^{[1]}_T)$ gets the structure of strict $T_L$-algebra (resp. strict $T_R$-algebra) making the left diagram of (\ref{eq:2cell-for-q_T-and-its-polynomial}) a 2-cell in $\Algs {T_L}$ (resp. $\Algs {T_R}$). The $T_L$ and $T_R$-actions on $(s^{[1]}_T,p^{[1]}_T,t^{[1]}_T)$ are 2-cells
\[ \begin{array}{c} {(s_T,p_T,t_T) \comp (s^{[1]}_T,p^{[1]}_T,t^{[1]}_T) \longrightarrow (s^{[1]}_T,p^{[1]}_T,t^{[1]}_T)} \\ {(s^{[1]}_T,p^{[1]}_T,t^{[1]}_T) \comp (s_T,p_T,t_T) \longrightarrow (s^{[1]}_T,p^{[1]}_T,t^{[1]}_T)} \end{array} \]
in $\Polyc{\Cat}$, which are sent by $\PFun{\Cat}$ to the left and right $T$-module structures
\[ \begin{array}{lccr} {a_L : TT^{[1]}_{\Sigma} \longrightarrow T^{[1]}_{\Sigma}} &&&
{a_R : T^{[1]}_{\Sigma}T \longrightarrow T^{[1]}_{\Sigma}.} \end{array} \]
\end{constn}
\begin{rem}\label{rem:left-right-T-modules-further-explanation}
The compatibility of $a_L$ and $a_R$ with $d_T$, $c_T$ and $\alpha_T$ says exactly that
\[ \xygraph{{\xybox{\xygraph{!{0;(2.5,0):(0,.5)::}
% top row
{TT^{[1]}_{\Sigma}}="p0" [r] {T^2}="p1"
"p0":@<1.5ex>"p1"|(.45){}="t1"^-{Td_T}
"p0":@<-1.5ex>"p1"|(.45){}="b1"_-{Tc_T} "t1":@{}"b1"|(.15){}="d1"|(.85){}="c1" "d1":@{=>}"c1"^-{T\alpha_T}
% bottom row
"p0" [d] {T^{[1]}_{\Sigma}}="q0" [r] {T}="q1"
"q0":@<1.5ex>"q1"|(.45){}="t2"^-{d_T} "q0":@<-1.5ex>"q1"|(.45){}="b2"_-{c_T} "t2":@{}"b2"|(.15){}="d2"|(.85){}="c2" "d2":@{=>}"c2"^-{\alpha_T}
% vertical arrows
"p0":"q0"_-{a_L} "p1":"q1"^-{\mu^T}}}}
[r(5)]
{\xybox{\xygraph{!{0;(2.5,0):(0,.5)::}
% top row
{T^{[1]}_{\Sigma}T}="p0" [r] {T^2}="p1"
"p0":@<1.5ex>"p1"|(.45){}="t1"^-{d_TT}
"p0":@<-1.5ex>"p1"|(.45){}="b1"_-{c_TT} "t1":@{}"b1"|(.15){}="d1"|(.85){}="c1" "d1":@{=>}"c1"^-{\alpha_TT}
% bottom row
"p0" [d] {T^{[1]}_{\Sigma}}="q0" [r] {T}="q1"
"q0":@<1.5ex>"q1"|(.45){}="t2"^-{d_T} "q0":@<-1.5ex>"q1"|(.45){}="b2"_-{c_T} "t2":@{}"b2"|(.15){}="d2"|(.85){}="c2" "d2":@{=>}"c2"^-{\alpha_T}
% vertical arrows
"p0":"q0"_-{a_R} "p1":"q1"^-{\mu^T}}}}} \]
commutes, that is to say, that $\mu^T(T\alpha_T) = \alpha_Ta_L$ and $\mu^T(\alpha_TT) = \alpha_Ta_R$. In particular evaluating the compatibility of the left action at $X \in \Cat/I$, one has a strict $T$-algebra structure $a_{L,X} : TT^{[1]}_{\Sigma}X \to T^{[1]}_{\Sigma}X$, with respect to which $d_{T,X}$ and $c_{T,X}$ are strict $T$-algebra morphisms, and $\alpha_{T,X}$ is an algebra 2-cell.
\end{rem}
\begin{rem}\label{rem:T-mod-sigma-X-alg-structure}
Since $T$ preserves all sifted colimits by \cite{Weber-PolynomialFunctors} Theorem 4.5.1, the forgetful 2-functor $U^T : \Algs T \to \Cat/I$ creates reflexive coidentifiers. Thus by Remark \ref{rem:left-right-T-modules-further-explanation}, for all $X \in \Cat/I$, $T/\Sigma(X)$ gets a strict $T$-algebra structure $a_X : TT/\Sigma(X) \to T/\Sigma(X)$ making $q_{T,X} : TX \to T/\Sigma(X)$ a strict $T$-algebra morphism. On the other hand,
\[ \overline{q}_{T}(T/\Sigma(X),\mu^{T/\Sigma}_X) = (T/\Sigma(X),\mu^{T/\Sigma}_Xq_{T,T/\Sigma(X)}) \]
is also a strict $T$-algebra structure whose underlying object in $\Cat/I$ is $T/\Sigma(X)$. To see that these agree, note that $a_X$ is the unique morphism making the square on the left
\[ \xygraph{{\xybox{\xygraph{!{0;(2,0):(0,.5)::} {T^2X}="p0" [r] {TT/\Sigma(X)}="p1" [d] {T/\Sigma(X)}="p2" [l] {TX}="p3" "p0":"p1"^-{Tq_{T,X}}:"p2"^-{a_X}:@{<-}"p3"^-{q_{T,X}}:@{<-}"p0"^-{\mu^T_X}}}}
[r(6)]
{\xybox{\xygraph{!{0;(2,0):(0,.5)::} {T^2X}="p0" [r] {TT/\Sigma(X)}="p1" [r(1.5)] {(T/\Sigma)^2X}="p2" [l(1.5)d] {T/\Sigma(X)}="p3" [l] {TX}="p4" "p0":"p1"^-{Tq_{T,X}}:"p2"^-{q_{T,T/\Sigma(X)}}:"p3"^-{\mu^{T/\Sigma}_X}:@{<-}"p4"^-{q_{T,X}}:@{<-}"p0"^-{\mu^T_X}}}}} \]
commute in $\Cat/I$. Since $q_T$ is a morphism of monads the diagram on the right above commutes, and so $a_X = \mu^{T/\Sigma}_Xq_{T,T/\Sigma(X)}$.
\end{rem}
\begin{lem}\label{lem:comm-T-algebras-charn}
Let $T$ be an operad with set of colours $I$ and $(X,x)$ be a strict $T$-algebra. Then $(X,x)$ is commutative iff $x\alpha_{T,X}$ is an identity.
\end{lem}
\begin{proof}
Let $(X,x)$ be a strict $T$-algebra. Then to say that $(X,x)$ is commutative is to say that the action $x$ factors as $x = x'q_{T,X}$ where $x'$ is a $T/\Sigma$-algebra structure. Thus $x\alpha_{T,X}$ is an identity. Conversely if $x\alpha_{T,X}$ is an identity, then one induces $x':T/\Sigma(X) \to X$ in $\Algs T$ unique such that $x'q_{T,X} = x$ by the universal property of $q_{T,X}$. Since $q_T$ is a morphism of 2-monads, $q_{T,X}\eta^T_X = \eta^{T/\Sigma}_X$, thus $x'\eta^{T/\Sigma}_X = x\eta^T_X = 1_X$, and so $x'$ satisfies the unit axiom for a strict $T/\Sigma$-algebra. The associative law for $x'$ is the bottom right square in
\[ \xygraph{!{0;(3,0):(0,.3333)::} {T^2X}="p0" [r(2)] {TX}="p1" [d] {T/\Sigma(X)}="p2" [d] {X}="p3" [l] {T/\Sigma(X)}="p4" [l] {TX}="p5" [u] {TT/\Sigma(X)}="p6" [r] {(T/\Sigma)^2(X)}="p7" "p0":"p1"^-{\mu^T_X}:"p2"^-{q_{T,X}}:"p3"^-{x'}:@{<-}"p4"^-{x'}:@{<-}"p5"^-{q_{T,X}}:@{<-}"p6"^-{Tx'}(:@{<-}"p0"^-{Tq_{T,X}},:"p7"^-{q_{T,T/\Sigma(X)}}(:"p4"^-{T/\Sigma(x')},:"p2"^-{\mu^{T/\Sigma}_X}))} \]
the outside of which commutes by the associative law for $x$, the top region commutes since $q_T$ is a morphism of 2-monads, and the bottom left square commutes by the naturality of $q_T$. Since $T$ preserves sifted colimits, $Tq_{T,X}$ is a coidentifying map and thus an epimorphism, thus $q_{T,T/\Sigma(X)}T(q_{T,X})$ is an epimorphism, so that $x'$'s associative law follows.
\end{proof}
\begin{defn}\label{defn:C_TX}
Let $T$ be an operad with set of colours $I$ and $(X,x)$ be a strict $T$-algebra. Then the morphism $r_X$ of $\Cat/I$ is defined to be the coidentifier
\[ \xygraph{!{0;(2,0):(0,1)::} {T^{[1]}_{\Sigma}X}="p0" [r] {X}="p1" [r] {C_TX.}="p2" "p0":@<1.5ex>"p1"|(.45){}="t"^-{xd_{T,X}}:"p2"^-{r_X} "p0":@<-1.5ex>"p1"|(.45){}="b"_-{xc_{T,X}} "t":@{}"b"|(.15){}="d"|(.85){}="c" "d":@{=>}"c"^-{x\alpha_{T,X}}} \]
\end{defn}
As explained in Section \ref{sec:comm-algebra}, $\alpha_{T,X}$ has a 1-section. Since $x$ has section $\eta^T_X$, the coidentifier of Definition \ref{defn:C_TX} is reflexive, and since $T$ preserves sifted colimits, this coidentifier is created by the forgetful 2-functor $U^T : \Algs T \to \Cat/I$. Hence $C_TX$ acquires a unique strict $T$-algebra structure $x' : TC_TX \to C_TX$ making $r_X$ a strict $T$-algebra morphism.
\begin{prop}\label{prop:explicit-adjunction-C_T}
In the context of Definition \ref{defn:C_TX}
\begin{enumerate}
\item $r_X$ is the component at $(X,x)$ of the unit of $C_T \ladj \overline{q}_T$.
\label{propcase:unit-C_T-adjunction}
\item If $(X,x)$ is a free $T$-algebra $(TZ,\mu^T_Z)$, then $r_{TZ} = q_{T,Z}$.
\label{propcase:unit-C_T-adjunction-free-case}
\end{enumerate}
\end{prop}
\begin{proof}
(\ref{propcase:unit-C_T-adjunction}): By the commutativity of 
\[ \xygraph{!{0;(3,0):(0,.5)::}
% top row
{T^{[1]}_{\Sigma}X}="p0" [r] {TX}="p1" [r] {X}="p2" "p0":@<1.5ex>"p1"|(.45){}="t1"^-{d_{T,X}}:"p2"^-{x} "p0":@<-1.5ex>"p1"|(.45){}="b1"_-{c_{T,X}} "t1":@{}"b1"|(.15){}="d1"|(.85){}="c1" "d1":@{=>}"c1"^-{\alpha_{T,X}}
% bottom row
"p0" [d] {T^{[1]}_{\Sigma}C_TX}="q0" [r] {TC_TX}="q1" [r] {C_TX}="q2" "q0":@<1.5ex>"q1"|(.45){}="t2"^-{d_{T,C_TX}}:"q2"_-{x'} "q0":@<-1.5ex>"q1"|(.45){}="b2"_-{c_{T,C_TX}} "t2":@{}"b2"|(.15){}="d2"|(.85){}="c2" "d2":@{=>}"c2"^-{\alpha_{T,C_TX}}
% vertical arrows
"p0":"q0"_-{T^{[1]}_{\Sigma}r_X} "p1":"q1"_-{Tr_X} "p2":"q2"^-{r_X}} \]
and since $r_Xx\alpha_{T,X}$ is an identity by definition, $x'\alpha_{T,C_TX}T^{[1]}_{\Sigma}(r_X)$ is an identity. By Theorem 4.5.1 of \cite{Weber-PolynomialFunctors}, $T^{[1]}_{\Sigma}$ preserves sifted colimits and thus reflexive coidentifiers in particular. Thus $T^{[1]}_{\Sigma}(r_X)$ is a coidentifiying map and so an epimorphism. Thus $x'\alpha_{T,C_TX}$ is an identity, and so $(C_TX,x')$ is commutative by Lemma \ref{lem:comm-T-algebras-charn}. Given any other commutative strict $T$-algebra $(Y,y)$, and a strict morphism $f : X \to Y$ of strict $T$-algebras, the commutativity of
\[ \xygraph{!{0;(3,0):(0,.5)::}
% top row
{T^{[1]}_{\Sigma}X}="p0" [r] {TX}="p1" [r] {X}="p2" "p0":@<1.5ex>"p1"|(.45){}="t1"^-{d_{T,X}}:"p2"^-{x} "p0":@<-1.5ex>"p1"|(.45){}="b1"_-{c_{T,X}} "t1":@{}"b1"|(.15){}="d1"|(.85){}="c1" "d1":@{=>}"c1"^-{\alpha_{T,X}}
% bottom row
"p0" [d] {T^{[1]}_{\Sigma}Y}="q0" [r] {TY}="q1" [r] {Y}="q2" "q0":@<1.5ex>"q1"|(.45){}="t2"^-{d_{T,Y}}:"q2"_-{y} "q0":@<-1.5ex>"q1"|(.45){}="b2"_-{c_{T,Y}} "t2":@{}"b2"|(.15){}="d2"|(.85){}="c2" "d2":@{=>}"c2"^-{\alpha_{T,Y}}
% vertical arrows
"p0":"q0"_-{T^{[1]}_{\Sigma}f} "p1":"q1"_-{Tf} "p2":"q2"^-{f}} \]
ensures that $fx\alpha_{T,X}$ is an identity since $y\alpha_{T,Y}$ is by Lemma \ref{lem:comm-T-algebras-charn}, and so there is a unique morphism $f' : C_TX \to Y$ of $\Algs T$ such that $f'r_X = f$.

(\ref{propcase:unit-C_T-adjunction-free-case}): From Remark \ref{rem:left-right-T-modules-further-explanation} one has $\alpha_{T,Z}a_{R,Z} = \mu^T_Z\alpha_{T,TZ}$ in
\[ \xygraph{!{0;(3,0):(0,.5)::}
% top row
{T^{[1]}_{\Sigma}TZ}="p0" [r] {T^2Z}="p1"
"p0":@<1.5ex>"p1"|(.45){}="t1"^-{d_{T,TZ}}
"p0":@<-1.5ex>"p1"|(.45){}="b1"_-{c_{T,TZ}} "t1":@{}"b1"|(.15){}="d1"|(.85){}="c1" "d1":@{=>}"c1"^-{\alpha_{T,TZ}}
% bottom row
"p0" [d] {T^{[1]}_{\Sigma}Z}="q0" [r] {TZ}="q1" [r] {W.}="q2" "q0":@<1.5ex>"q1"|(.45){}="t2"^-{d_{T,Z}}:"q2"_-{g} "q0":@<-1.5ex>"q1"|(.45){}="b2"_-{c_{T,Z}} "t2":@{}"b2"|(.15){}="d2"|(.85){}="c2" "d2":@{=>}"c2"^-{\alpha_{T,Z}}
% vertical arrows
"p0":"q0"_-{a_{R,Z}} "p1":"q1"_-{\mu^T_Z}} \]
For any morphism $g$ as indicated in the diagram, to say that $g\mu^T_Z\alpha_{T,TZ} = \id$, is to say that $g\alpha_{T,Z}a_{R,Z} = \id$, which in turn is equivalent to $g\alpha_{T,Z} = \id$ since $a_{R,Z}$ is an epimorphism as it has a section $T^{[1]}_{\Sigma}\eta^T_Z$. Thus $\mu^T_Z\alpha_{T,TZ}$ and $\alpha_{T,Z}$ have the same coidentifier.
\end{proof}
As we shall now see, this last result together with Theorem \ref{thm:T-mod-Sigma-polynomial} enables us to witness $\overline{q}_T$ as part of a Quillen equivalence when $T$ is $\Sigma$-free, with respect to the Lack model structures. The key to the proof of Theorem \ref{thm:Sigma-free-Quillen-equivalence} is that thanks to the developments of \cite{Lack-HomotopyAspects2Monads}, one has a cofibrant replacement which behaves well with respect to reflexive coidentifiers. We now recall the necessary background.

From \cite{Lack-HomotopyAspects2Monads}, the comonad corresponding to the adjunction
\[ \xygraph{!{0;(3,0):(0,1)::} {\Algs T}="p0" [r] {\Alg T}="p1" "p0":@<-1ex>"p1"_-{J}|-{}="b":@<-1ex>"p0"_-{Q}|-{}="t" "t":@{}"b"|-{\perp}} \]
in which $J$ is the inclusion, gives a cofibrant replacement for $\Algs T$, when $\Algs T$ possesses sufficient colimits for $Q$ to exist. This adjunction is fundamental to 2-dimensional monad theory \cite{BWellKellyPower-2DMndThy}. Writing $J$ as an inclusion, the component of the counit of $Q \ladj J$ at $X$ is a strict $T$-algebra morphism $q_X : QX \to X$, and the unit is a pseudo morphism $p_X : X \to QX$. From Theorem 4.2 of \cite{BWellKellyPower-2DMndThy}, $p_X \ladj q_X$ in $\Alg{T}$ with unit an identity and counit invertible. Thus in particular, $q_X$ is a weak equivalence in $\Algs T$. Theorem 4.12 of \cite{Lack-Codescent} says that the cofibrant objects are exactly the flexible algebras in the sense of \cite{BWellKellyPower-2DMndThy}, and that $QX$ is flexible for all $(X,x)$. Recall $(X,x)$ is \emph{flexible} when $q_X : QX \to X$ admits a section in $\Algs T$ ($p_X$ is only in $\Alg T$ in general).

From \cite{Lack-Codescent} $Q$ can be computed explicitly in terms of isocodescent objects in $\Algs T$. Formally, an \emph{isocodescent object} is a colimit weighted by the functor $IL\delta : \Delta \to \Cat$, where $\delta : \Delta \to \Cat$ is the standard inclusion which regards the ordinals as categories, $I : \tnb{Gpd} \to \Cat$ is the inclusion of groupoids, and $L$ is the left adjoint of $I$. In more explicit terms, a \emph{cocone} for a simplicial object $X : \Delta^{\op} \to \ca K$ in a 2-category $\ca K$ with vertex $Y$, amounts to a pair $(f_0,f_1)$, where $f_0 : X_0 \to Y$, $f_1$ is an invertible 2-cell $f_0d_1 \to f_0d_0$, and these satisfy $f_1s_0 = 1_{f_0}$ and $(f_1d_0)(f_1d_2) = f_1d_1$ (using the usual notation for face and degeneracy maps of $X$). When $\ca K$ admits cotensors with $[1]$, the cocone $(f_0,f_1)$ exhibits $Y$ as the isocodescent object $\tn{colim}(IL\delta,X)$ of $X$, when it satisfies a 1-dimensional universal property, which says: for any cocone $(g_0,g_1)$ with vertex $Z$, there is a unique $g':Y \to Z$ such that $g'f_0 = g_0$ and $g'f_1 = g_1$.

For $X \in \Algs T$, the bar construction of $X$ is the simplicial object $\ca R_TX : \Delta^{\op} \to \Algs T$ whose 2-truncation is 
\[ \xygraph{!{0;(2,0):(0,1)::} {T^3X}="p0" [r] {T^2X}="p1" [r] {TX}="p2"
"p2":"p1"|-{T\eta^T_X} "p1":@<1.5ex>"p2"^-{\mu^T_X} "p1":@<-1.5ex>"p2"_-{Tx} "p0":@<1.5ex>"p1"^-{\mu^T_{TX}} "p0":"p1"|-{T\mu^T_X} "p0":@<-1.5ex>"p1"_-{T^2x}} \]
and then $QX$ is its isocodescent object. This is described in detail in Remark 6.9 of \cite{Bourke-Thesis}. We denote by $c_X : TX \to QX$ the 1-cell datum of the isocodescent cocone. What makes this cofibrant replacement convenient for us is that isocodescent objects are both sifted and flexible colimits. Their siftedness was established in Proposition 4.38 of \cite{Bourke-Thesis}, and as explained in Remark 2.18 of \cite{Bourke-Thesis}, they can be constructed from coinserters and coequifiers, and so are flexible by \cite{BKPS-FlexibleLimits}. Intuitively, flexible colimits are those 2-categorical colimits which are homotopically well-behaved.
\begin{thm}\label{thm:Sigma-free-Quillen-equivalence}
If $T$ is a $\Sigma$-free operad, then with respect to the Lack model structures on $\Algs T$ and $\Algs {T/\Sigma}$,  $C_T \ladj \overline{q}_T$ is a Quillen equivalence.
\end{thm}
\begin{proof}
In this proof we use the term \emph{surjective equivalence} for a morphism $e : A \to B$ of a 2-category which has a pseudo-inverse section. It suffices to show that for all cofibrant objects $(X,x) \in \Algs T$, $r_X$ is a weak equivalence in $\Algs T$. In other words, we must show that for all flexible $(X,x)$, $U^Tr_X$ is an equivalence in $\Cat/I$.

Since $T$ is $\Sigma$-free, the monad morphism $q_T$ is the result of applying $\PFun{\Cat}$ to
\[ \xygraph{!{0;(1.5,0):(0,.5)::} {I}="p0" [ur] {E_T}="p1" [r] {B_T}="p2" [dr] {I}="p3" [dl] {\pi_0B_T}="p4" [l] {\pi_0E_T}="p5" "p0":@{<-}"p1"^-{s_T}:"p2"^-{p_T}:"p3"^-{t_T}:@{<-}"p4"^-{\pi_0t_T}:@{<-}"p5"^-{\pi_0p_T}:"p0"^-{\pi_0s_T}
"p1":"p5"_-{q_{E_T}} "p2":"p4"^-{q_{B_T}} "p1":@{}"p4"|-{\tn{pb}}} \]
by Theorem \ref{thm:T-mod-Sigma-polynomial}. Since $B_T$ is equivalent to a discrete category, any section $i : \pi_0B_T \to B_T$ is a pseudo-inverse section of $q_{B_T}$. Defining $i' : \pi_0E_T \to E_T$ as the unique functor such that $q_{E_T}i' = 1_{\pi_0E_T}$ and $p_Ti' = i\pi_0p_T$, $(i',i)$ is a pseudo-inverse section of $(q_{E_T},q_{B_T})$ in the 2-category $\Polyc{\Cat}(I,I)$. Applying $\PFun{\Cat}$ to this, one obtains a pseudo-inverse section $T/\Sigma \to T$ in $\tn{End}(\Cat/I)$ to $q_T$. Thus by Proposition \ref{prop:explicit-adjunction-C_T}(\ref{propcase:unit-C_T-adjunction-free-case}), $U^Tr_X$ is a surjective equivalence when $X$ is free.

For a general $(X,x)$ we will now show that $U^Tr_{QX}$ is an equivalence. The top left $3 \times 3$ part of
\[ \xygraph{!{0;(3,0):(0,.5)::}
% 1st row
{T^{[1]}_{\Sigma}T^3X}="p11" [r] {T^{[1]}_{\Sigma}T^2X}="p12" [r] {T^{[1]}_{\Sigma}TX}="p13" [r] {T^{[1]}_{\Sigma}QX}="p14"
"p13":"p14"^-{T^{[1]}_{\Sigma}c_X} "p13":"p12"|-{T^{[1]}_{\Sigma}T\eta^T_X} "p12":@<1.5ex>"p13"^-{T^{[1]}_{\Sigma}\mu^T_X} "p12":@<-1.5ex>"p13"_-{T^{[1]}_{\Sigma}Tx} "p11":@<1.5ex>"p12"^-{T^{[1]}_{\Sigma}\mu^T_{TX}} "p11":"p12"|-{T^{[1]}_{\Sigma}T\mu^T_X} "p11":@<-1.5ex>"p12"_-{T^{[1]}_{\Sigma}T^2x}
% 2nd row
"p11" [d] {T^4X}="p21" [r] {T^3X}="p22" [r] {T^2X}="p23" [r] {TQX}="p24"
"p23":"p24"^-{Tc_X} "p23":"p22"|-{T^2\eta^T_X} "p22":@<1.5ex>"p23"^-{T\mu^T_X} "p22":@<-1.5ex>"p23"_-{T^2x} "p21":@<1.5ex>"p22"^-{T\mu^T_{TX}} "p21":"p22"|-{T^2\mu^T_X} "p21":@<-1.5ex>"p22"_-{T^3x}
% 3rd row
"p21" [d] {T^3X}="p31" [r] {T^2X}="p32" [r] {TX}="p33" [r] {QX}="p34"
"p33":"p34"^-{c_X} "p33":"p32"|-{T\eta^T_X} "p32":@<1.5ex>"p33"^-{\mu^T_X} "p32":@<-1.5ex>"p33"_-{Tx} "p31":@<1.5ex>"p32"^-{\mu^T_{TX}} "p31":"p32"|-{T\mu^T_X} "p31":@<-1.5ex>"p32"_-{T^2x}
% 4th row
"p31" [d] {T/\Sigma(T^2X)}="p41" [r] {T/\Sigma(TX)}="p42" [r] {T/\Sigma(X)}="p43" [r] {C_TQX}="p44"
"p43":"p44" "p43":"p42" "p42":@<1.5ex>"p43" "p42":@<-1.5ex>"p43" "p41":@<1.5ex>"p42" "p41":"p42" "p41":@<-1.5ex>"p42"
% 1st column
"p11":@<2ex>"p21"|-{}="pc1" "p11":@<-2ex>"p21"|-{}="pd1":"p31"_-{\mu^T_{T^2X}}:"p41"_-{q_{T,T^2X}}
"pd1":@{}"pc1"|(.25){}="d1"|(.75){}="c1" "d1":@{=>}"c1"^-{\alpha_T}
% 2nd column
"p12":@<2ex>"p22"|-{}="pc2" "p12":@<-2ex>"p22"|-{}="pd2":"p32"_-{\mu^T_{TX}}:"p42"_-{q_{T,TX}}
"pd2":@{}"pc2"|(.25){}="d2"|(.75){}="c2" "d2":@{=>}"c2"^-{\alpha_T}
% 3rd column
"p13":@<2ex>"p23"|-{}="pc3" "p13":@<-2ex>"p23"|-{}="pd3":"p33"_-{\mu^T_X}:"p43"_-{q_{T,X}}
"pd3":@{}"pc3"|(.25){}="d3"|(.75){}="c3" "d3":@{=>}"c3"^-{\alpha_T}
% 4th column
"p14":@<2ex>"p24"|-{}="pc4" "p14":@<-2ex>"p24"|-{}="pd4":"p34"^-{a}:"p44"^-{r_{QX}}
"pd4":@{}"pc4"|(.25){}="d4"|(.75){}="c4" "d4":@{=>}"c4"^-{\alpha_T}} \]
is serially commutative by naturality and the monad axioms of $T$, and the third row exhibits $QX$ as an isocodescent object as explained above. Since $T$ and $T^{[1]}_{\Sigma}$ are sifted colimit preserving, the first two rows are also isocodescent objects, and $a$ is the $T$-algebra structure of $QX$. The first three columns are reflexive coidentifiers by Proposition \ref{prop:explicit-adjunction-C_T}(\ref{propcase:unit-C_T-adjunction-free-case}), and fourth column is too by the definition of $r_{QX}$.

All these colimits are preserved and reflected by $U^T$ since they are sifted, and $T$ preserves all sifted colimits. Regarding the bottom two rows of the above diagram as isocodescent objects in $\Cat/I$, $U^Tr_X$ is the result of applying 
\[ \tn{colim}(IL\delta,-) : [\Delta^{\op},\Cat/I] \longrightarrow \Cat/I \]
to a morphism whose components are surjective equivalences by the free case. Since the weight $IL\delta$ for isocodescent objects is flexible, $\tn{colim}(IL\delta,-)$ sends componentwise surjective equivalences to surjective equivalences by Theorem 21(b) of \cite{BourkeGarner-PieAlgebras}. Thus $U^Tr_{QX}$ is a surjective equivalence.

When $(X,x)$ is flexible, $q_X$ has a section in $\Algs T$, and by Theorem 4.2 of \cite{BWellKellyPower-2DMndThy}, any such section is a pseudo inverse of $q_X$ in $\Algs T$. Thus $C_Tq_X$ is a surjective equivalence. In the naturality square
\[ \xygraph{!{0;(2,0):(0,.5)::} {QX}="p0" [r] {X}="p1" [d] {C_TX}="p2" [l] {C_TQX}="p3" "p0":"p1"^-{q_X}:"p2"^-{r_X}:@{<-}"p3"^-{C_Tq_X}:@{<-}"p0"^-{r_{QX}}} \]
$q_X$, $C_Tq_X$ and $r_{QX}$ are thus weak equivalences, and so the result follows from the 2 out of 3 property.
\end{proof}

\appendix
\section{Proof of Lemma \ref{lem:monoidally-stable-local-coidentifier}}
\label{sec:proof-lem-SigFree-charn}

In this appendix we supply the proof of Lemma \ref{lem:monoidally-stable-local-coidentifier}. This will be achieved below using Lemmas \ref{lem:composition-DI} and \ref{lem:general-coidentifier-stability}. Denote by $\mathfrak{E}_I$ the full sub-2-category of $\Polyc{\Cat}(I,I)$ consisting of those polynomials
\[ \xygraph{{I}="p0" [r] {E}="p1" [r] {B}="p2" [r] {I}="p3" "p0":@{<-}"p1"^-{s}:"p2"^-{p}:"p3"^-{t}} \]
such that $p$ is a discrete fibration with finite fibres, and $B$ is equivalent to a discrete category. By Lemma \ref{lem:ess-disc-inheritance} $E$ is also equivalent to a discrete category.
\begin{lem}\label{lem:composition-DI}
A composite of polynomials in $\mathfrak{E}_I$ is in $\mathfrak{E}_I$.
\end{lem}
\begin{proof}
We suppose that $(s_1,p_1,t_1)$ and $(s_2,p_2,t_2)$ are in $\mathfrak{E}_I$ and form their composite $(s_3,p_3,t_3)$ as in
\[ \xygraph{{I}="b1" [r] {E_1}="b2" [r] {B_1}="b3" [r] {I}="b4" [r] {E_2}="b5" [r] {B_2}="b6" [r] {I.}="b7" "b4" [u] {P}="p1" [u] {F}="dl" ([r(1.5)] {B_3}="dr", [l(1.5)] {E_3}="p2")
"b1":@{<-}"b2"_-{s_1}:"b3"_-{p_1}:"b4"_-{t_1}:@{<-}"b5"_-{s_2}:"b6"_-{p_2}:"b7"_-{t_2} "dl":"p1"_-{}(:"b3"_-{},:"b5"^-{}) "b2":@{<-}"p2"_-{}:"dl"_-{}:"dr"_-{}:"b6"^(.7){} "b1":@{<-}"p2"^-{s_3} "dr":"b7"^-{t_3} "p2":@/^{1pc}/"dr"^-{p_3}
"b3" [u(1.25)] {\scriptstyle{\tn{pb}}} "b5" [u(1.25)] {\scriptstyle{\tn{dpb}}} "b4" [u(.5)] {\scriptstyle{\tn{pb}}}} \]
The functors which are discrete fibrations, discrete opfibrations and have finite fibres are closed under composition and stable by pullback along arbitrary functors. Thus $p_3$ is such a functor since $p_1$ and $p_2$ are. By Lemma \ref{lem:ess-disc-inheritance} one deduces successively that $P$, $B_3$, $F$ and $E_3$ are all equivalent to discrete categories.
\end{proof}
We consider now a 2-cell $(\alpha_1,\alpha_2)$ which has a 1-section, and a $1$-cell $(q_1,q_2)$
\begin{equation}\label{diag:stable-local-coidentifier}
\begin{gathered}
\xygraph{!{0;(1.5,0):(0,1)::} {I}="p0" [ur] {E_3}="p1" [r] {B_3}="p2" [dr] {I}="p3" [l] {B_2}="p4" [l] {E_2}="p5" [d(.75)] {E_1}="p6" [r] {B_1}="p7" "p0":@{<-}"p1"^-{}:"p2"^-{p_3}:"p3"^-{}:@{<-}"p4"^-{}:@{<-}"p5"^-{p_2}:"p0"^-{}
"p1":@/_{1pc}/"p5"_(.5){}|(.5){}="dal" "p2":@/_{1pc}/"p4"_(.5){}|(.5){}="dar"
"p1":@/^{1pc}/"p5"^(.5){}|(.5){}="cal" "p2":@/^{1pc}/"p4"^(.5){}|(.5){}="car"
"dal":@{}"cal"|(.3){}="dl"|(.7){}="cl" "dl":@{=>}"cl"^-{\alpha_2}
"dar":@{}"car"|(.3){}="dr"|(.7){}="cr" "dr":@{=>}"cr"^-{\alpha_1}
"p0":@{<-}"p6"_-{}:"p7"_-{p_1}:"p3"_-{}
"p5":"p6"^-{q_2} "p4":"p7"^-{q_1}}
\end{gathered}
\end{equation}
in $\mathfrak{E}_I$ such that $(q_1,q_2)(\alpha_1,\alpha_2) = \id$.
\begin{lem}\label{lem:general-coidentifier-stability}
If in (\ref{diag:stable-local-coidentifier}) $q_1$ is the coidentifier in $\Cat$ of $\alpha_1$, then $(q_1,q_2)$ is a coidentifier of $(\alpha_1,\alpha_2)$ in $\Polyc{\Cat}(I,I)$, which is preserved by composition on either side by polynomials from $\mathfrak{E}_I$.
\end{lem}
\begin{proof}
Since $\Sigma_{B_1} : \Cat/B_1 \to \Cat$ creates colimits, one can regard $q$ and $\alpha_1$ as a coidentifier in $\Cat/B_1$. Then $r$ and $\alpha_2$ can be regarded as obtained from these by applying $\Delta_{p}$, and thus since $p$ is exponentiable, $r$ is the coidentifier of $\alpha_2$ in $\Cat/E_1$, and thus also in $\Cat$ since $\Sigma_{E_1}$ creates colimits.

To show that (\ref{diag:stable-local-coidentifier}) is a coidentifier in $\Polyc{\Cat}(I,I)$ it suffices to show that given $r_1$, $r_2$ and $f$ as in
\[ \xygraph{!{0;(1.5,0):(0,.6667)::}
% top row
{E_3}="p0" [r] {E_2}="p1" [r] {E_1}="p2" [r] {X}="p3" "p0":@<1.5ex>"p1"|(.4){}="pt"^-{}:"p2"_-{q_2}:@{.>}"p3"_-{s_2} "p0":@<-1.5ex>"p1"|(.4){}="pb"_-{} "p1":@/^{1pc}/"p3"^-{r_2}
"pt":@{}"pb"|(.25){}="pd2"|(.75){}="pc2" "pd2":@{=>}"pc2"^{\alpha_2}
% bottom row
"p0" [d] {B_3}="q0" [r] {B_2}="q1" [r] {B_1}="q2" [r] {Y}="q3" "q0":@<1.5ex>"q1"|(.4){}="qt"^-{}:"q2"^-{q_1}:@{.>}"q3"^-{s_1} "q0":@<-1.5ex>"q1"|(.4){}="qb"_-{} "q1":@/_{1pc}/"q3"_-{r_1}
"qt":@{}"qb"|(.25){}="qd2"|(.75){}="qc2" "qd2":@{=>}"qc2"^{\alpha_1}
% vertical arrows
"p0":"q0"_-{p_3} "p1":"q1"^-{p_2} "p2":"q2"^-{p_1} "p3":"q3"^-{f}} \]
making the square $(r_2,f,r_1,p_2)$ a pullback, then the induced square $(s_2,f,s_1,p_1)$ is also a pullback. To check that the induced square is a pullback on objects, consider $b \in B_1$ and $x \in X$ such that $s_1b = fx$. Since $q_1$ is surjective on objects, one has $b_2$ such that $q_1b_2 = b$, and so there is $e_2 \in E_2$ unique such that $p_2e_2 = b_2$ and $r_2e_2 = x$. Thus $e = q_2e_2$ satisfies $p_1e = b$ and $s_2e = x$. To see that $e$ is unique with these properties, consider $e' \in E_1$ such that $p_1e' = b$ and $s_2e' = x$. By the pullback $(q_2,p_1,q_1,p_2)$ there is $e'_2 \in E_2$ unique such that $p_2e'_2 = b_2$ and $q_2e'_2 = e'$. But then $r_2e_2 = r_2e'_2$ and $p_2e_2 = p_2e'_2$, and so since $(p_2,r_2)$ are jointly monic $e_2 = e'_2$, and so $e = e'$. To check that the induced square is a pullback on arrows, given arrows $\beta$ in $B_1$ and $\alpha$ in $X$, one establishes the existence of an arrow $\gamma$ in $E_1$ such that $p_1\gamma = \beta$ and $s_2\gamma = \alpha$ as in the objects case, and the uniqueness follows since $E_1$ is equivalent to a discrete category.

It remains to show that taking a situation (\ref{diag:stable-local-coidentifier}) in which $q_1$ is a coidentifier of $\alpha_1$ in $\Cat$, and either pre or post-composing it with a polynomial from $\mathfrak{E}_I$, gives another such situation. Let $P \in \mathfrak{E}_I$ and let the diagram on the left
\[ \xygraph{{\xybox{\xygraph{!{0;(1.25,0):(0,1)::} {I}="p0" [ur] {E'_3}="p1" [r] {B'_3}="p2" [dr] {I}="p3" [l] {B'_2}="p4" [l] {E'_2}="p5" [d(.75)] {E'_1}="p6" [r] {B'_1}="p7" "p0":@{<-}"p1"^-{}:"p2"^-{p'_3}:"p3"^-{}:@{<-}"p4"^-{}:@{<-}"p5"^-{p'_2}:"p0"^-{}
"p1":@/_{1pc}/"p5"_(.5){}|(.5){}="dal" "p2":@/_{1pc}/"p4"_(.5){}|(.5){}="dar"
"p1":@/^{1pc}/"p5"^(.5){}|(.5){}="cal" "p2":@/^{1pc}/"p4"^(.5){}|(.5){}="car"
"dal":@{}"cal"|(.3){}="dl"|(.7){}="cl" "dl":@{=>}"cl"^-{\alpha'_2}
"dar":@{}"car"|(.3){}="dr"|(.7){}="cr" "dr":@{=>}"cr"^-{\alpha'_1}
"p0":@{<-}"p6"_-{}:"p7"_-{p'_1}:"p3"_-{}
"p5":"p6"^-{q'_2} "p4":"p7"^-{q'_1}}}}
[r(5)]
{\xybox{\xygraph{!{0;(1.25,0):(0,1)::} {I}="p0" [ur] {E''_3}="p1" [r] {B''_3}="p2" [dr] {I}="p3" [l] {B''_2}="p4" [l] {E''_2}="p5" [d(.75)] {E''_1}="p6" [r] {B''_1}="p7" "p0":@{<-}"p1"^-{}:"p2"^-{p''_3}:"p3"^-{}:@{<-}"p4"^-{}:@{<-}"p5"^-{p''_2}:"p0"^-{}
"p1":@/_{1pc}/"p5"_(.5){}|(.5){}="dal" "p2":@/_{1pc}/"p4"_(.5){}|(.5){}="dar"
"p1":@/^{1pc}/"p5"^(.5){}|(.5){}="cal" "p2":@/^{1pc}/"p4"^(.5){}|(.5){}="car"
"dal":@{}"cal"|(.3){}="dl"|(.7){}="cl" "dl":@{=>}"cl"^-{\alpha''_2}
"dar":@{}"car"|(.3){}="dr"|(.7){}="cr" "dr":@{=>}"cr"^-{\alpha''_1}
"p0":@{<-}"p6"_-{}:"p7"_-{p''_1}:"p3"_-{}
"p5":"p6"^-{q''_2} "p4":"p7"^-{q''_1}}}}} \]
denote the result of post composing (\ref{diag:stable-local-coidentifier}) by $P$, and let the diagram on the right denote the result of pre composing (\ref{diag:stable-local-coidentifier}) by $P$. Our task is to show that in $\Cat$, $q'_1$ is a coidentifier of $\alpha'_1$ and $q''_1$ is a coidentifier of $\alpha''_1$. In the case of $q'_1$ note that $\alpha'_1 = T\alpha_1$ and $q'_1 = Tq_1$ where $T = \PFun{\Cat}_{I,I}(P)$. By \cite{Weber-PolynomialFunctors} Theorem 4.5.1, $T$ preserves sifted colimits and so $q'_1$ is indeed the coidentifier of $\alpha'_1$. In the case of $q''_1$ if one denotes by
\begin{equation}\label{eq:coid-endos-for-proof-of-main-lemma}
\begin{gathered}
\xygraph{!{0;(1.5,0):(0,1)::} {Q}="p0" [r] {R}="p1" [r] {S}="p2" "p0":@<1.5ex>"p1"|(.45){}="t"^-{d}:"p2"^-{k} "p0":@<-1.5ex>"p1"|(.45){}="b"_-{c} "t":@{}"b"|(.15){}="d"|(.85){}="c" "d":@{=>}"c"^-{\zeta}}
\end{gathered}
\end{equation}
the diagram in $\tn{End}(\Cat/I)$ which is the result of applying $\PFun{\Cat}_{I,I}$ to (\ref{diag:stable-local-coidentifier}), and by $X \to I$ the third map comprising the polynomial $P$, then $\alpha''_1$ and $q''_1$ are the result of evaluating (\ref{eq:coid-endos-for-proof-of-main-lemma}) at $X \in \Cat/I$. Note in particular, that when $X = 1$ this will be $\alpha_1$ and $q_1$, with $q_1$ a coidentifier of $\alpha_1$ by hypothesis. In the general case note that one has
\[ \xygraph{!{0;(1.75,0):(0,.6667)::}
% top row
{QX}="p0" [r] {RX}="p1" [r] {SX}="p2"
"p0":@<1.5ex>"p1"|(.4){}="pt"^-{d_X}:"p2"^-{k_X} "p0":@<-1.5ex>"p1"|(.4){}="pb"_-{c_X}
"pt":@{}"pb"|(.25){}="pd2"|(.75){}="pc2" "pd2":@{=>}"pc2"^{\zeta_X}
% bottom row
"p0" [d] {Q1}="q0" [r] {R1}="q1" [r] {S1.}="q2" "q0":@<1.5ex>"q1"|(.4){}="qt"^-{d_1}:"q2"_-{k_1} "q0":@<-1.5ex>"q1"|(.4){}="qb"_-{c_1} 
"qt":@{}"qb"|(.25){}="qd2"|(.75){}="qc2" "qd2":@{=>}"qc2"^{\zeta_1}
% vertical arrows
"p0":"q0"_-{Q(!)} "p1":"q1"^-{R(!)} "p2":"q2"^-{S(!)}} \]
One can regard $k_X$ and $\zeta_X$ as living in $\Cat/SX$, and since $\Sigma_{SX} : \Cat/SX \to \Cat$ creates colimits, $k_X$ coidentifies $\zeta_X$ in $\Cat/SX$ iff it does so in $\Cat$. Similarly $k_1$ may be regarded as a coidentifier of $\zeta_1$ in $\Cat/S1$. Since the 2-natural transformations $d$, $c$ and $k$ are cartesian, one may regard the top row as being obtained from the bottom by pulling back along $S(!)$. Since $S$ is familial and opfamilial by \cite{Weber-PolynomialFunctors} Theorem 4.4.5, it preserves fibrations and opfibrations, whence $S(!)$ is a fibration and opfibration, and so is exponentiable. Thus $\Delta_{S(!)}$ preserves colimits, and so $\zeta_X = \alpha''_1$ is indeed coidentified by $k_X = q''_1$.
\end{proof}
\begin{proof}
(\emph{of Lemma \ref{lem:monoidally-stable-local-coidentifier}}).
The situation (\ref{diag:local-coidentifier-in-Poly}) clearly conforms to that of (\ref{diag:stable-local-coidentifier}), and by definition $q_{B_T}$ is the coidentifier of $\alpha_{B_T}$ in $\Cat$. To show that (\ref{diag:local-coidentifier-in-Poly}) is a monoidally stable coidentifier, one must show that it is a coidentifier in $\Polyc{\Cat}(I,I)$ which is preserved by pre or post composition with composites of
\[ \xygraph{*{\xybox{\xygraph{!{0;(1.25,0):(0,.5)::} {I}="p0" [r] {E_T}="p1" [r] {B_T}="p2" [r] {I}="p3" "p0":@{<-}"p1"^-{s_T}:"p2"^-{p_T}:"p3"^-{t_T}}}}
[r(6)]
*!(0,-.05){\xybox{\xygraph{!{0;(1.25,0):(0,.5)::} {I}="p0" [r] {E^{[1]}_T}="p1" [r] {B^{[1]}_T}="p2" [r] {I.}="p3" "p0":@{<-}"p1"^-{s^{[1]}_T}:"p2"^-{p^{[1]}_T}:"p3"^-{t^{[1]}_T}}}}} \]
Since these are in $\mathfrak{E}_I$ the result follows by Lemmas \ref{lem:composition-DI} and \ref{lem:general-coidentifier-stability}.
\end{proof}
%

%% \bibliographystyle{plain}
%% \bibliography{../../../typesetting/Templates/Master}

\end{document}